\documentclass[defaultenums, eulercal, tikz, nogeometry]{nmd/nmd-article}
\RequirePackage[papersize={\nmdwidth+2\nmdhmargins,
  1.45\nmdwidth+2.66\nmdhmargins}, hmarginratio=1:1,
  top=0.85\nmdhmargins, textwidth=\nmdwidth,
  textheight=1.55\nmdwidth, headsep=10bp, headheight=15bp]{geometry}
\usepackage[shortlabels]{enumitem}
\usepackage[subfigure]{tocloft}
\usepackage{nccmath}

\ifnmd
\fi

\title{Roots of Alexander polynomials of random positive 3-braids}

\author{Nathan M. Dunfield}
\givenname{Nathan}
\surname{Dunfield}
\address{ Dept.~of Math., MC-382 \\
          University of Illinois \\
          1409 W. Green St. \\
          Urbana, IL 61801 \\
          USA
}
\email{nathan@dunfield.info}
\urladdr{http://dunfield.info}

\author{Giulio Tiozzo}
\address{Department of Mathematics\\
         University of Toronto\\
         40 St. George Street \\
         Toronto, ON M5S 2E4, Canada}
\email{tiozzo@math.utoronto.ca}
\urladdr{http://www.math.toronto.edu/tiozzo}

\arxivreference{}
\arxivpassword{}

\let\Re\undefined
\DeclareMathOperator{\Re}{Re}

\DeclareMathOperator{\Mod}{Mod}

\newcommand{\Radem}{{\mathfrak{R}}}
\newcommand{\HRadem}{{\widetilde{\mathfrak{R}}}}

\newcommand{\matr}[4]{%
  \left( \begin{array}{cc} #1 & #2 \\ #3 & #4 \end{array} \right)}

\newcommand{\smallvec}[2]{\left(\begin{smallmatrix}
        \scriptstyle #1  \\ \scriptstyle #2
        \end{smallmatrix}\right)}

\newcommand{\Braid}[1]{\mathrm{Br}_{#1}}

\newcommand{\D}{\mathbb{D}}
\renewcommand{\P}{\mathbb{P}}
\newcommand{\norminf}[1]{{\norm{#1}_\infty}}
\DeclareMathOperator{\Sym}{Sym}
\newcommand{\dt}{\mathit{dt}}

\DeclareMathOperator{\conj}{conj}

\newcommand{\bif}{\mathit{bif}}
\renewcommand{\Dbar}{\kernoverline{0}{\mathbb{D}}{2.5}}
\newcommand{\cAbar}{\kernoverline{4}{\cA}{2.5}}
\newcommand{\odd}{\mathit{odd}}
\newcommand{\even}{\mathit{even}}

\definecolor{giuliocomment}{rgb}{0.067, 0.412, 0.067}

\begin{document}

\begin{abstract}
Motivated by an observation of Dehornoy, we study the roots of
Alexander polynomials of knots and links that are closures of positive
$3$-strand braids.  We give experimental data on random such braids
and find that the roots exhibit marked patterns, which we refine into
precise conjectures.  We then prove several results along those lines,
for example that generically at least 69\% of the roots are on the
unit circle, which appears to be sharp. We also show there is a large
root-free region near the origin.  We further study the
equidistribution properties of such roots by introducing a Lyapunov
exponent of the Burau representation of random positive braids, and a
corresponding bifurcation measure.  In the spirit of Deroin and
Dujardin, we conjecture that the bifurcation measure gives the
limiting measure for such roots, and prove this on a region with
positive limiting mass.  We use tools including work of Gambaudo and
Ghys on the signature function of links, for which we prove a central
limit theorem.
\end{abstract}

\maketitle

\vspace{-0.5cm}

\setcounter{tocdepth}{1}
\setlength{\cftbeforesecskip}{0pt}
\tableofcontents

\section{Introduction}
\label{sec: intro}

\begin{figure}
  \centering
  \begin{tikzoverlay}[height=0.45\textwidth]{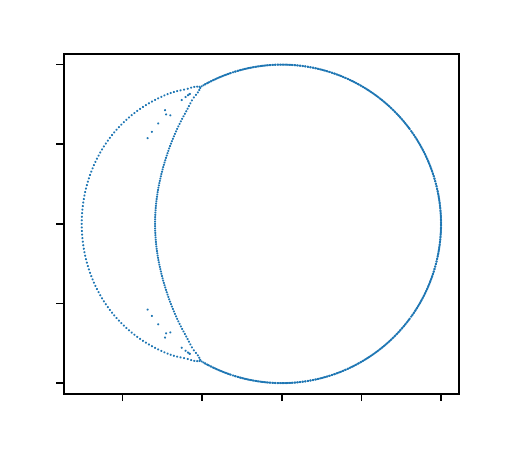}
  \begin{scope}[font=\scriptsize]
    \draw (24.032419, 8.169935) node[below] {$-1.0$};
    \draw (39.643748, 8.169935) node[below] {$-0.5$};
    \draw (55.255078, 8.169935) node[below] {$0.0$};
    \draw (55.255078, 2) node[below, font=\footnotesize] {(a)};
    \draw (70.866408, 8.169935) node[below] {$0.5$};
    \draw (86.477737, 8.169935) node[below] {$1.0$};
    \draw (9.640523, 13.115576) node[left] {$-1.0$};
    \draw (9.640523, 28.726906) node[left] {$-0.5$};
    \draw (9.640523, 44.338235) node[left] {$0.0$};
    \draw (9.640523, 59.949565) node[left] {$0.5$};
    \draw (9.640523, 75.560895) node[left] {$1.0$};
    \begin{scope}[shift={(55.25507798, 44.33823529)},
      xscale=31.22265937, yscale=31.22265937]
    \end{scope}
  \end{scope}
\end{tikzoverlay}
  \begin{tikzoverlay}[height=0.45\textwidth]{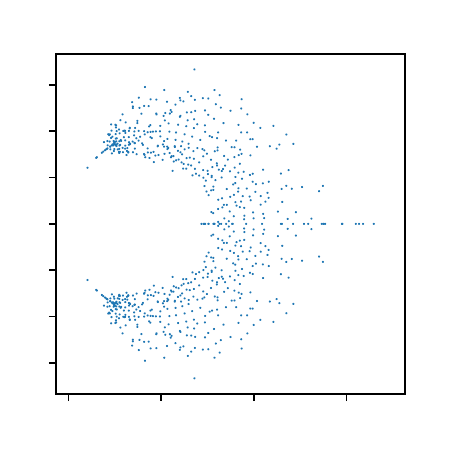}
  \begin{scope}[font=\scriptsize]
    \draw (15.231815, 9.259259) node[below] {$-1$};
    \draw (35.810801, 9.259259) node[below] {$0$};
    \draw (50, 2) node[below, font=\footnotesize] {(b)};
    \draw (56.389788, 9.259259) node[below] {$1$};
    \draw (76.968774, 9.259259) node[below] {$2$};

    \draw (9.259259, 19.381521) node[left] {$-1.5$};
    \draw (9.259259, 29.671014) node[left] {$-1.0$};
    \draw (9.259259, 39.960507) node[left] {$-0.5$};
    \draw (9.259259, 50.250000) node[left] {$0.0$};
    \draw (9.259259, 60.539493) node[left] {$0.5$};
    \draw (9.259259, 70.828986) node[left] {$1.0$};
    \draw (9.259259, 81.118479) node[left] {$1.5$};
    
    \begin{scope}[shift={(35.81080137, 50.25000000)},
      xscale=20.57898614, yscale=20.57898614]
    \end{scope}
  \end{scope}
\end{tikzoverlay}
  
  \caption{Plots of the roots of the Alexander polynomials of two
    knots which are closures of 3-strand braids.  In both cases, the
    polynomials have degree 764, but the one at left comes from a
    positive braid (of length 762) whereas the one at right is from an
    arbitrary braid (of length 1,598).  The two braids were chosen
    randomly using the uniform measure on the monoid generators
    $\{\sigma_1, \sigma_2\}$ and
    $\{\sigma_1, \sigma_1^{-1}, \sigma_2, \sigma_2^{-1}\}$
    respectively. For the
    positive braid at left, 69.3\% of the roots are on the unit circle,
    including all 504 roots where $\Re(z) > -0.5$, only four of which
    are roots of unity; this is compatible with Conjecture~\ref{conj:
      circle}, which predicts asymptotically a value of 69.1\%. In
    contrast, the arbitrary braid at right has 10.2\% of its roots on the unit
    circle.}
  \label{fig: dehornoy}
\end{figure}

\begin{figure}
  \centering
  \begin{tikzoverlay}[width=0.8\textwidth]{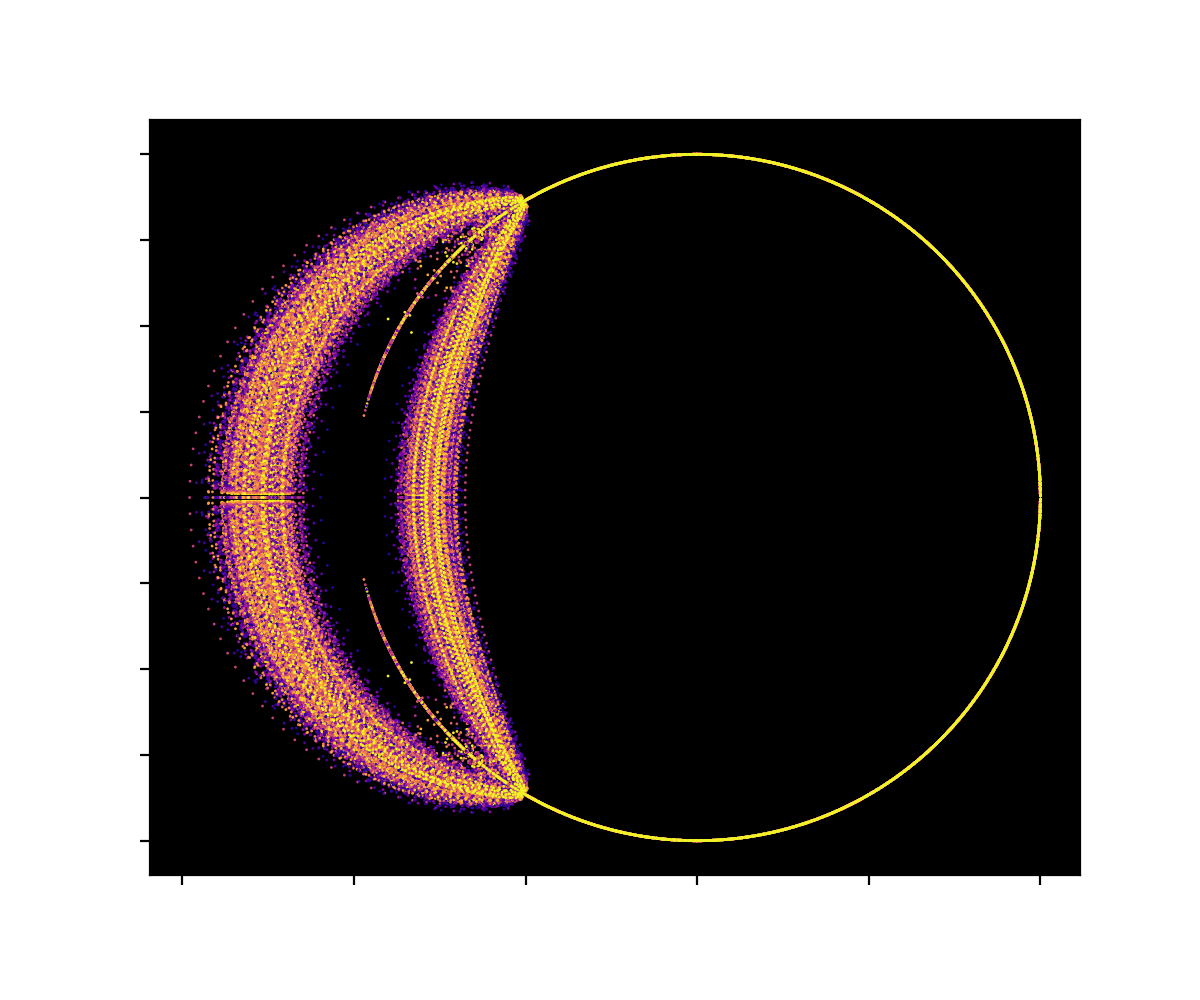}
  \begin{scope}[font=\scriptsize]
    \draw (15.201246, 8.796296) node[below] {$-1.5$};
    \draw (29.500499, 8.796296) node[below] {$-1.0$};
    \draw (43.799753, 8.796296) node[below] {$-0.5$};
    \draw (58.099006, 8.796296) node[below] {$0.0$};
    \draw (72.398259, 8.796296) node[below] {$0.5$};
    \draw (86.697512, 8.796296) node[below] {$1.0$};
    \draw (10.879630, 13.276493) node[left] {$-1.0$};
    \draw (10.879630, 27.575747) node[left] {$-0.5$};
    \draw (10.879630, 41.875000) node[left] {$0.0$};
    \draw (10.879630, 56.174253) node[left] {$0.5$};
    \draw (10.879630, 70.473507) node[left] {$1.0$};
  \end{scope}

  \begin{scope}[shift={(58.09900583, 41.87500000)},
                xscale=28.59850654, yscale=28.59850654]
  \end{scope}
\end{tikzoverlay}

  \caption{The roots of $\Delta_K(t)$ for 2,500 knots $K$
    coming from random positive 3-braids, chosen so the word lengths
    have mean 500 and standard deviation 170; some 1.2 million roots
    are plotted.  The color of each point indicates the degree of
    $\Delta_K(t)$ where higher degrees are lighter colors.}
  \label{fig: pos many}
\end{figure}

\begin{figure}
  \centering
  \begin{tikzoverlay}[width=0.7\textwidth]{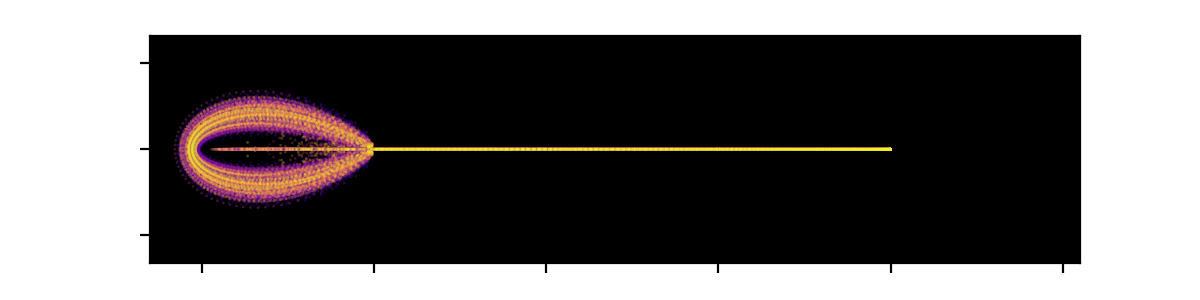}
  \begin{scope}[font=\scriptsize]
    \draw (16.805556, 1.504630) node[below] {$-2$};
    \draw (31.157407, 1.504630) node[below] {$-1$};
    \draw (45.509259, 1.504630) node[below] {$0$};
    \draw (59.861111, 1.504630) node[below] {$1$};
    \draw (74.212963, 1.504630) node[below] {$2$};
    \draw (88.564815, 1.504630) node[below] {$3$};
    \draw (10.879630, 5.386574) node[left] {$-0.5$};
    \draw (10.879630, 12.562500) node[left] {$0.0$};
    \draw (10.879630, 19.738426) node[left] {$0.5$};
    \begin{scope}[shift={(45.50925926, 12.56250000)},
      xscale=14.35185185, yscale=14.35185185]
    \end{scope}
  \end{scope}
\end{tikzoverlay}

  \begin{tikzoverlay}[width=0.7\textwidth]{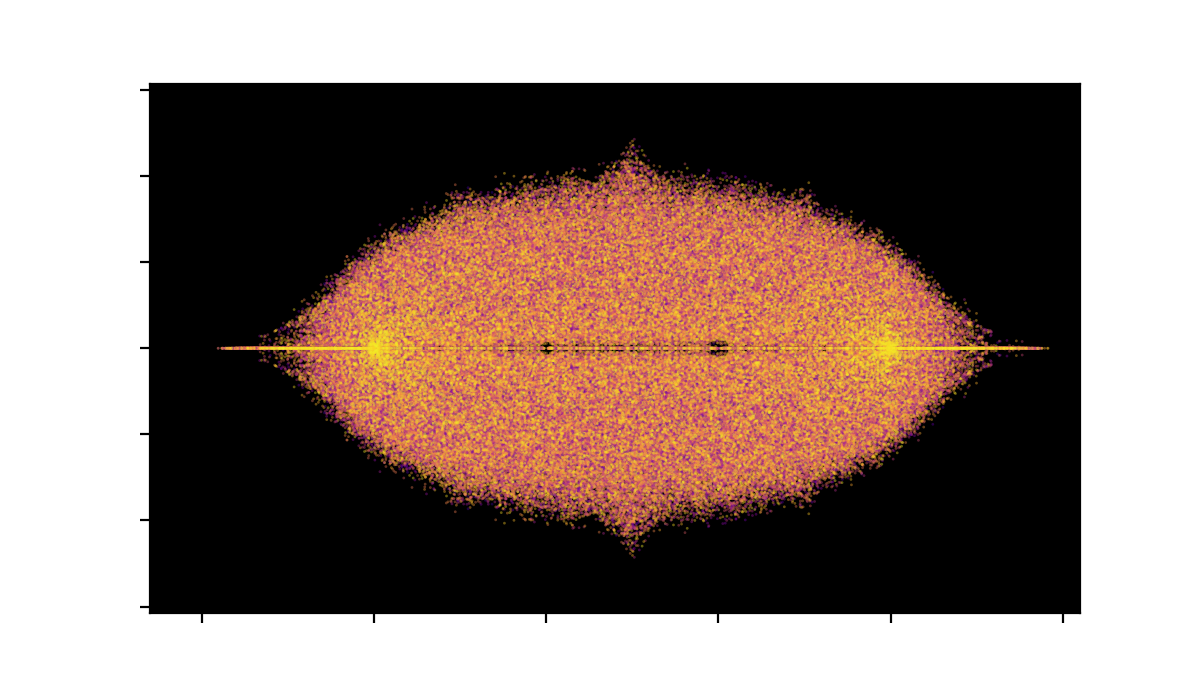}
  \begin{scope}[font=\scriptsize]
    \draw (16.805556, 5.671296) node[below] {$-2$};
    \draw (31.157407, 5.671296) node[below] {$-1$};
    \draw (45.509259, 5.671296) node[below] {$0$};
    \draw (59.861111, 5.671296) node[below] {$1$};
    \draw (74.212963, 5.671296) node[below] {$2$};
    \draw (88.564815, 5.671296) node[below] {$3$};
    \draw (10.879630, 7.784722) node[left] {$-1.5$};
    \draw (10.879630, 14.960648) node[left] {$-1.0$};
    \draw (10.879630, 22.136574) node[left] {$-0.5$};
    \draw (10.879630, 29.312500) node[left] {$0.0$};
    \draw (10.879630, 36.488426) node[left] {$0.5$};
    \draw (10.879630, 43.664352) node[left] {$1.0$};
    \draw (10.879630, 50.840278) node[left] {$1.5$};
    \begin{scope}[shift={(45.50925926, 29.31250000)},
      xscale=14.35185185, yscale=14.35185185]
    \end{scope}
  \end{scope}
\end{tikzoverlay}

  \caption{Another picture contrasting positive (top) versus generic
    (bottom) 3-braids, this time in the ``trace coordinates'' of
    Section~\ref{sec: alex props}.  Both plots are based on 2,500
    random braid words chosen so that $\deg \Delta_K(t)$ has mean
    about 500 and standard deviation 170; the words in
    $\sigma_1, \sigma_2$ themselves have mean lengths of about 500 and
    1,000 respectively.  In the positive case, the words used are the
    same as in Figure~\ref{fig: pos many}.  }
  \label{fig: contrast}
\end{figure}

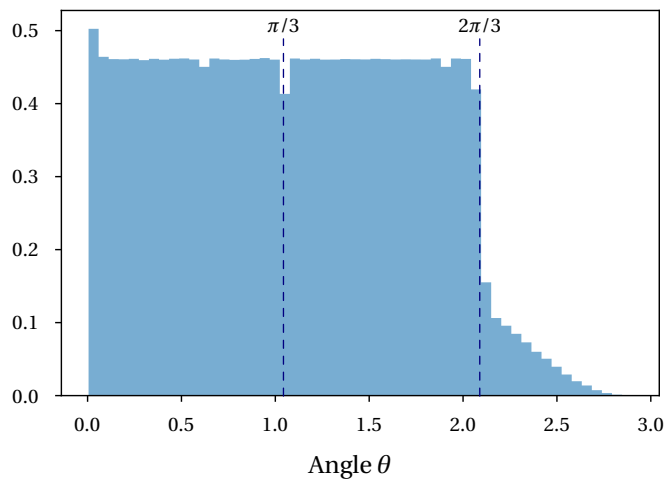
\begin{figure}
  \centering
  \begin{tikzoverlay}[width=0.7\textwidth]{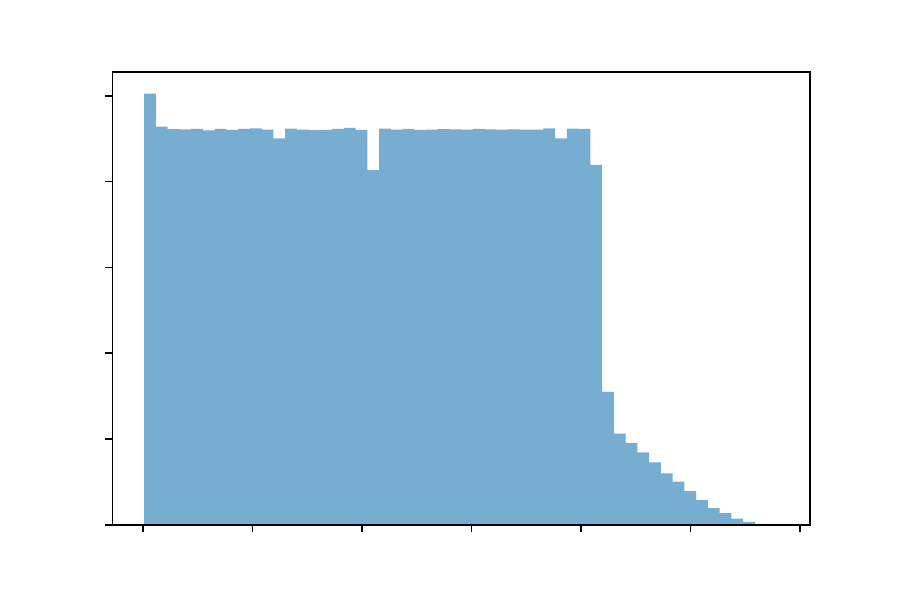}
  \begin{scope}[font=\scriptsize]
    \draw (50, 2) node[below, font=\footnotesize] {Angle $\theta$};
    
    \draw (15.896807, 6.712963) node[below] {$0.0$};
    \draw (28.064075, 6.712963) node[below] {$0.5$};
    \draw (40.231342, 6.712963) node[below] {$1.0$};
    \draw (52.398609, 6.712963) node[below] {$1.5$};
    \draw (64.565877, 6.712963) node[below] {$2.0$};
    \draw (76.733144, 6.712963) node[below] {$2.5$};
    \draw (88.900411, 6.712963) node[below] {$3.0$};
    \draw (10.879630, 8.333333) node[left] {$0.0$};
    \draw (10.879630, 17.872870) node[left] {$0.1$};
    \draw (10.879630, 27.412406) node[left] {$0.2$};
    \draw (10.879630, 36.951943) node[left] {$0.3$};
    \draw (10.879630, 46.491479) node[left] {$0.4$};
    \draw (10.879630, 56.031016) node[left] {$0.5$};
    
    \begin{scope}[shift={(15.89680744, 8.33333333)},
      xscale=24.33453461, yscale=95.39536435, dashed,
      color=blue!50!black, line cap=round, line width=0.5pt]
      \coordinate (pi on three) at (1.0471975511965976, 0);
      
      \draw (pi on three) -- +(0, 0.489) node[above=-0.5, color=black] {$\pi/3$};
      \draw ($2*(pi on three)$) -- +(0, 0.489) node[above=-0.5, color=black] {$2\pi/3$};
    \end{scope}
  \end{scope}
\end{tikzoverlay}

  \caption{A histogram of the roots on the upper half
    of the unit circle from Figure~\ref{fig: pos many}, in terms of
    the usual polar angle $\theta$.
    }
  \label{fig: on circle}
\end{figure}

\begin{figure}
  \centering
  \hspace{0.7cm} \begin{tikzpicture}[nmdstd]
  \node[above right] at (0, 0) {%
    \begin{tikzoverlay}[height=5.89cm]{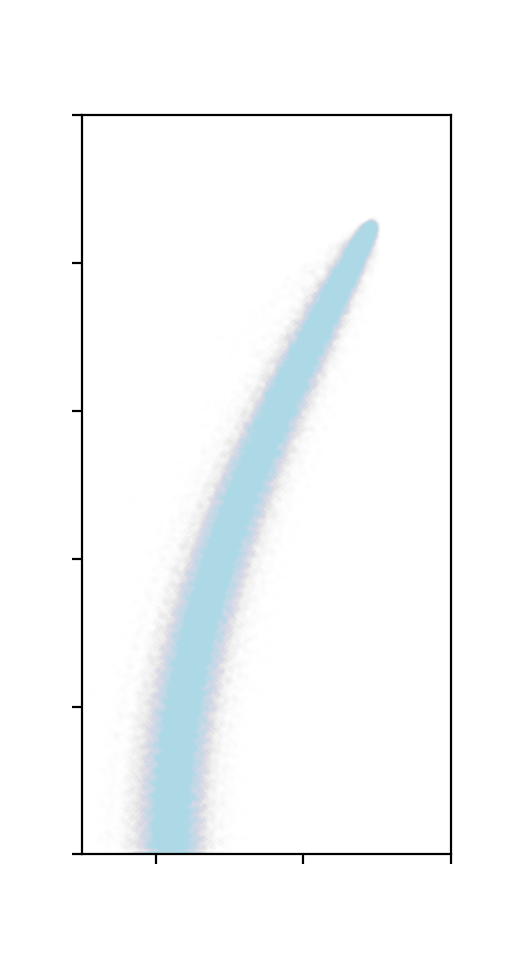}
      \begin{scope}[font=\scriptsize]
        \draw (29.926923, 15.8) node[below] {$-0.8$};
        \draw (58.357692, 15.8) node[below] {$-0.6$};
        \draw (86.788462, 15.8) node[below] {$-0.4$};
        \draw (11.972222, 20.307692) node[left] {$0.0$};
        \draw (11.972222, 48.738462) node[left] {$0.2$};
        \draw (11.972222, 77.169231) node[left] {$0.4$};
        \draw (11.972222, 105.600000) node[left] {$0.6$};
        \draw (11.972222, 134.030769) node[left] {$0.8$};
        \draw (11.972222, 162.461538) node[left] {$1.0$};
        \begin{scope}[shift={(143.65000000, 20.30769231)},
                xscale=142.15384615, yscale=142.15384615]
           \node[above left] at (120:0.98) {$\zeta_3$};
           \draw[black] (-0.782, 0) arc (180:144.25:1.470787);
           \draw[nmdmedium, dashed] ([shift=(154.15807:1)] 0, 0) arc
                (154.158067:113.57817:1);
         \end{scope}
      \end{scope}
    \end{tikzoverlay}};
  \node[above right] at (3.5, -0.1) {%
    \begin{tikzoverlay}[height=6cm]{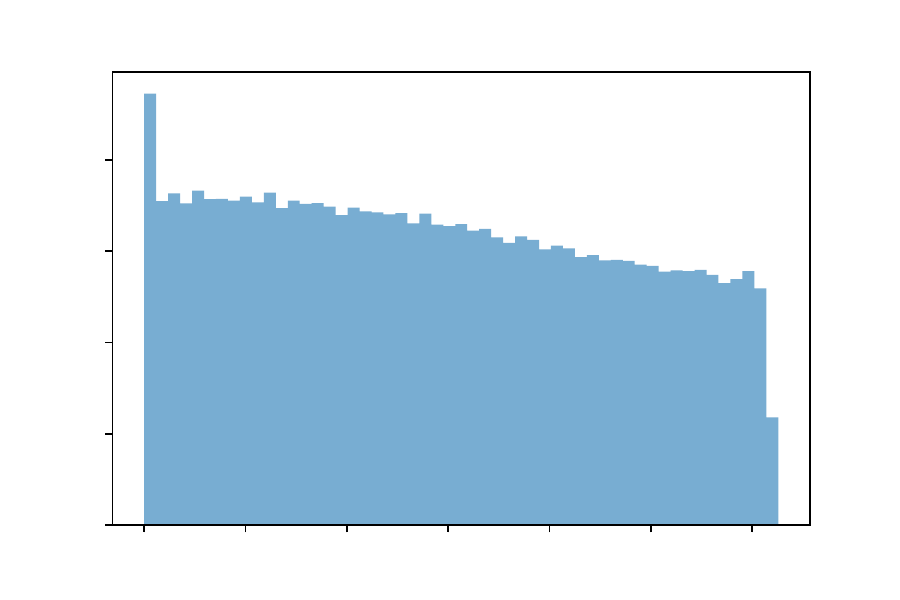}
      \begin{scope}[font=\scriptsize]
        \draw (16.022727, 6.712963) node[below] {$0.0$};
        \draw (27.279850, 6.712963) node[below] {$0.1$};
        \draw (38.536972, 6.712963) node[below] {$0.2$};
        \draw (49.794094, 6.712963) node[below] {$0.3$};
        \draw (61.051217, 6.712963) node[below] {$0.4$};
        \draw (72.308339, 6.712963) node[below] {$0.5$};
        \draw (83.565462, 6.712963) node[below] {$0.6$};
        \draw (10.879630, 8.333333) node[left] {$0.0$};
        \draw (10.879630, 18.478539) node[left] {$0.5$};
        \draw (10.879630, 28.623745) node[left] {$1.0$};
        \draw (10.879630, 38.768951) node[left] {$1.5$};
        \draw (10.879630, 48.914157) node[left] {$2.0$};
        \begin{scope}[shift={(16.02272727, 8.33333333)},
          xscale=112.57122404, yscale=20.29041198]
        \end{scope}
      \end{scope}
    \end{tikzoverlay}};
\end{tikzpicture}

  \caption{To study what is happening in Figure~\ref{fig: pos many} in
    the \emph{interior} of the disc, consider those roots shown in the
    close-up at left.  We radially project these onto the circular arc
    shown joining $-0.782$ to $\zeta_3$ (the choice of
    this arc is somewhat arbitrary). The resulting histogram is
    at right, where the horizontal axis is the angle in radians along
    said arc, measured clockwise from the real axis. }
    \label{fig: on arc}
\end{figure}

Dehornoy observed in \cite[\S 3.2]{Dehornoy2015} that the Alexander
polynomial $\Delta_K(t)$ of the closure of a \emph{positive} braid has
roots that ``seem to accumulate on very particular curves'' in the
complex plane; see for example Figure~\ref{fig: dehornoy}(a).  Here,
we explore this phenomena in the simplest case of positive 3-strand
braids, and develop detailed conjectures about its cause and
underlying structure.  We are able to prove some of these conjectures,
but are unable to resolve them all.  Throughout this introduction, see
Section~\ref{sec: background} for precise definitions.

We found the following phenomena for random positive 3-strand
braids based on a detailed analysis of the data plotted in
Figure~\ref{fig: pos many}. Because of the symmetry of the roots
coming from Lemma~\ref{lem: sym roots}, we focus on the closed unit
disk $\Dbar$ in $\C$ where $\abs{z} \leq 1$.
 
\begin{enumerate}[(P1)]
\item
  \label{item: on curves}
  The roots of the Alexander polynomial $\Delta_K(t)$ \emph{nearly}
  all lie either on the unit circle or on a single continuous curve
  joining the roots of unity $\zeta_3 = e^{2 \pi i/3}$ and
  $\zetabar_3$; see Figure~\ref{fig: dehornoy}(a) for a typical
  example.  Said curve meets the real axis in a single point in
  $(-1, -1/2)$.

\item
  \label{item: two-thirds circle}

  Asymptotically, $2/3$ of the roots lie on the arc of the unit
  circle where $-1/2 < \Re(z)$ and those roots are evenly
  distributed there. Compare Figure~\ref{fig: on circle}.

\item
  \label{item: other arc}

  The roots on the curve from $\zeta_3$ to $\zetabar_3$ are close
  to equidistributed with respect to arc length, see Figure~\ref{fig:
    on arc}.

\item
  \label{item: big zero-free region}
  There are no roots in the region with $\abs{z} < 1$ and $-1/2 < \Re(z)$.
  In fact, there are no roots with $\abs{z} < 1$ to the right of the
  curve from $\zeta_3$ to $\zetabar_3$.
  
\end{enumerate}

Motivated by the notion of bifurcation current from
\cite{DeroinDujardin2012}, we combine these phenomena into an
overarching conjecture as follows. Let $\Braid{3}$ be the \3-strand
braid group.  Any $w \in \Braid{3}$ has a braid closure $\what$ which
is a link (a disjoint union of knots) of at most three components.
Consider the product space $\Omega := \{\sigma_1, \sigma_2\}^\N$ with
the measure $\mu^\N$ where $\mu(\sigma_1) = \mu(\sigma_2) = 1/2$.
Choosing $(g_i)_{i \in \N} \in \Omega$ with respect to $\mu^{\N}$ gives a
random walk $w_n := g_1 g_2 \cdots g_n$ and an associated sequence of
links $\what_n$.

Ma showed in \cite{Ma2014}, building on
\cite{Maher2010}, that the link exteriors $S^3 \setminus \what_n$ are
hyperbolic with probability tending to 1 as $n \to \infty$.  Because
of parity issues, $\what_n$ always has two components when $n$ is odd,
and either one or three components when $n$ is even; in the latter
case, one gets a knot with probability $2/3$ asymptotically
\cite{Ma2013}.

To state our conjecture about the roots of $\Delta_{\what_n}$, for any
$w \in \Braid{3}$ we define the set
$Z_w := \setdef{z \in \C}{\Delta_{\what}(z) = 0}$, whose elements are
considered with multiplicity.  Consider the uniform probability
measure supported on $Z_w$, that is
\[
  \nu_w: = \frac{1}{\# Z_w} \sum_{z \in Z_w}{\delta_z}.
\]
We propose that the measures $\nu_{w_n}$ almost surely have a limit,
which is moreover independent of the sample path:
\begin{conjecture}
  \label{conj: main}
  There is a compactly supported measure $\nu_\infty$ on
  $\C$ such that for almost every sequence
  $(g_i)_{i \in \N} \in \Omega$ one has
  $\nu_{w_n} \to \nu_\infty$ weakly as $n \to \infty$.
  Moreover:
  \begin{enumerate}[(C1)]
  \item
    \label{conj: support}
    The support of $\nu_\infty$ is contained in the union of the unit
    circle, a continuous curve from $\zeta_3$ to $\zetabar_3$, and the
    image of said curve under $z \mapsto 1/z$.  The curve crosses the real
    axis at about $-0.78$.
    
  \item
    \label{conj: circle}
    Divide the circle into left and right arcs at $\zeta_3$ and
    $\zetabar_3$: 
    \begin{equation}
      \label{eq: arcs}
      \cA_L := \setdef{ t= e^{i \theta}}{|\theta - \pi| < \pi/3}
      \mtext{and}
      \cA_R := \setdef{ t = e^{i \theta}}{| \theta| < 2 \pi /3}.
    \end{equation}
    As in \ref{item: two-thirds circle}, exactly 2/3 of the mass of
    $\nu_\infty$ is on $\cAbar_R$, and $\nu_\infty$ is a multiple of
    Lebesgue measure on that arc.  Moreover, $\nu_\infty$ is
    absolutely continuous with respect to Lebesgue measure on $\cA_L$,
    with total mass $(7 - 3 \sqrt{5})/12 \approx 0.024316$. As a
    consequence, the total mass of the unit circle is
    $(5 - \sqrt{5})/4 \approx 0.690983$.
    
  \item
    \label{conj: other arc}
    As in \ref{item: other arc}, the measures along the curves joining
    $\zeta_3$ to $\zetabar_3$ are absolutely continuous with respect
    to Lebesgue measure on them.
  \end{enumerate}
\end{conjecture}

\subsection{The spectral radius of the Burau representation}

As described in Section~\ref{sec: background}, the Alexander
polynomial of any closed \3-braid can be computed via the (reduced)
Burau representation
$B_t \maps \Braid{3} \to \GL{2}{\big(\Z[t^{\pm 1}]\big)}$.  To each
\emph{positive} braid $w \in \Braid{3}$, we define in
Section~\ref{sec: spec} the \emph{spectral radius function}
$\rho_w \maps \Dbar \to \R_{\geq 0}$ to be the absolute value of the
largest eigenvalue of $B_t(w)$.  (For arbitrary braids, this makes
sense except possibly at $t = 0$, where $\rho_w$ can have a pole.)
The function $\rho_w$ turns out to be continuous and subharmonic, and
takes the value $1$ at any root of $\Delta_w(t)$; see Lemmas~\ref{lem:
  obs} and~\ref{lem: spec basics}.

\begin{figure}
  \centering
  \begin{tikzpicture}
    \node at (0, 0) {\begin{tikzoverlay}[width=6.6cm]{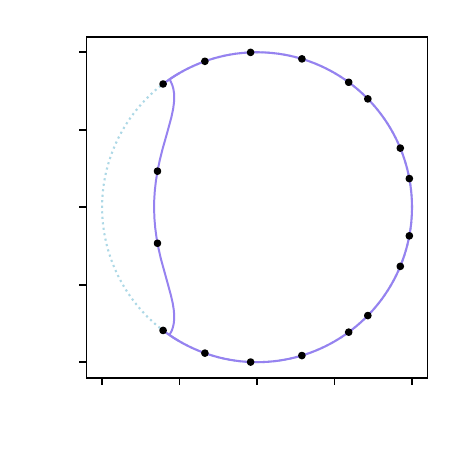}
  \begin{scope}[nmdstd, font=\scriptsize]
  \draw (22.684343, 12.833333) node[below] {$-1.0$};
  \draw (39.902357, 12.833333) node[below] {$-0.5$};
  \draw (57.120370, 12.833333) node[below] {$0.0$};
  \draw (74.338384, 12.833333) node[below] {$0.5$};
  \draw (91.556397, 12.833333) node[below] {$1.0$};
  \draw (16.000000, 19.517677) node[left] {$-1.0$};
  \draw (16.000000, 36.735690) node[left] {$-0.5$};
  \draw (16.000000, 53.953704) node[left] {$0.0$};
  \draw (16.000000, 71.171717) node[left] {$0.5$};
  \draw (16.000000, 88.389731) node[left] {$1.0$};
  \draw (42.750000, 54) node[font=\small] {$R_w$};
  \begin{scope}[shift={(57.12037037, 53.95370370)},
    xscale=34.43602694, yscale=34.43602694]
  \end{scope}
\end{scope}
\end{tikzoverlay} };
    \node at (8.0, 0.29) {\begin{tikzoverlay}[height=8.1cm]{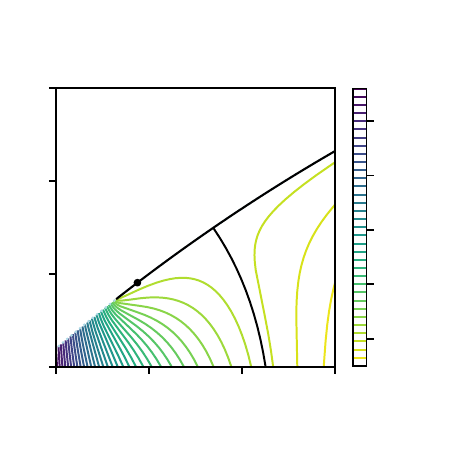}
  \begin{scope}[nmdstd, font=\scriptsize]
  \draw (12.500000, 15.259259) node[below] {$-0.65$};
  \draw (33.166667, 15.259259) node[below] {$-0.60$};
  \draw (53.833333, 15.259259) node[below] {$-0.55$};
  \draw (74.500000, 15.259259) node[below] {$-0.50$};
  \draw (9.259259, 18.500000) node[left] {$0.75$};
  \draw (9.259259, 39.166667) node[left] {$0.80$};
  \draw (9.259259, 59.833333) node[left] {$0.85$};
  \draw (9.259259, 80.500000) node[left] {$0.90$};
  \draw (84.695741, 24.739216) node[right] {$1.0$};
  \draw (84.695741, 36.817647) node[right] {$1.5$};
  \draw (84.695741, 48.896078) node[right] {$2.0$};
  \draw (84.695741, 60.974510) node[right] {$2.5$};
  \draw (84.695741, 73.052941) node[right] {$3.0$};
  \draw (50, 50) node[above left, font=\small] {$R_w$};
  \begin{scope}[shift={(281.16666667, -291.50000000)},
                xscale=413.33333333, yscale=413.33333333]
  \end{scope}
  \begin{scope}[shift={(78.37500000, -42.90000000)},
                xscale=3.08000000, yscale=67.63921569]
  \end{scope}
  \end{scope}
\end{tikzoverlay} };
  \end{tikzpicture}

  \vspace{-0.5cm}

  \caption{At left is the closed set $R_w$ in $\Dbar$ for the braid
    $w = (\sigma_1 \sigma_2 \sigma_1)^3 \sigma_2^4 \sigma_1^6
    \sigma_2^3$. The roots of $\Delta_\what$ are indicated by the
    black dots.  At right is a closeup of part of $R_w$, which is shown
    in black, against a contour plot of the restriction of $\rho_w$ to
    $\Dbar$.}
  \label{fig: R_w}
\end{figure}
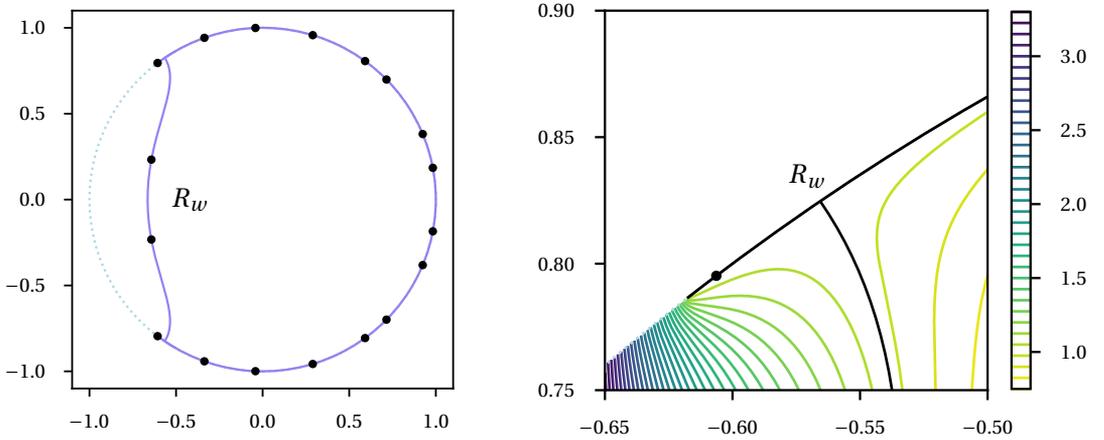

Set $R_w := \setdef{z \in \Dbar}{\rho_w(z) = 1}$. As noted, this
contains the roots $Z_w \cap \Dbar$, and a motivating example is shown
in Figure~\ref{fig: R_w}.  We now describe our results about $R_w$
which bear on our conjectures.  First, we prove in
Lemma~\ref{lem:compact} that the arc $\cA_R$ from \eqref{eq: arcs} is
contained in $R_w$, consistent with \ref{item: two-thirds circle}.
The intersection $R_w \cap \bdry \D$ sometimes consists of just
this arc, but it can be bigger and even disconnected, as in
Figure~\ref{fig: spec on circle}.

  
For the intersection between $R_w$ and the real axis, we prove the
following as part of Theorem~\ref{thm: rho neg reals}:
\begin{theorem}
  Suppose $\what$ is a knot, and 
  $w$ gives a pseudo-Anosov map of the thrice-punctured disc.  Then the
  set $R_w$ intersects the real axis in exactly two points: $1$ and a
  point in $\big(-1, (\sqrt{5} - 3)/2\big]$ where
  $(\sqrt{5} - 3)/2 \approx -0.38$.
\end{theorem}
This is consistent with our conjecture of an arc of $R_w$ joining
$\zeta_3$ to $\zetabar_3$.  Recall that the condition that $w$ is
pseudo-Anosov is generic: the probability that a random word of length
$n$ is not pseudo-Anosov tends to $0$ exponentially in $n$ \cite{Maher2012}.

\begin{figure}
  \centering
  \begin{tikzoverlay}[width=8cm]{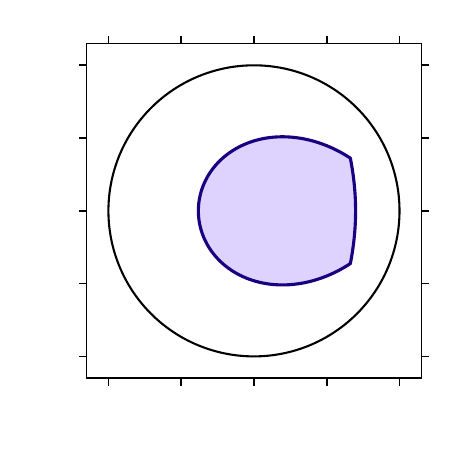}
  \begin{scope}[nmdstd, font=\scriptsize]
  \draw (24.092190, 12.708333) node[below] {$-1.0$};
  \draw (40.263688, 12.708333) node[below] {$-0.5$};
  \draw (56.435185, 12.708333) node[below] {$0.0$};
  \draw (72.606683, 12.708333) node[below] {$0.5$};
  \draw (88.778180, 12.708333) node[below] {$1.0$};
  \draw (16.000000, 20.800523) node[left] {$-1.0$};
  \draw (16.000000, 36.972021) node[left] {$-0.5$};
  \draw (16.000000, 53.143519) node[left] {$0.0$};
  \draw (16.000000, 69.315016) node[left] {$0.5$};
  \draw (16.000000, 85.486514) node[left] {$1.0$};
  \draw (62.7, 53) node[font=\small] {$\cT$};
\end{scope}
\end{tikzoverlay}
  \caption{The region $\cT$ inside the unit disk from
    Theorem~\ref{thm: root-free intro}.  If $w \in \Braid{3}$ is
    positive and not a power of $\sigma_1$ or $\sigma_2$, then $\cT$ 
    is disjoint from $R_w$ and hence does not contain any root of
    $\Delta_\what$.  Compare with Figure~\ref{fig: pos many}.
    }
  \label{fig: new T intro}
\end{figure}
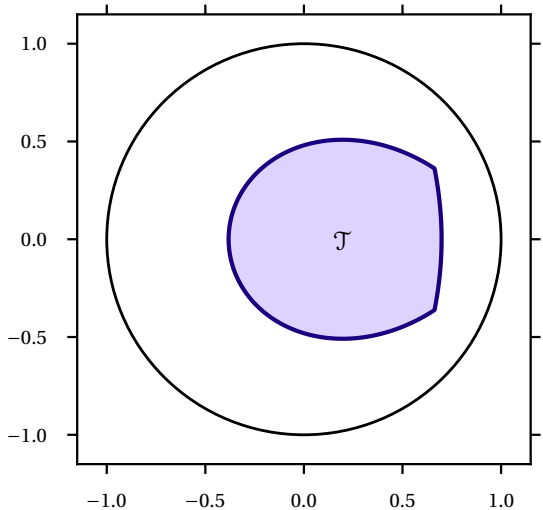

We show the existence of a root-free region around $t = 0$ in
Theorem~\ref{thm: root-free}, which we restate here as:
\begin{theorem}
  \label{thm: root-free intro}
  The region $\cT \subset \D$ of area $\approx 0.91$ shown in
  Figure~\ref{fig: new T intro} is disjoint from $R_w$ for all positive $w$
  that are not a power of $\sigma_1$ or $\sigma_2$.
\end{theorem}
In the exceptional case of $w = \sigma_i^n$, one has that $ \rho_w$ is
identically $1$ on $\Dbar$ and hence $R_w = \Dbar$; here,
$\Delta_w = 0$ for such links.  In Section~\ref{S:right-plane}, we
give a concrete conjecture about $B_t(w)$ that implies that $\cT$
could be enlarged to include the part of $\D$ where $\Re(z) > 0$, at
least when $\what$ is a knot. 

A key technical tool for studying the roots is the following
statement, which may be of independent interest.  A Laurent polynomial
$p(t) \in \Z[t^{\pm1}]$ is \emph{definite} when all its coefficients
have the same sign; for example, $t^3 + t + 2$, $-2t^2 - 1$, and $0$
are all definite, but $t^3 - 2t + 1$ is not.  Combining Theorem
\ref{thm:possign} and Corollary \ref{cor:tracedef}, we show:
\begin{theorem}\label{thm:intro-possign} 
  For $w \in \Braid{3}$, all entries of $B_{-t}(w)$ are definite.
  Moreover the trace $\tr\big(B_{-t}(w)\big)$ is definite.
\end{theorem}

\subsection{Roots on the unit circle}

For a link $L$, the roots of $\Delta_L$ on the unit circle $S^1$ are
closely related to the much-studied \emph{signature function}
$\sigma_L \maps S^1 \to \Z$; see Section~\ref{sec: sig fn}.  When
$L = \what$ for $w \in \Braid{3}$, the function $\sigma_L$ can be
analyzed using the work of Gambaudo and Ghys \cite{GambaudoGhys2005}.
We use this to prove two types of lower bounds on the number of roots
of $\Delta_L$ on the unit circle.  

The strongest results are for the arc $\cA_R$.  
Specifically, Corollary~\ref{cor: 2/3 of roots} shows that, for essentially all positive braids, at least
$\frac{2}{3}\big(\deg(\Delta_L) - 1\big)$ of the roots of $\Delta_L$ occur on this arc, 
while 
Corollary~\ref{cor: arc lower bound} gives that any
sub-arc has at least the expected number of roots given its length; 
these results confirm part of \ref{conj: circle}, that predicts the limiting measure on $\cA_R$ to be Lebesgue measure.  

On the other hand, the results on the arc $\cA_L$ are perhaps the most surprising. 
On that arc, we analyze the number of roots of 
\emph{random} positive braid generated by the uniform measure on
$\{\sigma_1, \sigma_2\}$.  Note that, on the arc $\cA_L$, the image of
the Burau representation $B_t$ is a non-elementary semigroup, hence
the drift (Lyapunov exponent) of the random walk $B_t(w_n)$ is
positive, and the leading eigenvalue of $B_t(w_n)$ grows exponentially
fast in $n$: thus, one could naively conjecture that asymptotically
there are no roots on $\cAbar_L$.  However, remarkably this is not the
case: in fact, we prove (Corollary \ref{cor: circle root summary})
that a \emph{positive} proportion of roots lies on the arc $\cAbar_L$.
In particular, we show that almost surely at least
$(7 - 3 \sqrt{5})/12 \approx 2.4\%$ of the roots of $\Delta_\what$ are
on this arc, partially confirming \ref{conj: circle}.  This is a
corollary of our analysis of $\abs{ \sigma_\what(-1) }$ for such a
random positive word.  By \cite{GambaudoGhys2005}, the value
$\sigma_\what(-1)$ is essentially a quasimorphism on $\Braid{3}$.
Using Bj\"orklund-Hartnick's central limit theorem for quasimorphisms
\cite{BjorklundHartnick2011}, we prove the following as
Theorem~\ref{T:clt}:

\begin{theorem}
  \label{T:clt-intro}
  Consider the random walk
  $w_n := g_1 \cdots g_n$ on $\Braid{3}$ generated by the uniform
  measure on $\{\sigma_1, \sigma_2\}$.  The signature
  $\abs{\sigma_{\widehat{w}_n}(-1)}$ follows a central limit theorem,
  with positive drift and positive variance.  Namely, setting $\ell_0 = (5 - \sqrt{5})/4$, there exists
  $\nu_0 > 0$ such that for any $a < b$ we have
  \[
    \P\left(\frac{\abs{\sigma_{\widehat{w}_n}(-1)} - \ell_0 n}{\nu_0
        \sqrt{n}} \in [a, b] \right) \to \frac{1}{ \sqrt{2 \pi}}
    \int_a^b e^{-t^2/2} \dt.
  \]
\end{theorem}

From Theorem~\ref{T:clt-intro}, specifically the existence and value of the
drift $\ell_0$, and the previously mentioned
Corollaries~\ref{cor: 2/3 of roots} and \ref{cor: arc lower bound}, we
have:

\begin{corollary}
  \label{cor: circle root summary}
   Consider the random walk
  $w_n$ generated by the uniform
  measure on $\{\sigma_1, \sigma_2\}$. 
Then for almost every sample path $(w_n)$:
\begin{enumerate}
\item The proportion of roots in the unit circle satisfies
  \[
    \liminf_{n \to \infty}
       \frac{\#\setdef{t \in S^1}{\Delta_{\what_n}(t) = 0}}{n}
    \geq \frac{5 - \sqrt{5}}{4} > \frac{2}{3}.
  \]
  Moreover, 
  \[
    \liminf_{n \to \infty} 
    \frac{
      \#\setdef{t \in \cAbar_L}{\Delta_{\what_n}(t) = 0} }
      {n}
    \geq \frac{7 - 3\sqrt{5}}{12} > 0.
  \]
  
\item For any $0 \leq \theta_1 < \theta_2 < 2\pi/3$, we have
  \[
    \liminf_{n \to \infty} \frac{
      \#\setdef{t = e^{i \theta}}{\mbox{
          $\theta \in [\theta_1, \theta_2]$ and $\Delta_{\what_n}(t) = 0$}}}
      {n}
    \geq \frac{\theta_2 - \theta_1}{2 \pi}. 
    \]
    As a consequence, 
     \[ \liminf_{n \to \infty} 
    \frac{
      \#\setdef{t \in \cA_R}{\Delta_{\what_n}(t) = 0} }
      {n}
    \geq \frac{2}{3}.
 \]
 \end{enumerate}
\end{corollary}

The most difficult part of Theorem~\ref{T:clt-intro} is computing the
drift $\ell_0$.  The predictions of \ref{conj: circle} include that
the lower bounds in Corollary~\ref{cor: circle root summary} are all
sharp, and that the limit infimums can be replaced by limits.

\subsection{Lyapunov exponents and bifurcation measures}
\label{sec: lyap intro}

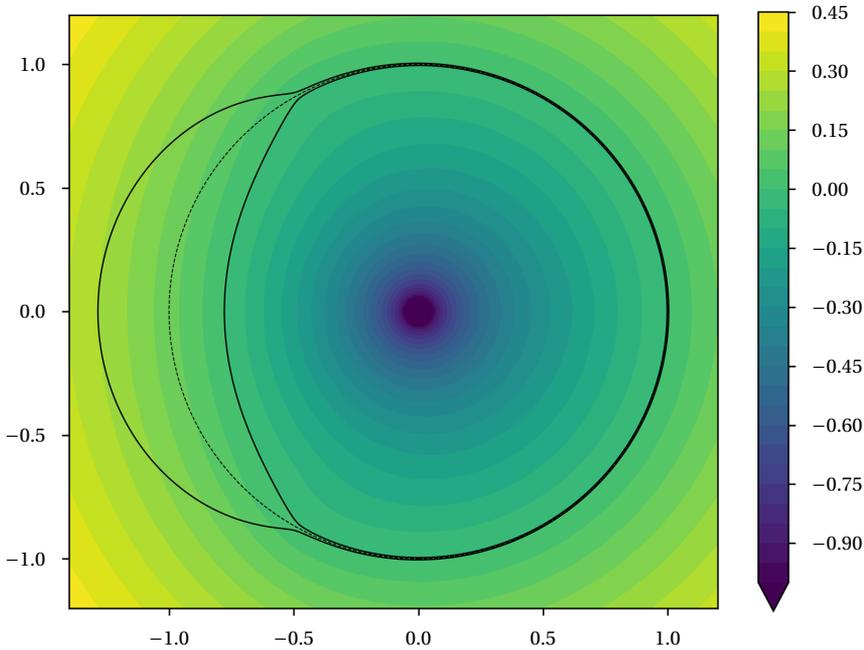
\begin{figure}
  \centering
  \begin{tikzoverlay}[width=12cm]{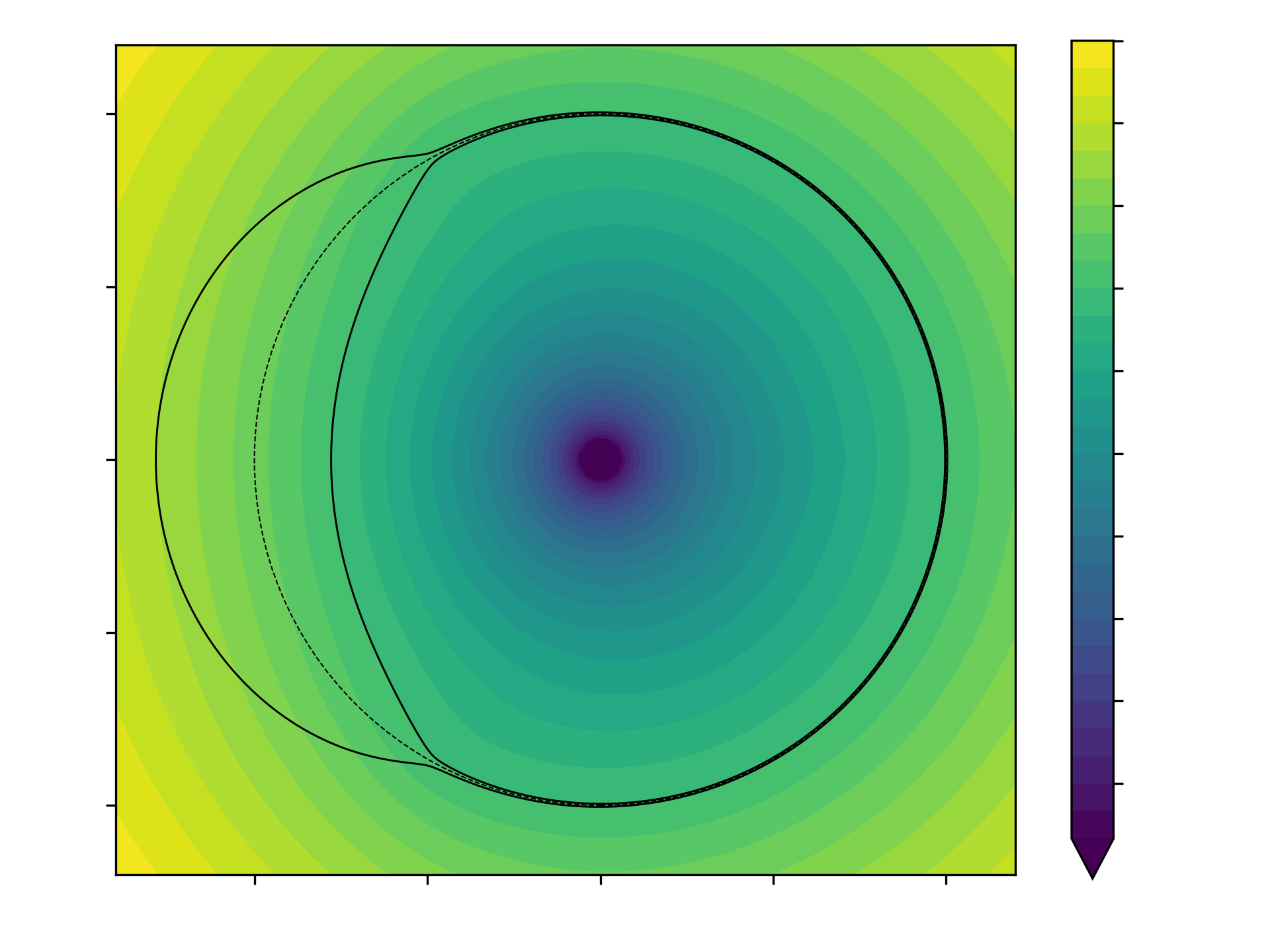}
  \begin{scope}[nmdstd, font=\scriptsize]
  \draw (20.031050, 4.590411) node[below] {$-1.0$};

  \draw (33.649506, 4.590411) node[below] {$-0.5$};

  \draw (47.267962, 4.590411) node[below] {$0.0$};

  \draw (60.886418, 4.590411) node[below] {$0.5$};

  \draw (74.504874, 4.590411) node[below] {$1.0$};

  \draw (7.617188, 11.556891) node[left] {$-1.0$};

  \draw (7.617188, 25.175347) node[left] {$-0.5$};

  \draw (7.617188, 38.793803) node[left] {$0.0$};

  \draw (7.617188, 52.412260) node[left] {$0.5$};

  \draw (7.617188, 66.030716) node[left] {$1.0$};
  
  \draw (89.197660, 13.269260) node[right] {$-0.90$};

  \draw (89.197660, 19.772329) node[right] {$-0.75$};

  \draw (89.197660, 26.275397) node[right] {$-0.60$};

  \draw (89.197660, 32.778465) node[right] {$-0.45$};

  \draw (89.197660, 39.281534) node[right] {$-0.30$};

  \draw (89.197660, 45.784602) node[right] {$-0.15$};

  \draw (89.197660, 52.287670) node[right] {$0.00$};

  \draw (89.197660, 58.790738) node[right] {$0.15$};

  \draw (89.197660, 65.293807) node[right] {$0.30$};

  \draw (89.197660, 71.796875) node[right] {$0.45$};
  \begin{scope}[shift={(47.26796207, 38.79380342)},
                xscale=27.23691239, yscale=27.23691239]
  \end{scope}
  \begin{scope}[shift={(84.37825521, 52.28767012)},
                xscale=3.30030716, yscale=19.50920488]
              \end{scope}
  \end{scope}
\end{tikzoverlay}

  \vspace{0.2cm}
  
  \caption{A contour plot of a numerical approximation to the Lyapunov
    exponent function $\lambda(t)$, together with the locus
    $\lambda(t) = \log^+ \abs{t}$ drawn in black, which includes
    $\Dbar \cap \{\lambda(t) = 0\}$. For scale, the unit circle is
    shown as a dotted line. Note the similarity with Figure~\ref{fig:
      pos many}.  The approximation of $\lambda$ at each $t$ is based
    on averaging the result of 10,000 walks of length 300, and this
    was computed at 90,000 distinct points.  }
  \label{fig: lyap}
\end{figure}

For random positive braids $w_n = g_1 \cdots g_n$ in $\Braid{3}$
generated by a finitely supported measure $\mu$, we define for each
$t \in \C$ the \emph{Lyapunov exponent} of the random walk as
\begin{equation*}
  \lambda(t) := \lim_{n \to \infty} \frac{1}{n} \int_{\Braid{3}}
    \log \big\Vert B_t(g) \big\Vert \ d \mu^{* n}(g) 
\end{equation*}
where $\norm{\cdotspaced}$ is e.g.~the operator norm.  As a function
$\C \to [-\infty, \infty)$, we show in Lemma~\ref{lem: weak sub} that
$\lambda$ is continuous and subharmonic on $\C \setminus \cAbar_R$.
See Figure~\ref{fig: lyap} for a plot of $\lambda$.

Because $\abs{\det B_t(g)} = \abs{t}^{\#{g}}$ is not $1$, where $\#{g}$ the
length of the positive word $g$, we consider
\[
  \chi(t) := \max \big\{ \lambda(t), \, \log^+ |t| \big\}
  \mtext{where $\log^+ x := \max \big\{ 0, \, \log x \big\}$.}
\]
The \emph{bifurcation measure} is then
\[
  \nu_\bif := \frac{1}{2 \pi} d d^c \chi,
\]
where $d d^c$ in this case equals the usual Laplacian (but we avoid
the notation $\Delta$, already denoting the Alexander polynomial).  We
posit that $\nu_\bif$ is the limiting measure $\nu_\infty$ in
Conjecture~\ref{conj: main}, and we prove the following partial result
towards this.

Let us define
$F := \overline{\setdef{ t \notin \cA_R }{\lambda(t) = \log^+|t|}}$
and $U := \mathbb{C} \setminus F$.  The set $U$ contains $\cA_L$, and
we conjecture it is the complement of a continuous loop in $\C$
intersecting the unit circle at $\zeta_3$ and $\zetabar_3$, namely the
union of two curves posited in \ref{conj: support}.  Hence $U$ would
contain $S^1 \setminus \{\zeta_3, \zetabar_3\}$ and it should have
$\nu_\infty$-measure approximately $69.1\%$.
We prove the following equidistribution result as
Theorem~\ref{T:equid} and Proposition~\ref{prop: last conj}:

\begin{theorem} \label{T:equid-intro}
  Let $\mu$ be a generating, finitely supported measure on the
  positive semigroup $\Braid{3}^+$, and let $w_n$ be the random walk
  driven by $\mu$.  Then for almost every sample path $(w_n)$, we have
  the convergence
  \[
    \lim_{n \to \infty} \nu_{w_n}\vert_{U} = \nu_\bif\vert_{U}
  \]
  as measures on $U$, i.e.~as elements of the dual of
  $C^0_c(U)$. Moreover, if the set $\setdef{t \in \D}{\lambda(t) = 0}$
  has zero Lebesgue measure, then the above convergence holds with
  $U = \C$.
\end{theorem}

The theorem in particular applies if $\mu$ is the uniform measure on
$\{ \sigma_1, \sigma_2 \}$, as considered so far.  As a consequence
(Corollary~\ref{C:zero-density}), the limiting measure assigns zero
mass to the locus where $\lambda$ is negative:

\begin{corollary}
Let $\mu$ be as above, and let $K$ be a compact set contained in
  $\setdef{t \in \D }{\lambda(t) < 0}$. Then for almost every sample
  path $(w_n)$,
$$\lim_{n \to \infty} \frac{1}{n} \# \setdef{ t \in K}{\Delta_{\what_n}(t) = 0} = \nu_\bif(K) = 0.$$
\end{corollary}

\subsection{Further open questions}
It is natural to conjecture that the set
$\setdef{t \in \D}{\lambda(t) = 0}$ has zero Lebesgue measure, which
would imply that the equidistribution of Theorem~\ref{T:equid-intro}
holds everywhere on $\C$. Let us note that in general the zero set of
a subharmonic function need not have zero measure; moreover, it would
be easier to prove such a statement knowing that $\lambda(t)$ is
smooth in $\D$.  However, one should note that in general the function
$\lambda(t)$ is H\"older continuous \cite{LePage}, but there are
examples of random walks, even in two generators with polynomial
entries in $t$, for which $\lambda(t)$ is not smooth
\cite{Simon-Taylor}. Note that, assuming $\nu_{\bif}$ charges the arc
$\cA_L$, the derivative of $\lambda(t)$ in the radial direction cannot
be continuous across $\cA_L$, then $\lambda(t)$ cannot be globally
smooth.

Many similar phenomena can be observed for braid groups with more than
three strands, as already remarked in \cite{Dehornoy2015}. Our
techniques here should be useful for studying these phenomena, though
some methods (e.g.~Section~\ref{S:definite}) are very specific to the
\3-strand case.  Regardless, we do not attempt doing so in this paper.
Finally, it would be interesting to study what happens for
non-positive braids, as the bottom part of Figure~\ref{fig: contrast}
is quite intriguing. While one can still study the spectral radius in
this setting, the coefficients of the Burau representation are no
longer holomorphic in $t$.

\subsection{Acknowledgements} Dunfield was partially supported by US
National Science Foundation grants DMS-1811156 and DMS-2303572, and by
a fellowship from the Simons Foundation (673067, Dunfield).  Tiozzo
was partially supported by NSERC grant RGPIN 2017-06521 and an Ontario
Early Researcher Award.  We thank Partha Dey for helpful conversations
related to Section~\ref{sec: on circle}.  Finally, we thank the
referee for their very helpful suggestions, and especially for
pointing us to \cite{BockerViana2017} which improved Lemma~\ref{lem:
  weak sub}.

\section{Background}
\label{sec: background}

We consider the 3-strand braid group
\[
  \Braid{3} = \spandef{\sigma_1, \sigma_2}{\sigma_1 \sigma_2 \sigma_1
    = \sigma_2 \sigma_1 \sigma_2}
\]
and its reduced Burau representation
$B_t \maps \Braid{3} \to \GL{2}{\big(\Z[t^{\pm 1}]\big)}$ defined by
\begin{align*}
  B_t(\sigma_1) &:= \matr{-t}{1}{0}{1}  & B_t(\sigma_1^{-1}) &:=
  \matr{-1/t}{1/t}{0}{1} \\
  B_t(\sigma_2) &:= \matr{1}{0}{t}{-t}  &B_t(\sigma_2^{-1}) &:= \matr{1}{0}{1}{-1/t}.
\end{align*}
A \emph{knot} is a smooth embedding of $S^1$ into $S^3$; a \emph{link}
is a disjoint union of knots.  An \emph{orientation} of a link is a
choice of preferred direction on each component. Given a braid
$w \in \Braid{3}$, let $\what$ denote the oriented link in $S^3$ that
is its braid closure.  The Alexander polynomial $\Delta_{\what}(t)$ of
$\what$ is an element of $\Z[t^{\pm 1}]$ that is well-defined up to
multiplication by a unit, i.e.~by $\pm t^n$ for $n \in \Z$. Burau
showed, see e.g.~\cite[Theorem~3.11]{Birman1974}, that for
$w \in \Braid{3}$ we have:
\begin{equation}\label{eq:alex}
\Delta_{\what}(t) = \frac{\det( B_t(w) - 1)}{t^2 + t + 1}.
\end{equation}
In particular, $t^2 + t + 1$ always divides the Laurent polynomial
$\det( B_t(w) - 1)$.

Consider the homomorphism $\Braid{3} \to \Z$ which sends each
$\sigma_i$ to $1$.  Concretely, given a word $w$ in the $\sigma_i$,
its image is the number of times $\sigma_1$ and $\sigma_2$ appear
minus the number of times $\sigma_1^{-1}$ and $\sigma_2^{-1}$ appear;
we denote this value by $\#w$.  Note that:
\begin{equation}\label{eq: B det}
  \det\big( B_t(w) \big) = (-t)^{\# w}.
\end{equation}

\subsection{Roots of the Alexander polynomial}
\label{sec: alex props}

As noted, the Alexander polynomial $\Delta_L(t)$ of an oriented link
$L$ is only defined up to multiplication by $\pm t^n$ for $n \in \Z$.
Hence, whether $0$ is a root of $\Delta_L(t)$ is ambiguous; our
convention is that $0$ is a root of $\Delta_L(t)$ if and only if the
latter is itself $0$.  For any link $L$, the Alexander polynomial is
symmetric with respect to $t \mapsto t^{-1}$ in that one has
$\Delta_{L}(1/t) = \pm t^e \Delta_L (t)$ for some $e \in \Z$, see
e.g.~\cite[Theorem~6.10]{Lickorish1997}. Hence the roots of
$\Delta_L(t)$ are invariant under $z \mapsto z^{-1}$ as is evident in
Figure~\ref{fig: dehornoy}, and we record this fact as:

\begin{lemma}
  \label{lem: sym roots}
  The set of roots $Z_L := \setdef{z \in \C}{\Delta_{L}(z) = 0}$
  is invariant under $z \mapsto 1/z$ and is determined by $Z_L \cap
  \Dbar$, where $\Dbar = \setdef{z \in \C}{\abs{z} \leq 1}$.
\end{lemma}
For the special case of a knot $K$, one has that $\Delta_K(t)$ is
independent of the orientation of $K$, and that $\Delta_K(1) = \pm 1$
by \cite[Theorem~6.10(ii)]{Lickorish1997}.  Also, for knots, 
$\deg \Delta_K$ is even and $\Delta_{K}(1/t) = +t^e \Delta_K (t)$
\cite[Corollary~6.11]{Lickorish1997}; hence $\Delta_K$ can be
normalized so that $\Delta_K(t) = t^d g(t + 1/t)$ where
$g(x) \in \Z[x]$ has degree $d$ with $g(0) \neq 0$. The roots of $g$
determine those of $\Delta_K$, and are what is plotted in
Figure~\ref{fig: contrast} in what we call the \emph{trace
  coordinates}.

\subsection{Alexander polynomials of positive braids}

A braid $w \in \Braid{n}$ is positive when it is a product of only
positive powers of the $\sigma_i$.
Burau's original application in 1935 of his representation was to
compute the degree of $\Delta_\what$ when $w \in \Braid{n}$ is
positive.  In our special setting, the result is:

\begin{theorem}[\cite{Burau1935}]
  \label{thm: burau pos}
  If $w$ is a positive 3-braid other than $\sigma_i^n$, then
  $\Delta_\what$ is monic of degree $\# w - 2$.
\end{theorem}
By Stallings, one knows in this situation that moreover $\what$ is a
fibered link \cite{Stallings1978}.

\begin{proof}[Proof of Theorem~\ref{thm: burau pos}]
  The link $\what$ only depends on $w$ up to conjugacy, which allows us to
  cyclically permute its letters without changing $\what$.  Therefore, we
  can assume $w = \sigma_1 \sigma_2 w'$ where $w'$ is a positive word
  of length $\# w - 2$.  If $M$ is a matrix of polynomials in $t$,
  define $\deg M = \max \big(\deg M_{i,j}(t) \big)$ and note $\deg(M
  N) \leq \deg(M) + \deg(N)$ for any pair of such matrices.  Now
  $B_t(\sigma_i)$ both have degree 1, as does $B_t(\sigma_1
  \sigma_2)$.  Hence $\deg(B_t(w)) \leq \# w - 1$.  We have
  \begin{equation}
    \label{eq: burau det}
    \det(B_t(w) - 1) = \det(B_t(w)) - \tr(B_t(w)) + 1 = (-t)^{\#w}
    - \tr(B_t(w)) + 1
  \end{equation}
  and since the degree of $\tr(B_t(w))$ is at most $\deg(B_t(w))$,
  we see the right hand side has leading term $(-t)^{\#w}$.  Since $t$
  divides every entry of $B_t(\sigma_1 \sigma_2)$, it divides every
  entry of $B_t(w)$ and hence divides $\tr(B_t(w))$ as well.  In
  particular, the constant term of the right hand side is indeed $1$.
  The conclusion is now immediate from the formula that
  $\Delta_\what(t) = \det(B_t(w) - 1)/(t^2 + t + 1)$.
\end{proof}

\subsection{Squier's form}
\label{sec: squier}

Squier in \cite{Squier1984} discovered that the Burau representation
preserves a bilinear form which is in a certain sense Hermitian; see
also \cite[Section~2.3]{GambaudoGhys2005} for a geometric
interpretation.  Consider a formal variable $s$ where $s^2 = t$ and
regard $\Z[t^{\pm 1}]$ as a subring of $\Z[s^{\pm 1}]$.  For a matrix
$A$ with entries in $\Z[s^{\pm 1}]$, let $A^h$ denote taking the
transpose and applying the ring involution induced by $s \mapsto s^{-1}$.
Then for all $w \in \Braid{3}$ one has:
\begin{equation}\label{eq:form}
  B_t(w)^h \cdot J \cdot B_t(w) =
    J \mtext{where} J = \matr{s + s^{-1}}{-s^{-1}}{-s}{s+s^{-1}}.
\end{equation}
Thus the Burau representation preserves the ``Hermitian'' form
associated to $J$.

If we replace $s$ with a complex number of norm 1, say
$s = e^{i \theta/2}$ for $\theta \in \R$ so that $t = e^{i \theta}$,
then each matrix $B_t(w)^h$ is the ordinary complex conjugate
transpose $B_t(w)^*$ of $B_t(w)$ and $J^* = J$.  Thus, one gets a
genuine Hermitian form $H_s$ on $\C^2$ from $J$, and the
representation $B_t \maps \Braid{3} \to \GL{2}{\C}$ is unitary with
respect to $H_s$. 

\begin{lemma}
  \label{lem: sig on circle}
  For $t \in \C$ of norm 1, the Hermitian form $H_s$ is positive
  definite when $t \in \cA_R$, degenerate but nontrivial when
  $t \in \{\zeta_3, \zetabar_3\}$, and of signature $(1, 1)$ for
  $t \in \cA_L$.  Equivalently, the image of $B_t$ is conjugate into
  $\mathrm{U}(2)$, $\mathrm{U}{(1, 0)}$, or $\mathrm{U}(1, 1)$
  respectively.
\end{lemma}

\begin{proof}
We need to determine the signs of the eigenvalues of the Hermitian
matrix $J$.  Set $t = e^{i \theta}$ for $\theta \in (-\pi, \pi]$.  As
$\det J = t + t^{-1} + 1$, we have $\det J > 0$ and $J$ is definite if
and only if $2\cos \theta > -1$, or equivalently $t \in \cA_R$.
Moreover, $J$ is positive definite in that range as
$\tr J = 4 \cos(\theta/2) > 0$.  When $t \in \cA_L$, we have
$\det J < 0$ and so $J$ has signature $(1, 1)$.  When
$\theta = \pm\frac{2 \pi}{3}$, we have $\det J = 0$ and $\tr J = 2$,
which gives the remaining claim.
\end{proof}

\subsection{Classification of semigroups generated by $B_t(\sigma_i)$}
The results of this subsection are not used until Section~\ref{sec:
  Lyapunov}, and so you may initially wish to skip ahead to
Section~\ref{sec: MCG}.  Let $t \in \C^\times$, and let $u$ be such
that $u^2 = -t$.  The map
$\sigma_i \mapsto \tau_i := u^{-1} B_{-u^2}(\sigma_i)$ induces a
homomorphism $\Braid{3} \to \PSL{2}{\C} = \Isom^+(\H^3)$.  For each
$t$, let $\Gamma^+_t$ denote the semigroup of $\PSL{2}{\C}$ generated
by $\tau_1, \tau_2$; note that $\Gamma_t^+$ depends only on $t$ and
not on the choice of $u$, since swapping $u$ by $-u$ multiplies the
matrix by $-1$, yielding the same element in $\PSL{2}{\C}$.

We obtain the following classification, depending on $t$.

\begin{proposition} \label{P:non-elt}
Let $\Gamma^+_t$ be the semigroup generated by $\tau_1, \tau_2$. We have: 
\begin{enumerate}
\item
  \label{item: nonelem}
  if $|t| \neq 0, 1$, then $\tau_1$ and $\tau_2$ are two
  independent loxodromic elements, so $\Gamma_t^+$ is non-elementary;

\item
  \label{item: loxo}
  if $t \in \cA_L$, then $\Gamma^+_t$ contains two
  independent loxodromic elements;
 
\item
  \label{item: elliptic}
  if $t \in \cA_R$, then every element of $\Gamma^+_t$ is elliptic,
  with a common fixed point inside $\H^3$;
 
\item
  \label{item: 3rd roots}
  if $t$ is $\zeta_3$ or $\zetabar_3$, then both $\tau_1$ and $\tau_2$
  are elliptic, of order $6$, with a unique common fixed point in
  $\bdry \H^3$. 
\end{enumerate}
\end{proposition}


\begin{proof}
Throughout, we exclude $t = 0$. Recall a subgroup of $\PSL{2}{\C}$ is
\emph{reducible} when it fixes a point in $\bdry \H^3$. A subgroup
generated by two elements is reducible if and only if the trace of
their commutator is $2$; in our case
$\tr\big([\tau_1, \tau_2]\big) = 1 - t - 1/t$, so this happens exactly
when $t$ is $\zeta_3$ or $\zetabar_3$.  In the latter case, one easily
checks the conclusions in (\ref{item: 3rd roots}) by noting that a
common eigenvector of the $\tau_i$ is $(1, t + 1)$.  As
$\tr(\tau_1) = \tr(\tau_2) = u + 1/u$, both are hyperbolic provided
$u + 1/u \notin [-2, 2]$, which is equivalent to $\abs{u} \neq 1$ and
hence $\abs{t} \neq 1$. Thus in case (\ref{item: nonelem}), the
$\tau_i$ are loxodromic and must be independent, as if they shared a
fixed point then the whole group would be reducible.

So we have reduced to the case of $\abs{t} = 1$ and $t$ not $\zeta_3$
or $\zetabar_3$.  First suppose $t \in \cA_L$ and so $\Re(t) <
-1/2$. Set $\alpha = \tau_1^5 \tau_2$ and
$\beta = \tau_1^4 \tau_2 \tau_1 = \tau_1^{-1} \alpha \tau_1$.  We
calculate
\[
  \tr(\alpha) = \tr(\beta) = -1 + t + t^{-1} = -1 + 2 \Re(t)
\]
which is less than $-2$ and hence $\alpha$ and $\beta$ are loxodromic.
They are independent as if
$\tr\big([\alpha, \beta]\big) = t^3 + t^{-3}$ is $2$ then $t^3 = 1$,
contradicting $\Re(t) < -1/2$. This establishes case~(\ref{item: loxo}).

Finally, if $t \in \cA_R$, Lemma~\ref{lem: sig on circle} shows that
$\Gamma^+_t$ lies in a conjugate of $\mathrm{PU}_2$, which is a
maximal compact subgroup of $\PSL{2}{\C}$; such compact subgroups are
precisely the stabilizers of points in $\H^3$, proving
case~(\ref{item: elliptic}).
\end{proof}

\subsection{Associated mapping classes}
\label{sec: MCG}

We now recall some connections between braid groups and various
mapping class groups, following \cite{FarbMargalit2012}.  Let
$S^b_{g, n}$ denote the orientable surface of genus $g$ with $b$
boundary components and $n$ punctures/marked points.  The
\emph{mapping class group} $\Mod(S)$ of a surface $S$ consists of
orientation-preserving homeomorphisms of $S$ that are the identity on
$\partial S$, modulo isotopies of such.  A 3-strand braid gives a
mapping class of the 3-punctured disc, leading to the group
isomorphism $\Braid{3} \cong \Mod(S^1_{0, 3})$ as per
\cite[Chapter~9]{FarbMargalit2012}.  Consider the order-two
homeomorphism $\iota$ of $S^1_1 := S^1_{1, 0}$ which rotates the
boundary of $S^1_{1}$ by angle $\pi$ and has three fixed points in the
interior of $S^1_{1}$.  Quotienting by $\iota$ gives a double cover
$S^1_{1} \to S^1_{0, 3}$ that is branched at the marked points of
$S^1_{0, 3}$.  By Theorem~9.2 of \cite{FarbMargalit2012}, this cover
induces an isomorphism $\Mod(S^1_{1}) \to \Braid{3}$; note here that
$\iota$ itself is \emph{not} an element of $\Mod(S^1_{1})$ since it
does not fix the boundary component pointwise, and hence is not in the
kernel of this map.

Capping off the boundary of $S^1_{1}$ with a punctured disk turns
$S^1_{1}$ into $S_{1,1} := S^0_{1,1}$.  This gives a
surjective homomorphism $\Mod(S^1_1) \to \Mod(S_{1,1})$ whose kernel
is generated by a Dehn twist parallel to the boundary of $S^1_{1}$
\cite[Chapter~9]{FarbMargalit2012}.  The map
$\Mod(S_{1, 1}) \to \SL{2}{\Z}$ induced by taking the action on
$H_1(S_{1, 1}; \Z)$ is an isomorphism
\cite[Section~2.2.4]{FarbMargalit2012}.  The Burau representation at
$t = -1$ also has image in $\SL{2}{\Z}$, and by
e.g.~\cite[Proposition~2.1]{GambaudoGhys2005} is in fact the
composition:
\[
  \Braid{3} \overset{\cong} \longrightarrow
  \Mod(S^1_1) \twoheadrightarrow \Mod(S_{1, 1})
  \overset{\cong} \longrightarrow \SL{2}{\Z}.
\]
If we set
$\Omega = \sigma_1 \sigma_2 \sigma_1 = \sigma_2 \sigma_1 \sigma_2$,
then by \cite[Section 9.2]{FarbMargalit2012} the center of $\Braid{3}$
is $\pair{\Omega^2}$ and the kernel of
$B_{-1} \maps \Braid{3} \twoheadrightarrow \SL{2}{\Z}$ is
$\pair{\Omega^4}$.

Given a $w \in \Braid{3} \cong \Mod(S^1_{0, 3})$, we can cap off the
unique boundary component with a punctured disc to get an element of
$\Mod(S^0_{0, 4})$. We say $w$ is periodic, reducible, or
pseudo-Anosov depending on the Nielsen-Thurston type of the associated
mapping class of $S_{0, 4}^0$, the sense of Theorem 13.2 of
\cite{FarbMargalit2012}.  Using that forgetting the puncture gives an
isomorphism $\Mod(S^1_1) \to \Mod(S_1)$ by
\cite[Section~2.2.4]{FarbMargalit2012} and the concrete description of
the elements of $\Mod(S_1)$ in \cite[Section~13.1]{FarbMargalit2012},
one deduces the well-known:

\begin{proposition}
  \label{prop: braid trichotomy}
  Suppose $w \in \Braid{3}$ and let $A = B_{-1}(w)$, which is in $\SL{2}{\Z}$.  Then:
  \begin{enumerate}
    \item If $\tr A \in \{-1, 0, 1\}$ or $A = \pm I$ then $w$ is
      periodic.

    \item If $\tr A = \pm 2$ and $A \neq \pm I$ then $w$ is reducible.

    \item \label{item: pA}
      If $\abs{\tr A} > 2$ then $w$ is pseudo-Anosov.  The
      invariant foliations have 1-prong singularities at all punctures
      of $S_{0, 4}^0$ and no singularities in the interior.  Finally,
      the pseudo-Anosov stretch factor is the spectral radius of $A$.
  \end{enumerate}
\end{proposition}

\section{Definiteness of the Burau representation} \label{S:definite}

A Laurent polynomial $p(t) \in \Z[t^{\pm1}]$ is \emph{definite} when
all its coefficients have the same sign; for example, $t^3 + t + 2$,
$-2t^2 - 1$, and $0$ are all definite, but $t^3 - 2t + 1$ is not.  We
will show the following, which will be a key technical tool for the
results in Sections~\ref{sec:realax} and~\ref{sec: Lyapunov}.
\begin{theorem}\label{thm:possign}
  For $w \in \Braid{3}$, all entries of $B_{-t}(w)$ are definite.
\end{theorem}
A definite $p(t)$ is \emph{positive} when all its coefficients are
nonnegative, and \emph{negative} when $-p(t)$ is positive; here,
$p(t) = 0$ is the unique polynomial that is both positive and
negative. Despite Theorem~\ref{thm:possign}, it is often the case that
the entries of $B_{-t}(w)$ have different signs in this sense,
e.g.~$B_{-t}(\sigma_1^2 \sigma_2) = \mysmallmatrix{-t}{t + t^2}{-t}{t}$.
However, we will show using Theorem~\ref{thm:goodsign} below that
\begin{corollary}\label{cor:tracedef}
  For any $w \in \Braid{3}$, the trace of $B_{-t}(w)$ is always definite.
\end{corollary}

To prove these results, we use the Garside structure described
in \cite[Chapter~9]{WordProcessing}.
Set $\Omega = \sigma_1 \sigma_2 \sigma_1 = \sigma_2 \sigma_1
\sigma_2$. Any $w \in \Braid{3}$ can be uniquely written in its
\emph{left-greedy canonical form}:
\begin{equation}\label{eq:normal}
  w =  \Omega^k w_1 w_2 \cdots w_n \mtext{for some $k \in \Z$},
\end{equation}
where each $w_i$ is one of $\sigma_1$, $\sigma_2$, $\sigma_1 \sigma_2$
or $\sigma_2 \sigma_1$ and the only transitions allowed for $w_{k}
\to w_{k+1}$ are those in Figure~\ref{fig: canonical}.  These can be
summarized as saying that if $a \in \{1, 2\}$ and
$b = 3 - a$, then the allowed transitions are:
\[
  \sigma_a \to \{\sigma_a,\  \sigma_a \sigma_b\} \mtext{and}
  \sigma_b \sigma_a \to \{\sigma_a, \ \sigma_a \sigma_b\}.
\]

\begin{figure}
  \centering
  \begin{tikzpicture}[
    nmdstd,
    node distance=55,
    inner sep=5pt,
    every node/.style={shape=rectangle, draw},
    line width=0.8pt,
    ]
    \node (s1) {$\sigma_1$};
    \node (s2) [below of=s1]{$\sigma_2$};
    \node (s1s2) [right of=s1]{$\sigma_1 \sigma_2$};
    \node (s2s1) [right of=s2]{$\sigma_2 \sigma_1$};

    \draw[->] (s1.east) -- (s1s2.west);
    \draw[->] (s2.east) -- (s2s1.west);
    \draw[<->] (s1s2.south) -- (s2s1.north);
    \draw[->] (s1s2.south west) -- (s2.north east);
    \draw[->] (s2s1.north west) -- (s1.south east);

    \draw[->]  (s1) to[out=145, in=-145, looseness=4.8] (s1);
    \draw[->]  (s2) to[out=145, in=-145, looseness=4.8] (s2);
  \end{tikzpicture}
  \caption{The allowed transitions $w_k \to w_{k+ 1}$ in the
    left-greedy canonical form.}
  \label{fig: canonical}
\end{figure}
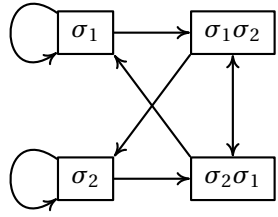

We will refine Theorem~\ref{thm:possign} into the following, which
implies Corollary~\ref{cor:tracedef} for all the braids it covers.
\begin{theorem}\label{thm:goodsign}
  Suppose $w \in \Braid{3}$ is not conjugate to a braid where $n \leq 1$ in
  its canonical form. Then there exists
  $w'$ conjugate to $w$ such that the entries of $B_{-t}(w')$ are all
  definite of the same sign.  If $w$ is a positive word, then $w'$ can be
  chosen to be a positive word as well.
\end{theorem}

Now set $R = \Z[t^{\pm 1}]$ and consider the subset $V$ of $R^2$ of
vectors both of whose entries are definite. Note that $V$ is not
closed under addition, much less is an $R$-submodule.  Now $V$ is the
union of its proper subsets $V_{++}, V_{+-}, V_{-+}, V_{--}$, where
the signs in $V_{\alpha , \beta}$ denote whether the entries
of each vector are positive or negative (top to bottom).  The
$V_{\alpha , \beta}$ are closed under addition, but do have
some overlap; for example, the vector $\smallvec{t}{0}$ is in
$V_{++} \cap V_{+-}$ and $\smallvec{0}{0}$ is in all four.  Set $V_1 =
V_{++} \cup V_{--}$ and $V_2 = V_{+-} \cup V_{-+}$.  We call $V_{++}$
and $V_{--}$ the \emph{parts} of $V_1$, and similarly $V_{+-}$ and $V_{-+}$
are the \emph{parts} of $V_2$. A key tool will be:
\begin{lemma}\label{lem: pos neg perm}
  Suppose $a \in \{1, 2\}$ and $b = 3 - a$.  Under the
  $B_{-t}$-action on $R^2$ we have:
  \begin{equation}\label{eq:act}
  \sigma_a \cdot V_a \subset V_a \qquad (\sigma_a \sigma_b) \cdot V_b \subset
  V_a \qquad \Omega \cdot V_a = V_b \qquad \Omega^{-1} \cdot  V_a = V_b.
  \end{equation}
  If $V_{\alpha , \beta}$ is a part of $V_a$, then
  $\sigma_a \cdot V_{\alpha , \beta} \subset V_{\alpha , \beta}$.
  Also, for all $\alpha, \beta$ we have
  $\Omega \cdot V_{\alpha, \beta} = V_{\beta, -\alpha}$.  Finally,
  when $V_{\alpha, \, \beta}$ is a part of $V_b$:
  \begin{equation}
    \label{eq:sig ab action}
    (\sigma_a \sigma_b) \cdot V_{\alpha, \beta}
    \subset V_{(-1)^a \alpha, (-1)^b \beta}.
  \end{equation}
\end{lemma}
\begin{proof}
  Because
  \begin{equation}
    \label{eq:burauvals}
  \begin{gathered}
  B_{-t}(\sigma_1) = \matr{t}{1}{0}{1}  \qquad
  B_{-t}(\sigma_2) = \matr{1}{0}{-t}{t} \qquad
  B_{-t}(\Omega) =  \matr{0}{t}{-t^2}{0} \\
  B_{-t}(\sigma_1 \sigma_2) = \matr{0}{t}{-t}{t} \qquad
  B_{-t}(\sigma_2 \sigma_1) = \matr{t}{1}{-t^2}{0}
  \end{gathered}
  \end{equation}
  it follows that
  \begin{equation}
    \label{eq:Vactions}
    \begin{alignedat}{3}
    \sigma_1 \cdot V_{++} &\subset V_{++} \qquad\qquad &
    (\sigma_2 \sigma_1) \cdot V_{++} &\subset V_{+-} \qquad\qquad &
    \Omega \cdot V_{++} &= V_{+-} \\
    \sigma_1 \cdot V_{--} &\subset V_{--} &
    (\sigma_2 \sigma_1) \cdot V_{--} &\subset V_{-+} &
    \Omega \cdot V_{--} &= V_{-+} \\
    \sigma_2 \cdot V_{+-} &\subset V_{+-} &
    (\sigma_1 \sigma_2) \cdot V_{+-} &\subset V_{--} &
    \Omega \cdot V_{+-} &= V_{--} \\
    \sigma_2 \cdot V_{-+} &\subset V_{-+} &
    (\sigma_1 \sigma_2) \cdot V_{-+} &\subset V_{++} &
    \Omega \cdot V_{-+} &= V_{++}
    \end{alignedat}
  \end{equation}
  which then combine to give the claims in the lemma.
\end{proof}

\begin{proof}[Proof of Theorem~\ref{thm:possign}]
The key claim will be:
\begin{claim}\label{claim:Vi in Vj}
  Suppose $w \in \Braid{3}$ has canonical form
  $\Omega^k w_1 w_2 \cdots w_n$ and $w_1$ starts with $\sigma_i$ and
  $w_n$ ends with $\sigma_j$.  Then $w \cdot V_{j} \subset V_i$ if $k$
  is even and $w \cdot V_{j} \subset V_{3-i}$ if $k$ is odd.
\end{claim}
Since $\Omega$ interchanges the two $V_l$ by (\ref{eq:act}), we can
easily reduce to the case $k = 0$ in the canonical form of $w$.  To prove
the claim when $k = 0$, we induct on $n$ with the base case of $n = 1$
following immediately from (\ref{eq:act}).  For the general case,
consider $w = w_1 \cdot (w_2 w_3 \cdots w_n)$ and apply induction to
$w' = w_2 w_3 \cdots w_n$.  Since the claim holds for $w'$ and $w_1$
ends with the same $\sigma_\ell$ that $w_2$ starts with, the claim
follows for $w$ by (\ref{eq:act}).  To prove the theorem, note that
$\smallvec{1}{0}$ and $\smallvec{0}{1}$ are in $V_1 \cap V_2$.  Hence,
if $w_1$ starts with $\sigma_i$, by the claim both
$B_{-t}(w) \cdot \smallvec{1}{0}$ and
$B_{-t}(w) \cdot \smallvec{0}{1}$ are in $V_i$.  As those are the two
columns of $B_{-t}(w)$, we have shown that all entries of $B_{-t}(w)$
are definite.
\end{proof}

It is possible to use Lemma~\ref{lem: pos neg perm} to refine
Theorem~\ref{thm:possign} to determine signs of the entries of
$B_{-t}(w)$ from $w_1$, $w_n$, the value of $k$ modulo 4, and the
parities of the number of times $w_\ell$ is $\sigma_1 \sigma_2$ and
$\sigma_2 \sigma_1$ respectively.  The starting point is this
refinement of Claim~\ref{claim:Vi in Vj}:

\begin{lemma}
  \label{lem: sign details}
  Suppose $w \in \Braid{3}$ is in canonical form with $k = 0$ and
  $w_1$ starts with $\sigma_i$ and $w_n$ ends with $\sigma_j$.  Let
  $e$ and $f$ be the number of times that $w_\ell$ is
  $\sigma_1 \sigma_2$ and $\sigma_2 \sigma_1$ respectively.  If
  $V_{\alpha, \beta}$ is a part of $V_j$, then
  $w \cdot V_{\alpha, \beta} \subset V_{(-1)^e \alpha, \, (-1)^f
    \beta}$, with the latter one of the parts of $V_i$.
\end{lemma}

\begin{proof}
  The rough idea here is that (\ref{eq:Vactions}) shows that when we
  apply a single $w_k$ starting with $\sigma_c$ and ending
  with $\sigma_d$ to a part $V_{\gamma , \, \delta}$ of $V_d$, we get
  a subset of a part of $V_c$ with $\sigma_1 \sigma_2$ being the only
  generator that changes the first sign $\gamma$ and
  $\sigma_2 \sigma_1$ the only generator that changes the second sign
  $\delta$.  Formally, we induct on the length of $w$.  If $n = 1$,
  the claim follows from Lemma~\ref{lem: pos neg perm}, and is perhaps
  most easily checked by consulting (\ref{eq:Vactions}).  For $n > 1$,
  write $w = w_1 \cdot w'$ with $w' = w_2 \cdots w_n$ and set $e'$ and
  $f'$ to be the counts of $\sigma_1 \sigma_2$ and $\sigma_2 \sigma_1$
  for $w'$.

  If $w_1 = \sigma_a$, then $e = e'$ and $f = f'$.  Now $w_2$ must
  start with $\sigma_a$, and hence by induction we have
  $w' \cdot V_{\alpha, \beta}$ is contained in the part
  $V_{(-1)^{e'} \alpha, \, (-1)^{f'} \beta}$ of $V_a$. As $w_1$ sends
  each part of $V_a$ into itself, we learn
  $w \cdot V_{\alpha, \beta} \subset V_{(-1)^{e} \alpha, \,
    (-1)^{f} \beta}$ as needed.

  If $w_1 = \sigma_1 \sigma_2$, then $w_2$ must start with $\sigma_2$.
  By induction, as $e' = e - 1$ and $f' = f$, we have
  $w' \cdot V_{\alpha, \beta} \subset V_{(-1)^{e-1} \alpha, \, (-1)^f
    \beta}$, where the latter is a part of $V_2$.  The needed claim for
  $w \cdot V_{\alpha, \beta}$ now follows from (\ref{eq:sig ab
    action}).  The case of $w_1 = \sigma_2 \sigma_1$ is similar.
\end{proof}

\begin{corollary}
  Suppose $w \in \Braid{3}$ is in canonical form where $w_1$ starts
  with $\sigma_i$ and $w_n$ ends with $\sigma_j$. Let $e$ and $f$ be
  the number of times that $w_\ell$ is $\sigma_1 \sigma_2$ and
  $\sigma_2 \sigma_1$ respectively.  Then the signs of the entries of
  $B_{-t}(w)$ are given by:
  \begin{equation}\label{eq: sign details}
    \twobytwomatrix{0}{1}{-1}{0}^k
    \twobytwomatrix{(-1)^e}{(-1)^{e + j + 1}}{(-1)^{f + j + 1}}{ (-1)^f}
  \end{equation}
  or equivalently:
  \begin{equation}\label{eq: sign details alt}
    (-1)^e \twobytwomatrix{0}{1}{-1}{0}^k
    \twobytwomatrix{1}{(-1)^{j + 1}}{(-1)^{i + 1}}{ (-1)^{i + j}}.
  \end{equation}
\end{corollary}

\begin{proof}
  To begin, using that
  $\Omega \cdot V_{\alpha, \beta} = V_{\beta, -\alpha}$ from
  Lemma~\ref{lem: pos neg perm}, we can easily reduce to the case that
  $k = 0$.  Also, to see that (\ref{eq: sign details alt}) follows from
  (\ref{eq: sign details}), we argue that $(-1)^f = (-1)^{e + i + j}$.
  This is because if we view $w$ as a word in $\sigma_1$ and
  $\sigma_2$ only, then the $\sigma_1 \sigma_2$ and
  $\sigma_2 \sigma_1$ are the places where the word switches between
  the two $\sigma_i$.  Hence, if $i = j$, then $e = f$, whereas if
  $i \neq j$ then $f = e \pm 1$, giving the claimed
  $(-1)^f = (-1)^{e + i + j}$.

  Now suppose $j = 1$.  Then $\smallvec{1}{0}$ and
  $\smallvec{0}{1}$ are in $V_{++}$, which is a part of $V_1$, and so
  $w \cdot \smallvec{1}{0}$ and $w \cdot \smallvec{0}{1}$ are in
  $V_{(-1)^e, \, (-1)^f}$ by Lemma~\ref{lem: sign details}.  As these
  are the columns of $B_{-t}(w)$, the signs in this case are
  $\mysmallmatrix{(-1)^e}{(-1)^e}{(-1)^f}{(-1)^f}$ verifying (\ref{eq:
    sign details}) in this case.

  Suppose instead that $j = 2$.  Then $\smallvec{1}{0}$ is in $V_{+-}$
  and $\smallvec{0}{1}$ is in $V_{-+}$, both of which are parts of
  $V_2$.  Applying Lemma~\ref{lem: sign details} tells us that
  $B_{-t}(w)$ has signs
  $\mysmallmatrix{(-1)^e}{(-1)^{e+1}}{(-1)^{f+1}}{(-1)^f}$, completing
    the proof.
\end{proof}

We will need the following to prove Theorem ~\ref{thm:goodsign}.

\begin{lemma}\label{lem: strong form}
  Suppose $w \in \Braid{3}$ is not conjugate to a braid where
  $n \leq 1$ in its canonical form.  Then $w$ is conjugate a braid
  $w'$ whose canonical form is $\Omega^k w_1 w_2 \cdots w_{n-1} w_n$
  where $n \geq 2$ and $w_n$ ends with $\sigma_1$; when $k$ is even,
  $w_1$ starts with $\sigma_1$ and otherwise it starts with
  $\sigma_2$.  Moreover, if $w$ is positive, so is $w'$. 
\end{lemma}

\begin{proof}
Given two words in canonical form:
\[
  w = \Omega^k  w_1 w_2 \cdots w_{n-1} w_n \mtext{and}
  v = \Omega^\ell v_1 v_2 \cdots v_{m - 1} v_m
\]
we say that $w$ is \emph{cooler} than $v$ if either $k > \ell$ or
$k = \ell$ and $n < m$.

Two braids conjugate to $w$ are its \emph{cycling}
\[
  c =  \Omega^k (\Omega^{-k} w_n \Omega^{k}) w_1 w_2 \cdots w_{n-1}
\]
and its \emph{decycling}
\[
  d = \Omega^k w_2 w_3 \cdots w_n (\Omega^{k} w_1 \Omega^{-k}).
\]
Note here that $\Omega^{k} w_i \Omega^{-k}$ is $w_i$ when $k$
is even or $w_i$ with $\sigma_1$ and $\sigma_2$ interchanged when
$k$ is odd.  The word $w$ is \emph{summit} when neither of the
canonical forms of $c$ and $d$ are cooler than $w$ itself. (These
notions were introduced to solve the conjugacy problem in braid
groups, see \cite[Section~4]{El-RifaiMorton1994}.)

Any $w$ is conjugate to a summit $w'$ simply by repeatedly cycling and
decycling until no more improvement is possible. (To see this
terminates, we look at the conjugacy invariant
$\det B_t(w) = (-t)^{\# w}$ from (\ref{eq: B det}).  This forces
$\#w' = \#w$, and if $k'$ and $n'$ correspond to $w'$, we have
$n' + 3 k' \leq \#w' \leq 2n' + 3 k'$.  As $w'$ is cooler than $w$, we
know $k' \geq k$ which gives $n' + 3k \leq \#w'$. Thus,
$n' \leq \#w' - 3k = \#w - 3k$, giving finitely many possibilities for
$w'$.)  Any element of $\Braid{3}$ that can be written as a positive
word in the $\sigma_i$ has canonical form where $k \geq 0$ by
\cite[Theorem~2.9]{El-RifaiMorton1994}; and also (de)cycling a
positive word yields a positive word; therefore $w'$ is positive if
$w$ is, as required by the lemma.

Moreover, by conjugating by $\Omega$ (which interchanges the
$\sigma_i$), we can assume that $w'_{n'}$ is either $\sigma_1$ or
$\sigma_2 \sigma_1$.  By the hypotheses of the lemma, we also must
have $n' > 1$.  Finally, we analyze the first letter of $w'_1$.  If
$k'$ is even and $w'_1$ was either $\sigma_2$ or $\sigma_2 \sigma_1$,
then the cycling of $w'$ would be cooler than it is (here is where we
need $n' > 1$ as otherwise the cycling of $w'$ is equal to itself).  So
$w'_1$ starts with $\sigma_1$ when $k'$ is even.  Similarly if $k'$ is
odd, then $w'_1$ starts with $\sigma_2$ as otherwise its cycling would
be cooler.  Thus we have found the promised $w'$.
\end{proof}

We will need the following corollary of Lemma~\ref{lem: strong form}
when we get to Section~\ref{sec: root-free}.

\begin{corollary}\label{cor: norm with power(s)}
  Any positive $w \in \Braid{3}$ is conjugate to one of
  \begin{enumerate}
  \item \label{enum: one sig}
    $\Omega^k \sigma_1^a$ where $k \geq 0$ and $a \geq 0$.
  \item \label{enum: sig 1 sig 2}
    $\Omega^k \sigma_1 \sigma_2$ where $k \geq 0$.
  \item \label{enum: l even}
    $\Omega^k \sigma_1^{a_1} \sigma_2^{a_2} \cdots \sigma_1^{a_{\ell - 1}}\sigma_2^{a_\ell}$
  where $k \geq 0$, $\ell \geq 2$ is even, and all $a_i \geq 2$.
\item \label{enum: l odd}
  $\Omega^k \sigma_1^{a_1} \sigma_2^{a_2} \cdots  \sigma_2^{a_{\ell - 1}} \sigma_1^{a_{\ell}}$
  where $k \geq 1$ is odd, $\ell \geq 3$ is odd, and all $a_i \geq 2$.
\end{enumerate}
\end{corollary}

\begin{proof}
Suppose $w$ is conjugate to a braid $w'$ whose canonical form is
$\Omega^k$ or $\Omega^k w_1$. As $\#w = \#w'$, it follows that
$k \geq 0$, and hence we are in case (\ref{enum: one sig}) when
$w' = \Omega^k$.  If $w' = \Omega^k w_1$, we can conjugate by $\Omega$
if necessary to ensure $w_1$ starts with $\sigma_1$ and hence we are
either in case (\ref{enum: one sig}) or (\ref{enum: sig 1 sig 2}).

If instead $w$ is not conjugate to a braid whose canonical form has
$n \leq 1$, then by Lemma~\ref{lem: strong form}, the braid $w$ is
conjugate to a positive $w' = \Omega^k w_1 w_2 \cdots w_{n-1} w_n$
where $n \geq 2$ and $w_n$ ends with $\sigma_1$.  If
$w' = \Omega^k \sigma_1^a$, we are in case (\ref{enum: one sig}), so
we assume $\sigma_2$ also appears in $w_1 w_2 \cdots w_n$.

Because of the allowed transitions shown in Figure~\ref{fig:
  canonical}, every time $\sigma_i$ appears in
$w_1 w_2 \cdots w_{n-1} w_n$ it does so to a proper power, with the
possible exception of the first and last letters (e.g.~if
$w_n = \sigma_2 \sigma_1$).  When $k$ is even, Lemma~\ref{lem: strong
  form} gives that $w_1$ starts with $\sigma_1$, so we move all
$\sigma_1$'s at the end of $w_1 w_2 \cdots w_{n-1} w_n$ around to the
front of $w'$ and commute them through the $\Omega^k$ to end up in
(\ref{enum: l even}).  When $k$ is odd, we know $w_1$ starts with
$\sigma_2$, so we pull all the initial $\sigma_2$'s of
$w_1 w_2 \cdots w_{n-1} w_n$ through the $\Omega^k$ and put the
resulting $\sigma_1$'s at the end of $w'$ to end up in case
(\ref{enum: l odd}) or, if there are no $\sigma_2$' s left, in case
(\ref{enum: one sig}).
\end{proof}

\begin{remark}
  \label{rem: sigma pows}
  If a positive $w \in \Braid{3}$ is conjugate to $\sigma_1^a$ for $a
  \geq 1$, then $w$ is in fact either $\sigma_1^a$ or $\sigma_2^a$.  In
  the language of \cite{El-RifaiMorton1994}, one sees that
  $\sigma_1^a$ is in its super-summit set since it is unchanged by
  cycling and decycling.  It follows that the canonical form $\Omega^k
  w_1 w_2 \cdots w_{n-1} w_n$ for $w$ must have $k = 0$ and $n \geq
  a$.  Since $\#w = a$, it follows that no $w_i$ is $\sigma_1
  \sigma_2$ or $\sigma_2 \sigma_1$, making the only valid canonical
  forms $\sigma_1^a$ and $\sigma_2^a$.
\end{remark}

\begin{proof}[Proof of Theorem~\ref{thm:goodsign}]
Let $w'$ be the conjugate braid to $w$ given by Lemma~\ref{lem: strong
  form}, whose canonical form is
$w' = \Omega^k w_1 w_2 \cdots w_{n-1} w_n$ where $n \geq 2$ and $w_n$
is ends with $\sigma_1$.  If $k$ is even, Lemma~\ref{lem: strong form}
further specifies that $w_1$ starts with $\sigma_1$.  Since
$w_1 \cdots w_n$ both starts and ends with $\sigma_1$, equation
(\ref{eq: sign details alt}) gives that all entries of $B_{-t}(w)$
have the same sign.  If instead $k$ is odd, then $w_1$ must start with
$\sigma_2$ per Lemma~\ref{lem: strong form}, and again equation
(\ref{eq: sign details alt}) gives that all entries of $B_{-t}(w)$
have the same sign.
\end{proof}

\begin{proof}[Proof of Corollary~\ref{cor:tracedef}]
  Since the trace of a matrix whose entries all have the same sign is
  definite, Theorem~\ref{thm:goodsign} reduces us to the case when $w$
  is conjugate to a canonical form $\Omega^k$ or $\Omega^k w_1$.
  Conjugating by $\Omega$, we can further assume $w_1$ is either
  $\sigma_1$ or $\sigma_2 \sigma_1$.

  When $k$ is even, we have
  $B_{-t}(\Omega^k) = (-t)^{3\ell} \mysmallmatrix{1}{0}{0}{1}$ where
  $k = 2 \ell$ by (\ref{eq:burauvals}).  We now calculate that
  \[
  \tr\big(B_{-t}(\Omega^k)\big) = 2 (-t)^{3\ell} \quad
  \tr\big(B_{-t}(\Omega^k \sigma_1)\big) = (-t)^{3\ell}(t + 1) \quad
  \tr\big(B_{-t}(\Omega^k \sigma_1 \sigma_2)\big) = (-t)^{3\ell} t
  \]
  all of which are definite.

  When instead $k$ is odd, say $k = 2 \ell + 1$, we have
  $B_{-t}(\Omega^k) = (-t)^{3\ell} \mysmallmatrix{0}{t}{-t^2}{0}$ and
  \[
  \tr\big(B_{-t}(\Omega^k)\big) = 0 \quad
  \tr\big(B_{-t}(\Omega^k \sigma_1)\big) = (-t)^{3\ell}(t + 1) \quad
  \tr\big(B_{-t}(\Omega^k \sigma_1 \sigma_2)\big) = (-t)^{3\ell} t
  \]
  which are again all definite, proving the corollary.
\end{proof}

\begin{remark} \label{R:nice-pres}
  One can make this story more obviously symmetric between $i = 1$ and
  $i = 2$ by setting $u^2 = -t$ and conjugating by
  $\mysmallmatrix{u}{0}{0}{1}$, which results in
  \[
    \sigma_1 \mapsto  \matr{u^2}{u}{0}{1} \quad \sigma_2 \mapsto \matr{1}{0}{-u}{u^2} \quad
   \sigma_1 \sigma_2 \mapsto  \matr{0}{u^3}{-u}{u^2} \quad \sigma_2 \sigma_1 \mapsto  \matr{u^2}{u}{-u^3}{0}
  \]
  as well as $\Omega \mapsto u^3 \mysmallmatrix{0}{1}{-1}{0}$.
  These are interchanged by conjugating by $E = \mysmallmatrix{0}{1}{-1}{0}$,
  and one also has $E V_i = E^{-1} V_i = V_{3 - i}$ where now $V_i$ is
  defined in terms of $u $ rather than $t$.
\end{remark}

\section{Properties of the spectral radius}
\label{sec: spec}

We begin with a simple observation:
\begin{lemma}
  \label{lem: obs}
  Suppose $\alpha \in \C^\times$ and $w \in \Braid{3}$.  If $\alpha$ is
  a root of $\Delta_\what$, then $1$ is an eigenvalue of
  $B_\alpha(w) \in \GL{2}{\C}$.  The converse is true provided
  $\alpha \neq e^{\pm 2 \pi i/3}$.  If $w$ is a positive braid, these
  claims hold for $\alpha = 0$ as well.
\end{lemma}

\begin{proof}
By (\ref{eq:alex}), in $\Z[t^{\pm 1}]$ we have
\begin{equation}\label{eq: rewritten}
  \det(B_t(w) - 1) = (t^2 + t + 1) \Delta_\what(t).
\end{equation}
Thus for $\alpha \in \C^\times$, if $\Delta_\what(\alpha) = 0$ then
$\det(B_\alpha(w) - 1) = 0$ and so $1$ is an eigenvalue of
$B_\alpha(w)$.  The partial converse also follows from (\ref{eq:
  rewritten}). Lastly, when $w$ is positive, the entries of $B_t(w)$
are polynomials and so (\ref{eq: rewritten}) can be evaluated at $t =
0$ as well.
\end{proof}

For any braid $w \in \Braid{3}$, we consider the function $\rho_w
\maps \C^\times \to [0, \infty)$ which sends $t \in \C^\times$ to the spectral
radius of the matrix $B_t(w)$; that is, $\rho_w(t)$ is the largest
absolute value of an eigenvalue of $B_t(w)$.  For a positive braid,
the entries of $B_t(w)$ are polynomials in $t$ and hence $\rho_w(0)$
makes sense as well, giving $\rho_w \maps \C \to [0, \infty)$.

\begin{lemma}\label{lem: spec basics}
  For a positive $w \in \Braid{3}$, every root of $\Delta_\what$ in
  $\Dbar$ occurs at a point where $\rho_w = 1$.  That is,
  $(Z_w \cap \Dbar) \subset R_w$ in the language of Section~\ref{sec:
    intro}.  Moreover, the function $\rho_w(t)$ is continuous and
  subharmonic on $\C$; in fact, $\log(\rho_w(t))$ is subharmonic.
\end{lemma}
Taking the domain of $\rho_w$ to be $\C^\times$ rather than $\C$, the
same holds for arbitrary braids.

\begin{proof}[Proof of Lemma~\ref{lem: spec basics}]
First, we show $(Z_w \cap \Dbar) \subset R_w$.  By
Lemma~\ref{lem: obs}, when $t \in \C$ is a root of $\Delta_\what$ then
1 must be an eigenvalue of $B_w(t)$.  When $1$ is an eigenvalue, the
only other eigenvalue will be $\det B_w(t) = (-t)^{\#w}$ where
$\#w \geq 0$.  When the root $t$ is in $\Dbar$, we have
$\abs{t^{\#w}} \leq 1$ and hence $\rho_w(t) = 1$ as desired.

That $\rho_w(t)$ and $\log(\rho_w(t))$ are subharmonic follows directly from Gelfand's formula 
$\rho_w(t) = \lim_n \Vert B_t(w)^n \Vert^{1/n}$ (see also
Theorems~1 and~$1'$ of \cite{Vesentini1970} and 
Lemma~\ref{lem: weak sub} for a more general statement)
so it remains to show that $\rho_w(t)$ is continuous.  
The characteristic polynomial of $B_t(w)$ is
\begin{equation}
  \label{eq: char poly}
  \lambda^2 - \tr\big(B_t(w)\big) \lambda + \det\big(B_t(w)\big) =
  \lambda^2 - \tr\big(B_t(w)\big) \lambda + (-t)^{\#w}.
\end{equation}
If we set $p(t) = \tr\big(B_t(w)\big)$, which is just a
polynomial in $t$, then
\begin{equation}
  \label{eq: eigens}
  \lambda = \frac{p(t) \pm \sqrt{p(t)^2 - 4(-t)^{\#w}}}{2}.
\end{equation}
If the discriminant $p(t)^2 - 4(-t)^{\#w}$ is identically $0$,
then $\rho_w(t) = \abs{\lambda} = \abs{p(t)}/2 = \abs{t}^{\#w/2}$
which is continuous.  So we assume the
discriminant is a nonzero polynomial and hence the set
\begin{equation}
  \label{eq:disc}
  d(w) = \setdef{t}{p(t)^2 - 4(-t)^{\#w} = 0}
\end{equation}
is finite.

For a small enough open set $U$ disjoint from $d(w)$, we can choose
a branch cut for the square root function so that
$g(t) = \sqrt{p(t)^2 - 4(-t)^{\#w}}$ is well-defined and
holomorphic on $U$.  In particular:
\[
  \rho_w(t) = \frac{1}{2}\max\Big( \abs{p(t) + g(t)}, \
  \abs{p(t) - g(t)} \Big)
\]
is continuous on $U$ as needed.  Continuity of $\rho_w(t)$ at the
points in $d(w)$ follows since any $t_i \to t_\infty \in d(w)$ will
have all choices of corresponding $\lambda_i$ converging to
$p(t_\infty)/2$.
\end{proof}

\begin{remark}
  Note that the above proof shows that $\rho_w(t) = \abs{t}^{\#w/2}$
  for $t \in d(w)$.  Thus on $\D$ where $\abs{t} < 1$, the set $d(w)$
  is disjoint from $R_w$.
\end{remark}

\section{Behavior on the unit circle}
\label{sec: on circle}

We now turn to studying the spectral radius function $\rho_w(t)$ on
the unit circle, starting with:
\begin{lemma}
  \label{lem:compact}
  Let $w \in \Braid{3}$.  For all $t$ on the unit circle, the spectral
  radius $\rho_w(t) \geq 1$.  If $w \in \cA_R$, then $\rho_w(t) = 1$.
  Moreover, if the mapping class induced by $w$ is pseudo-Anosov, then
  $\rho_w(-1) > 1$ is its stretch factor and also the maximum value of
  $\rho_w$ on the unit circle.  If $w$ is not pseudo-Anosov, then
  $\rho_w(t)$ is identically $1$ on the unit circle.
\end{lemma}

\begin{proof}
When $t \in \C$ with $\abs{t} = 1$, each
$\abs{\det\big(B_t(\sigma_i)\big)} = 1$, and hence $\det(B_t(w))$ is a unit
complex number as well.  As $\det B_t(w)$ is also the product of the
eigenvalues of $B_t(w)$, it follows that at least one eigenvalue is
outside the \emph{open} unit disk and hence $\rho(B_t(w)) \geq 1$.  From
Lemma~\ref{lem: sig on circle}, when $|\theta| < \frac{2 \pi}{3}$, the
matrix $B_t(w)$ is conjugate into $\mathrm{U}(2)$; in particular, the
spectral radius $\rho(B_t(w))$ is $1$. The case of
$\theta = \pm \frac{2 \pi}{3}$ then follows by continuity of
$\rho(B_t(w))$ as a function of $t$.

Let $\phi_w$ be the mapping class that the braid $w$ induces on the
3-punctured disc $D$. Let $h(\phi_w)$ be the topological entropy of
$\phi_w$.  By equation (1-1) of \cite{BandBoyland2007}, we have
\begin{equation}
  \label{eq: entopy bound}
  e^{h(\phi_w)} \geq \sup\setdef{\rho_w(t)}{\mbox{$|t| = 1$}}.
\end{equation}
We know that $\phi_w$ is either finite order, reducible, or
pseudo-Anosov.  In the first case, $h(\phi_w) = 0$ and so (\ref{eq:
  entopy bound}) means that $\rho_w(t) \equiv 1$ on the unit circle.
For reducible $\phi_w$, the only possible reducing system is a single
curve surrounding a pair of punctures which is invariant under
$\phi_w$. The two resulting pieces are basically 3-punctured spheres
whose mapping class group is finite; hence $h(\phi_w) = 0$ in this
case and again $\rho_w(t) \equiv 1$.  Lastly, if $\phi_w$ is
pseudo-Anosov, the stretch factor $\exp(h(\phi_w))$ is
$\rho_w(-1) > 1$ by Proposition~\ref{prop: braid
  trichotomy}(\ref{item: pA}). So we have equality in (\ref{eq: entopy
  bound}), with $\rho_w$ achieving its max on the unit circle at $-1$.
\end{proof}

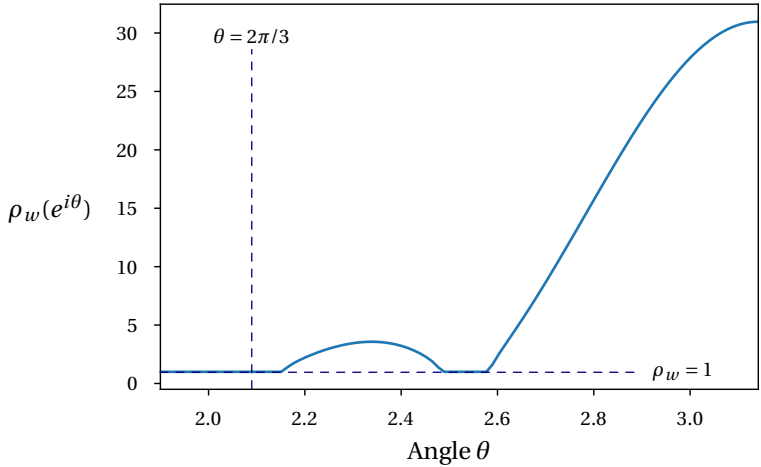
\begin{figure}
  \centering
  \begin{tikzoverlay}[width=0.7\textwidth]{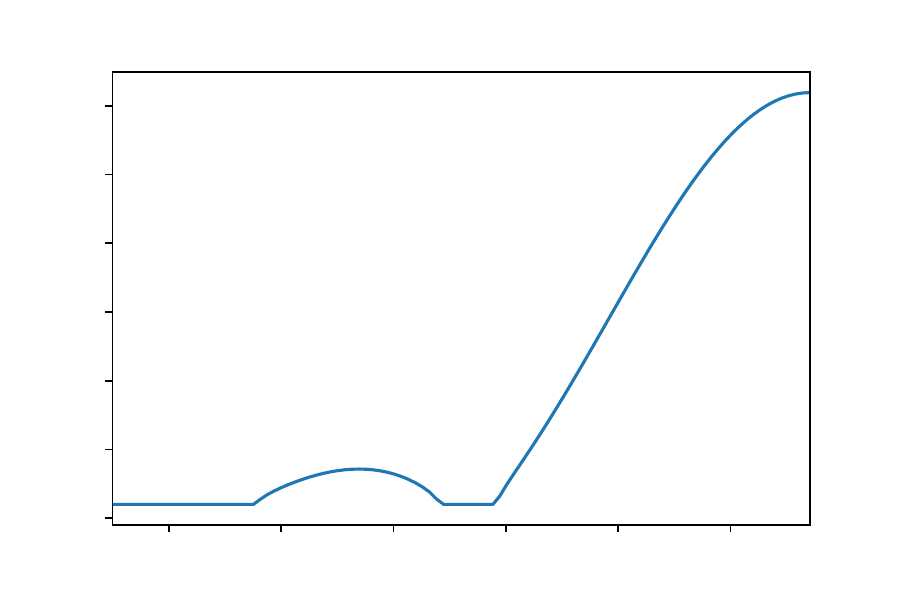}
  \begin{scope}[font=\scriptsize]
    \draw (18.741983, 6.712963) node[below] {$2.0$};
    \draw (31.225948, 6.712963) node[below] {$2.2$};
    \draw (43.709914, 6.712963) node[below] {$2.4$};
    \draw (56.193880, 6.712963) node[below] {$2.6$};
    \draw (68.677845, 6.712963) node[below] {$2.8$};
    \draw (81.161811, 6.712963) node[below] {$3.0$};
    \draw (10.879630, 9.094316) node[left] {$0$};
    \draw (10.879630, 16.728796) node[left] {$5$};
    \draw (10.879630, 24.363277) node[left] {$10$};
    \draw (10.879630, 31.997757) node[left] {$15$};
    \draw (10.879630, 39.632237) node[left] {$20$};
    \draw (10.879630, 47.266718) node[left] {$25$};
    \draw (10.879630, 54.901198) node[left] {$30$};
    \draw (50, 3) node[below, font=\footnotesize] {Angle $\theta$};
    \draw (5, 32) node[left, font=\footnotesize] {$\rho_w(e^{i \theta})$};
    \begin{scope}[shift={(-106.09767338, 9.09431606)},
      xscale=62.41982809, yscale=1.52689606,
      dashed, color=blue!50!black, line cap=round, line width=0.5pt]]
      \coordinate (pi on three) at (1.0471975511965976, -0.2);
      \draw ($2*(pi on three)$) -- +(0, 29)
            node[above=-0.5, color=black] {$\theta = 2\pi/3$};
      \draw (1.9, 1) -- (2.9, 1) node[right, color=black] {$\rho_w = 1$};
    \end{scope}
  \end{scope}
\end{tikzoverlay}
  \caption{The spectral radius $\rho_w$ on part of the unit circle for
    $w = \sigma_1 \sigma_2^2 \sigma_1^3 \sigma_2^3 \sigma_1^9
    \sigma_2^2$.  This example shows that $\rho_w$ can be
    non-monotonic on the circle and that $R_w \cap \partial \Dbar$ can
    be disconnected. In fact, $R_w$ itself is
    disconnected.}
  \label{fig: spec on circle}
\end{figure}

\begin{remark}
  \label{rem: unique max?}
  In the pseudo-Anosov case, we expect that $-1$ is the unique point
  on the unit circle where $\rho_w$ is maximal.  By Theorem~1.1 of
  \cite{BandBoyland2007}, one has $\rho_w(\zeta) < \rho_w(-1)$ for all
  roots of unity $\zeta \neq -1$.  On page 1,347 of
  \cite{BandBoyland2007}, the authors announce that they proved this
  for all $\zeta \neq -1$ on the unit circle by different methods, but
  no subsequent paper has appeared.
\end{remark}

\begin{remark}
  There are examples of positive braids whose spectral radius is
  larger than $1$ for any $t \in \cA_L$; for
  example, $w = \sigma_1 \sigma_2^5$ and $w = \sigma_1^3
  \sigma_2^3$. However, there are positive braids where the region
  $\setdef{\theta}{\rho_w(e^{i \theta}) = 1 }$ is strictly larger than
  this interval, e.g.~$w = \sigma_1^2 \sigma_2^3$.  Even for positive
  braids, the behavior of $\rho_w$ on the unit circle can be
  complicated, as shown in Figure~\ref{fig: spec on circle}.
\end{remark}

\subsection{Signature function}
\label{sec: sig fn}

An important invariant of an oriented link $L$ in $S^3$ is its
\emph{signature function}, which is a step function
$\sigma_L \maps S^1 \to \Z$ with discontinuities only at roots of
$\Delta_L$. Here, $\sigma_L$ is also called the Levine-Tristram or
equivariant signature, and its value $\sigma_L(\omega)$ is called the
$\omega$-signature of $L$ in \cite{GambaudoGhys2005}.  We now list
some key properties of $\sigma_L$; you can find details and further
references in \cite{Conway2021}.  The signature function is defined
using a Seifert matrix $V_L \in M_{n}(\Z)$ for $L$ and then
considering
\[
  W_L(t) = (1 - t)V_L+ (1 - t^{-1}) V_L^t = (1 - t)\big(V_L - t^{-1}V_L^t\big).
\]
Specifically, for $\omega \in S^1$, we define $\sigma_L(\omega)$ to be
the signature of the Hermitian form given by the matrix $W_L(\omega)$.
It turns out that $\sigma_L(1) = 0$ and
$\sigma_L(\omega) = \sigma_L(\omegabar)$ for all $\omega \in S^1$.
The value $\sigma_L(-1)$ is the Murasugi signature, which is often
simply called the \emph{signature} of $L$ and denoted $\sigma(L)$.

Now $\Delta_L$, which is only defined up to multiplication by a unit in
$\Z[t^{\pm 1}]$, is equal to $\det(V_L - t V^t_L)$ or, equivalently, to
$\det(V_L - t^{-1} V^t_L)$. If $\Delta_L \neq 0$, then
any discontinuity of $\sigma_L$ must be at a root of $\Delta_L$ by
\cite[Theorem~2.3]{GilmerLivingston2016}; roughly, a jump in
$\sigma_L$ at $\omega_0$ forces $W_L(\omega_0)$ to have 0 as an
eigenvalue and hence $\Delta_L(\omega_0) = 0$.  If $\omega_0$ is a
root of $\Delta_L$ of multiplicity $d$, then $\sigma_L(\omega)$
changes by at most $2 d$ as $\omega$ passes through $\omega_0$ by the
combination of Theorems~2.1 and~2.3 and Lemma~5.1 of
\cite{GilmerLivingston2016}.  Moreover, for a simple root
$\omega_0 \neq \pm 1$ of $\Delta_L$, the function $\sigma_L$ always
changes by $\pm 2$ \cite[Theorem~2.4]{GilmerLivingston2016}.
Consequently, we view $\sigma_L(\omega)$ as a signed count of roots of
$\Delta_L$ on the counter-clockwise arc from $1$ to $\omega$.  In
particular, one has the following, which goes back to Riley prior to
1983 for knots \cite[Equation~7]{Stoimenow2005} and is Theorem~2.3 in
\cite{GilmerLivingston2016} for links:

\begin{proposition}
  \label{prop: roots on arc}
  Suppose an arc on the upper half of the unit circle has endpoints
  $\omega_1$ and $\omega_2$, where both $\Delta_L(\omega_i) \neq 0$.
  Then the number of roots of $\Delta_L$ in that arc, counted with
  multiplicity, is at least
  $\frac{1}{2}\abs{\sigma_L(\omega_1) - \sigma_L(\omega_2)}$.
  Moreover, if $\Delta_L \neq 0$, the number of roots on the unit
  circle is at least $\abs{\sigma_L(-1)}$.
\end{proposition}

When $L = \what$ for $w \in \Braid{3}$, Gambaudo and Ghys
\cite{GambaudoGhys2005} calculate $\sigma_L(e^{i \theta})$ with
$0 \leq \theta < \frac 2 3 \pi$ in terms of the eigenvalues of
$B_{e^{i \theta}}(w)$. In particular, their Corollary~4.4 gives:

\begin{proposition}[\cite{GambaudoGhys2005}]
  \label{prop: sig for 3 braid}
  Suppose $w \in \Braid{3}$ and
  $e^{i \theta}$ for $0 \leq \theta < \frac{2\pi}{3}$ is not a root of
  $\Delta_\what$. Then $\sigma_\what(e^{i \theta})$ differs from
  $-\frac \theta \pi (\# w)$ by at most 2.
\end{proposition}
We note that $e^{i \theta}$ is assumed to be a root of unity in
Corollary~4.4 of \cite{GambaudoGhys2005}, but the above version follows
from this because $\sigma_\what$ is constant away from the roots of
$\Delta_\what$. 
\begin{corollary}
  \label{cor: arc lower bound}
  Suppose a link $L$ is $\what$ for $w \in \Braid{3}$ and
  $0 \leq \theta_1 < \theta_2 \leq \frac{2 \pi}{3}$.  If $\Delta_L \neq
  0$, then $\Delta_L$ has at least
  \[
    \frac{\theta_2 - \theta_1}{2\pi} \abs{\#w} - 2
  \]
  roots in the arc of $S^1$ joining $e^{i \theta_1}$ to $e^{i
    \theta_2}$.
\end{corollary}

\begin{proof} 
First, consider the case where $\theta_2 < \frac{2\pi}{3}$ and neither
$\omega_k = e^{i \theta_k}$ is a root of $\Delta_L$.  By
Proposition~\ref{prop: sig for 3 braid}, we have $\sigma(\omega_k) = -
\frac{\theta_k}{\pi} (\#w) + e_k$, where $\abs{e_k} \leq 2$.  Thus
\begin{equation*}\frac{1}{2}\absm{\big}{\sigma(\theta_2) - \sigma(\theta_1)} =
   \frac{1}{2}\abs{ \frac{-\theta_2 + \theta_1}{\pi} \cdot \# w + e_2 -e_1}
  \geq \frac{\theta_2 - \theta_1}{2 \pi}\abs{\#w} - 2 
\end{equation*}
and so then the result follows from Proposition~\ref{prop: roots on arc}.

For the general case, since $\Delta_L \neq 0$ we can find arbitrarily
small $\epsilon$ so that $\theta'_1 = \theta_1 + \epsilon$ and
$\theta'_2 = \theta_2 - \epsilon$ are in the previous case.  So there
are at least $\frac{\theta_2 - \theta_1 - 2 \epsilon}{2 \pi} \abs{\#w}
- 2$ roots in the arc joining $e^{i \theta'_1}$ to $e^{i \theta'_2}$.
Sending $\epsilon \searrow 0$ then proves the corollary.  
\end{proof}

\begin{corollary}
  \label{cor: 2/3 of roots}
  Suppose a link $L = \what$ for a positive $w \in \Braid{3}$ that is
  not $\sigma_i^n$.  Then $\Delta_L$ has at least
  $r_L := \frac{2}{3}\big(\deg(\Delta_L) - 1\big)$ roots on the arc
  $\cA_R$.  More precisely, suppose $\mu_L$ is the number of
  components of $L$ and let $\lceil r_L \rceil_\odd$ and
  $\lceil r_L \rceil_\even$ denote the smallest odd and even integers,
  respectively, that are at least $r_L$.  If $\mu_L$ is odd, then
  the number of roots in $\cA_R$ is at least
  $\lceil r_L \rceil_\even$.  Conversely, if $\mu_L$ is even, the
  number of roots is at least $\lceil r_L \rceil_\odd$.
\end{corollary}

\begin{proof}
As $w$ is not $\sigma_i^n$, we have $\deg(\Delta_L) = \#w - 2$ from
Theorem~\ref{thm: burau pos}; in particular, $\Delta_L \neq 0$. First
suppose that $L$ is a knot, so that $\Delta_L(1) \neq 0$ as per
Section~\ref{sec: background}. Pick $\epsilon$ so that
$\omega_2 = e^{i(2\pi/3 - \epsilon)}$ is not a root of $\Delta_L$.
Since $\sigma_L(1) = 0$, Proposition~\ref{prop: roots on arc} shows
the number of roots on the open arc from $1$ to $\omega_2$ is at least
$\frac{1}{2} \abs{\sigma_L(\omega_2)}$, which is at least
$\left(\frac{1}{3} - \frac{\epsilon}{2 \pi}\right) \# w - 1$ by
Proposition~\ref{prop: sig for 3 braid}. sending $\epsilon \searrow 0$
shows that there are at least
$\left\lceil \frac{1}{3}\big(\deg(\Delta_L) - 1\big) \right\rceil$
roots on the open arc from $1$ to $\zeta_3$.  Doubling this gives the
promised bound for the number of roots in $\cA_R$.

Now suppose $L$ has $\mu_L > 1$ components. By Lemma~5.1 of
\cite{GilmerLivingston2016}, $\Delta_L$ has a root of order at least
$\mu_L - 1$ at $1$.  Let $\omega_1 = e^{i \epsilon}$ and take
$\omega_2$ as before, choosing $\epsilon$ so neither $\omega_i$ is a
root of $\Delta_L$ and so there is no root of $\Delta_L$ between $1$
and $\omega_1$.  Then $\sigma_L(\omega_1)$ is equal to the quantity
$\mathrm{jump}^+(1)$ from \cite{GilmerLivingston2016}, and hence
$\abs{\sigma_L(\omega_1)} \leq \mu_L - 1$ by Theorem~2.1 there.
Arguing as before, we get there are at least
$\left\lceil \frac{1}{3} \#w - \frac{\mu_L + 1}{2} \right\rceil$ roots
in the open arc from $1$ to $\zeta_3$.  Doubling this and including
the multiplicity of the root at $1$ gives that the number of roots in
$\cA_R$ is at least
\begin{equation*}
 c_L := 2 \left\lceil \frac{1}{3}\big( \deg(\Delta_L) + 2\big) - \frac{\mu_L
     + 1}{2} \right\rceil + \mu_L - 1.
\end{equation*}
Using that $x \leq \lceil x \rceil  < x + 1$, we calculate
$r_L \leq c_L < r_L + 2$. Combing this with the parity of $c_L$ being the
opposite of that of $\mu_L$ gives that $c_L = \lceil r_L \rceil_\even$
when $\mu_L$ is odd and $\lceil r_L \rceil_\odd$ when $\mu_L$ is even,
proving the corollary. 
\end{proof}

The bound in Corollary~\ref{cor: 2/3 of roots} appears to be quite
sharp. For the 2,500 positive braids in Figure~\ref{fig: pos many},
each of which gives a knot $K$, the number of roots of $\Delta_K$ on
the arc $\cA_R$ was always either equal to the lower bound given in
Corollary~\ref{cor: 2/3 of roots}, or 2 more than it; indeed, the
lower bound was sharp for 62.6\% of these knots.  We note that
Corollary~\ref{cor: 2/3 of roots} can also be proved by the same
combinatorial method used to establish Lemma~\ref{L:find-roots} below,
which does not use the signature function.

\subsection{The total number of roots on the unit circle}

As $\abs{\sigma_L(-1)}$ gives a lower bound on the total number of
roots on the unit circle, we now study its behaviour for a random
closed \3-braid.  Specifically, consider the random walk
$w_n := g_1 \cdots g_n$ on $\Braid{3}$ generated by the uniform
measure on $\{\sigma_1, \sigma_2\}$.  The main result of this section
is:
\begin{theorem} \label{T:clt} The signature
  $\abs{\sigma_{\widehat{w}_n}(-1)}$ follows a central limit theorem,
  with positive drift and positive variance.  Namely, if
  $\ell_0 = (5 - \sqrt{5})/4 \approx 0.690983$ there is $\nu_0 > 0$ such
  that for any $a < b$ we have
  \[
    \P\left(\frac{\abs{\sigma_{\widehat{w}_n}(-1)} - \ell_0 n}{\nu_0
        \sqrt{n}} \in [a, b] \right) \to \frac{1}{ \sqrt{2 \pi}}
    \int_a^b e^{-t^2/2} \dt.
  \]
\end{theorem}
The convergence in Theorem~\ref{T:clt} is illustrated in
Figure~\ref{fig: sig conv}. 

\begin{figure}
  \centering
  \begin{tikzpicture}[nmdstd]
  \begin{tikzoverlay*}[width=0.8\textwidth]{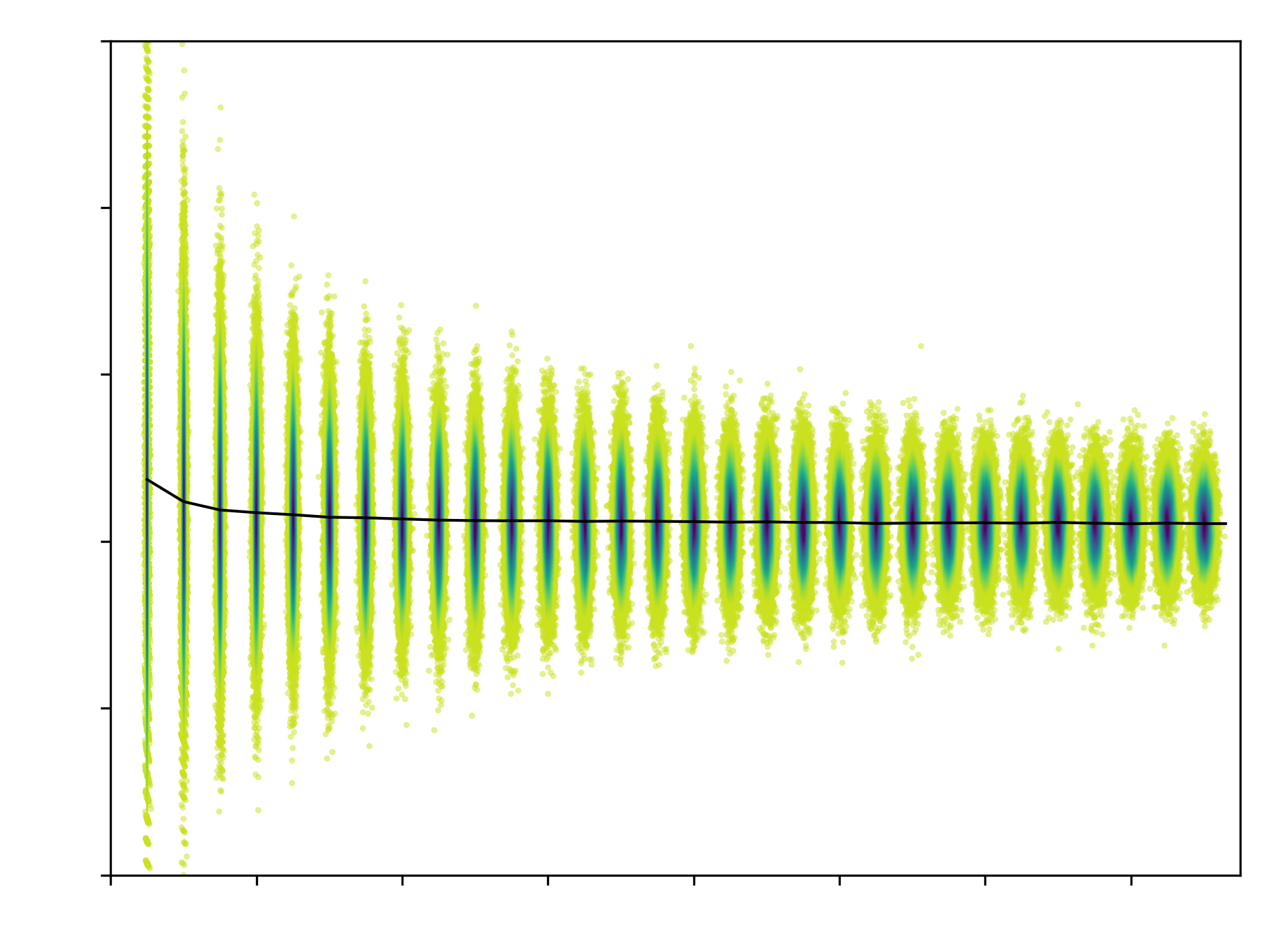}
    \draw (8.706597, 4.531250) node[below] {$0$};
    \draw (20.183972, 4.531250) node[below] {$2,000$};
    \draw (31.661346, 4.531250) node[below] {$4,000$};
    \draw (43.138721, 4.531250) node[below] {$6,000$};
    \draw (50, 0) node[below, font=\small] {$\#w = \deg(\Delta_\what)
      + 2$};
    \draw (54.616095, 4.531250) node[below] {$8,000$};
    \draw (66.093470, 4.531250) node[below] {$10,000$};
    \draw (77.570845, 4.531250) node[below] {$12,000$};
    \draw (89.048219, 4.531250) node[below] {$14,000$};
    \draw (7.187500, 6.050347) node[left] {$0.67$};
    \draw (7.187500, 19.199653) node[left] {$0.68$};
    \draw (7.187500, 32.348958) node[left] {$0.69$};
    \draw (0, 38) node[left, font=\small] {$\displaystyle
      \frac{\abs{\sigma_\what(-1)}}{\deg(\Delta_\what)}$};
    \draw (7.187500, 45.498264) node[left] {$0.70$};
    \draw (7.187500, 58.647569) node[left] {$0.71$};
    \draw (7.187500, 71.796875) node[left] {$0.72$};
  
    \begin{scope}[shift={(8.70659722, -874.95312500)},
      xscale=0.00573869, yscale=1314.93055556]
    \end{scope}
  \end{tikzoverlay*}
\end{tikzpicture}
  \caption{
    The normalized $\sigma_\what(-1)$ for 300,000 random
    positive 3-braids where $\what$ is a knot.  The braids were
    sampled in groups of 10,000 where $\#w \approx 500 \cdot k$ for
    $k = 1, \ldots, 30$.  The dark near-horizontal curve in the middle
    plots the mean of each group, which is about $0.691$ at the right
    side, consistent with Theorem~\ref{T:clt}.
  } 
  \label{fig: sig conv}
\end{figure}
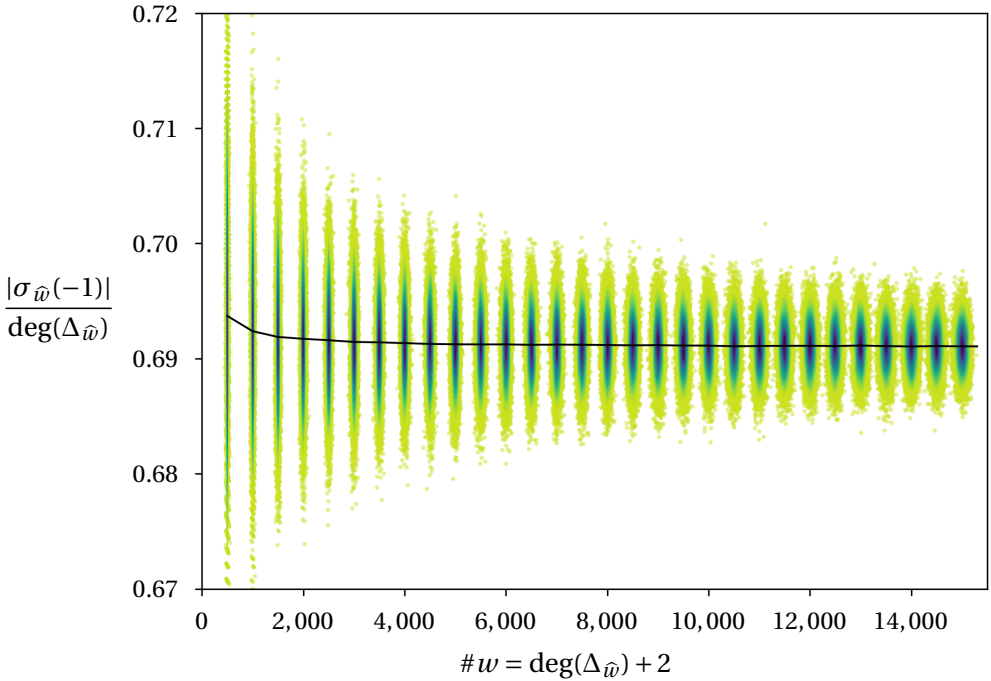


In order to prove this theorem, we need to introduce the Rademacher function. 

\subsection{The Rademacher function}

The matrices $a = \mysmallmatrix{0}{1}{-1}{0}$ and  $b =
\mysmallmatrix{1}{-1}{1}{0}$ give a free product decomposition
\[
  \PSL{2}{\Z} = \Z/2\Z * \Z/3\Z  = \spandef{a, b}{a^2 = b^3 = 1}.
\]
Each $g \in \PSL{2}{\Z}$ has a unique expression as a reduced word in
$a$, $b$, and $b^{-1}$ where no proper power of those elements
appears, for example $g = a b a b^{-1} a b^{-1}$; we refer to this
expression as the \emph{canonical word} of $g$.  Let $\#_b(g)$ be the number
of occurrences of $b$ in the canonical word for $g$, and similarly
for $\#_{b^{-1}}(g)$.  The \emph{Rademacher function}
$\Radem \maps \PSL{2}{\Z} \to \mathbb{Z}$ is
\[
  \Radem(g) := \#_b(g) - \#_{b^{-1}}(g),
\]
which is denoted $\mathfrak{Radem}$ in
\cite{GambaudoGhys2004}.   From \cite[page~1598]{GambaudoGhys2004}, we have:

\begin{lemma}
  \label{lem: quasi Radem}
  The Rademacher function $\Radem$ is a quasimorphism of defect $3$; in fact,
  \[
    \Radem(gh) - \Radem(g) - \Radem(h) \in \{-3, 0, 3 \}
    \mtext{for any $g, h \in \PSL{2}{\Z}$.}
  \]

\end{lemma}

Via the projection $\SL{2}{\Z} \to \PSL{2}{\Z}$, we also regard
$\Radem$ as a function on $\SL{2}{\Z}$.  The relevance of
this to our study of link signatures comes from
\cite{GambaudoGhys2005}:

\begin{lemma}
  \label{lem: sig from Radem}
  For each $w \in \Braid{3}$, we have
  \[
    \sigma_\what(-1) =  - \frac{1}{3} \Radem(B_{-1}(w)) - \frac{2}{3}
    \#w + e(w) \mtext{where $\abs{e(w)} \leq \displaystyle \frac{7}{3}$.}
  \]
\end{lemma}

\begin{proof}
The homogenization $\HRadem  := \lim_{n \to \infty} \frac{1}{n} \Radem(g^n)$
is also a quasimorphism on $\PSL{2}{\Z}$ where 
\begin{equation} \label{E:defect}
  \abs{\Radem(g) - \HRadem(g)} \leq 3 \mtext{for all $g \in \PSL{2}{\Z}$}
\end{equation}
since $\Radem$ has defect 3.  In \cite[Section~4.2]{GambaudoGhys2005},
they define a function $\Phi \maps \SL{2}{\Z} \to \Z$ where
\begin{equation} \label{eq: sig from phi} \sigma_\what(-1) = -
  \frac{1}{3} \Phi\big(B_{-1}(w)\big) - \frac{2}{3} \# w,
\end{equation}
see \cite[Theorem~4.2]{GambaudoGhys2005}.  For a hyperbolic element
$g$ of $\SL{2}{\Z}$, one has $\Phi(g) := \HRadem(g);$ the definition
for other elements is more complicated, see
\cite[page~570]{GambaudoGhys2005}, where there is a typo in that
$\Phi(-C^p)$ is $p$ not $-\frac{1}{3} p$.  Regardless, you can easily
check that $\Phi$ always satisfies
$\abs{\Phi(g) - \HRadem(g)} \leq 4$.  The lemma now follows from
(\ref{E:defect}) and (\ref{eq: sig from phi}).
\end{proof}

Viewing $B_{-1}(\sigma_1)$ and $B_{-1}(\sigma_2)$ as elements of
$\PSL{2}{\Z}$, they are $ba$ and $ab$ respectively.  Thus
Lemma~\ref{lem: sig from Radem} reduces Theorem~\ref{T:clt} to the
following two results:

\begin{theorem}
  \label{thm: CLT for PSLZ}
  Consider the random walk $w_n := g_1 \cdots g_n$ on $\PSL{2}{\Z}$
  generated by the uniform measure $\mu$ on $\{ba, ab\}$.  Then
  $\Radem(w_n)$ follows a central limit theorem with positive
  variance.
\end{theorem}

\begin{theorem}
  \label{thm: exact drift}
  The drift in Theorem~\ref{thm: CLT for PSLZ} is
  \[
    \ell := \lim_{n \to \infty} \E \left( \frac{\Radem(w_n)}{n} \right)
     = \frac{1}{4} (7 - 3 \sqrt{5}) \approx 0.0729490.
  \]
\end{theorem}
Since $\Radem$ is a quasimorphism,
$\ell = \lim_{n \to \infty} \frac{\Radem(w_n)}{n}$ for almost every
$(w_n)$ by Kingman's ergodic theorem.

\begin{proof}[Proof of Theorem \ref{thm: CLT for PSLZ}]

We will use the central limit theorem for quasimorphisms of
Bj\"orklund-Hartnick, specifically in the form of
\cite[Corollary~1.3]{BjorklundHartnick2011}.  Since our measure $\mu$
has finite support, to apply it and get a central limit theorem with
positive variance, we need only check that our quasimorphism $\Radem$
is not $\mu$-tame. That is, we need to show that there do not exist
constants $L$ and $C$ such that
\begin{equation} \label{E:tame}
\abs{\Radem(w_n) - L n } \leq C
\end{equation}
for any $n$ and any $w$ in the support of $\mu^{ * n}$.  

If we take $w = \sigma_2^n = (ab)^n$, then $\Radem(w) = n$; hence
\eqref{E:tame} can only hold if $L = 1$.  On the other hand, if we
take $w = (\sigma_1 \sigma_2 \sigma_1)^n = a^n$, then
$\Radem(w) = 0 $, hence \eqref{E:tame} can only hold if $L = 0$. Since
it is impossible for both these conditions hold, then \eqref{E:tame}
has no solution.  Thus $\Radem$ is not $\mu$-tame and hence the
variance of the limit is positive.
\end{proof}

In order to attack the more difficult Theorem~\ref{thm: exact drift},
we begin with a series of lemmas.  Let
$\cW = \{\sigma_1, \sigma_2 \}^*$ be the free monoid of finite words
in the formal symbols $\{\sigma_1, \sigma_2\}$, and $\cW_n$ the subset
of those of length exactly $n$.  Let $\rho \maps \cW \to \PSL{2}{\Z}$
be the monoid homomorphism induced by $\sigma_1 \mapsto ba$ and
$\sigma_2 \mapsto ab$.  Our measure $\mu^{*n}$ on $\PSL{2}{\Z}$ is the
push-forward of the uniform probability measure on $\cW_n$ under
$\rho$.

\begin{lemma}\label{lem: conj by a}
  Let $\conj_a$ be the automorphism of $\PSL{2}{\Z}$ where
  $g \mapsto a g a$.  The measure $\mu^{*n}$ is invariant under
  $\conj_a$.
\end{lemma}

\begin{proof} 
Let $\iota$ be the involution of $\cW$ that
interchanges $\sigma_1$ and $\sigma_2$.  Notice that $\rho$ is
equivariant in the sense that $\rho \circ \iota = \conj_a \circ \rho$, where
it suffices to verify this for $\sigma_1$ and $\sigma_2$ as all the
maps are homomorphisms.  Since the uniform measure on
$\cW_n$ is invariant under $\iota$, it follows $\mu^{*n}$ is
invariant under $\conj_a$.
\end{proof}

%
%

Consider the Cayley graph of
$G = \PSL{2}{\Z} = \spandef{a, b}{a^2 = b^3 = 1}$ with respect to the
generating set $S = \{ a, b, b^{-1}\}$.  Note that $G$ is a hyperbolic
group, and the elements $ab$ and $ba$ used in defining our measure
$\mu = \frac{1}{2} \delta_{ab} + \frac{1}{2} \delta_{ba}$ are
independent loxodromics (here, we are considering the action of $G$ on
its Cayley graph, not the usual action on $\mathbb{H}^2$, for which
$ab$ and $ba$ are parabolic).  Hence, e.g.~by \cite[Theorem
1.1]{Maher-Tiozzo}, the random walk driven by $\mu$ converges almost
surely to the Gromov boundary $\partial G$, which is a Cantor set.
Thus, if we let $\nu$ be the hitting measure on $\partial G$, we have
the convergence of measures $\mu^{*n} \to \nu$ on $G \cup \partial G$
(see \cite[page~219]{Maher-Tiozzo}).  Given $g \in G$, denote by
$G_{g *}$ the set of (finite or infinite) geodesics, starting at the
identity, that begin with the word $g$.  This is an open and closed
subset of $G \cup \partial G$. Thus, for any $g$,
\[
  \lim_{n \to \infty} \mu^{*n}(G_{g *}) = \nu(G_{g *}),
\]
and Lemma~\ref{lem: conj by a} implies that
\[
  \nu(G_{b*}) = \nu(G_{ab*}) \mtext{and}   \nu(G_{B*}) = \nu(G_{aB*}).
\]
Set $B := b^{-1}$ for convenience.  As $G_{a*} = \{a\} \cup G_{ab*}
\cup G_{aB*}$, we have $\nu(G_{a*}) = \nu(G_{ab*}) + \nu(G_{aB*})$.
Combined with $\nu(G_{a*}) + \nu(G_{b*}) +
\nu(G_{B*}) = 1$, we obtain 
\begin{equation} \label{L:half}
\nu(G_{a*}) = \nu(G_{b*}) + \nu(G_{B*}) = \mfrac{1}{2}.
\end{equation}

\begin{lemma} \label{L:independence}
  We have $\nu(G_{BaB*}) = 2 \nu(G_{B*})^2$.
\end{lemma}

\begin{proof}
For the walk $\omega = (w_n)$, let
$\tau(\omega) := \min \setdef{ k \geq 1}{w_k \in G_{B*}}$ be the first
time the walk enters $G_{B*}$, which is a stopping time.  Since each
step has length $2$, there are two possible locations where the walk
enters the $G_{B*}$, so $w_\tau$ is either $B$ or
$w_\tau = Ba$. Thus, using the notation
$w_\infty := \lim_{n \to \infty} w_n$,
\[
  \nu(G_{BaB*}) = \P(w_\infty \in G_{BaB*}) = \sum_{g = B, \ Ba}
  \P(w_\infty \in G_{BaB*} \vert w_\tau = g) \cdot \P(w_\tau = g).
\]
Now, by the Markov property of stopping times,
$\P(w_\infty \in G_{BaB*} \vert w_\tau = B) = \nu(G_{aB*})  =
 \nu(G_{B*})$ and
$ \P(w_\infty \in G_{BaB*} \vert w_\tau = Ba)  = \nu(G_{B*})$.
Therefore,
\begin{equation}
\label{tau-1}
\nu(G_{BaB*}) \ = \ \nu(G_{B*})  \hspace{-0.8em} \sum_{g = B, \ Ba}
\hspace{-0.5em} \P( w_\tau = g) \ = \ 
\nu(G_{B*}) \cdot \P(\mbox{some $w_n$ enters $G_{B*}$}).
\end{equation}
Similarly, we can decompose the probability of converging to $G_{B*}$ as 
\begin{align*}
  \nu(G_{B*}) &=  \P(w_{\infty} \in G_{B*} \vert w_{\tau} = B) \cdot \P(w_\tau = B) +
                \P(w_{\infty} \in G_{B*} \vert w_{\tau} = Ba) \cdot \P(w_\tau = Ba) \\
                &=  \P(w_{\infty} \in G_{B a*} \vert w_{\tau} = B)
                  \cdot \P(w_\tau = B) + \\
                & \hspace{5em} \P(w_{\infty} \in G_{B a B *} \cup G_{Bab*} \vert w_{\tau} = Ba) \cdot \P(w_\tau = Ba) \\
  &= \nu(G_{a*}) \cdot \P(w_\tau = B) + \nu(G_{b*} \cup G_{B*}) \cdot \P(w_\tau =
    Ba) \\
  &= \mfrac{1}{2} \, \P(\mbox{some $w_n$ enters $G_{B*}$}) \mtext{by equation \eqref{L:half}.}
\end{align*}
Thus $\P(\mbox{some $w_n$ enters $G_{B*}$}) = 2 \nu(G_{B*})$.
Combined with \eqref{tau-1}, we get $\nu(G_{BaB*}) = 2
\nu(G_{B*})^2$, proving the lemma.
\end{proof}

\begin{lemma} \label{L:shadow-measure}
$\nu(G_{b*}) = \frac{1}{4}(\sqrt{5}-1) \approx 0.309070$ and $\nu(G_{B*}) =
\frac{1}{4}(3 - \sqrt{5}) \approx 0.190983$. 
\end{lemma}

\begin{proof}
Since the hitting measure $\nu$ is $\mu$-stationary, 
$\nu = \frac{1}{2} (ab)_\star \nu + \frac{1}{2} (ba)_\star \nu$, and 
hence:
\begin{align*}
  2 \nu(G_{B*})  &= (ab)_\star \nu(G_{B*}) +  (ba)_\star \nu(G_{B*}) =  \nu(Ba G_{B*}) +
  \nu(a B G_{B*}) \\
  &=  \nu(G_{BaB*}) + \nu(G_{ab*}) = \nu(G_{BaB*}) + \nu(G_{b*}) \\
  &= 2 \nu(G_{B*})^2 + \left(\mfrac{1}{2} - \nu(G_{B*})\right) \mtext{by Lemma~\ref{L:independence}
    and equation \eqref{L:half}.}
\end{align*}
Thus we see that $\nu(G_{B*})$ is a root of $x^2 - (3/2) x + 1/4$, which has a
unique root $(3 - \sqrt{5})/4$ in $[0, 1]$.  So $\nu(G_{B*}) = (3 - \sqrt{5})/4$
and $\nu(G_{b*}) = 1/2 - \nu(G_{B*}) = (\sqrt{5}-1)/4$, completing the proof.
\end{proof}

\begin{proof}[Proof of Theorem~\ref{thm: exact drift}]

We need to compute the drift $\ell$. If we let $w_n = g_1 \dots g_n$
and $u_n = g_2 \dots g_n g_{n+1}$, we have, 
\[
  \E[\Radem(w_{n+1}) - \Radem(w_n)] = \mfrac{1}{2}
  \, \E[\Radem(ab u_n) - \Radem(u_n)] + \mfrac{1}{2} \, \E[
  \Radem(ba u_n) - \Radem(u_n) ]
\]
where $\E[ \Radem(w_n)] = \E[ \Radem(u_n)]$ since the laws of $w_n$ and $u_n$ are the same.
Now note that, by definition of the Rademacher function, 
\[
  \Radem(ba u_n) - \Radem(u_n) = \left\{ \begin{array}{ll}
      -2 & \mbox{if $u_n \in G_{ab*}$} \\
      +1  & \mbox{if $u_n \notin G_{ab*}$}
   \end{array} \right.
\]
so, we get
\begin{equation*}
  \lim_{n \to \infty} \E[\Radem(ba u_n) - \Radem(u_n)]
  = (-2) \nu(G_{ab*}) + (+1)(1 -\nu(G_{ab*})) = 1 - 3 \nu(G_{ab*}).
\end{equation*}
Now, since both $\Radem$ and $\mu^{*n}$ are invariant under $\conj_a$,
we have
$\Radem(a b u_n) = \Radem( \conj_a(ab u_n)) = \Radem (ba
\conj_a(u_n))$ has the same law as $\Radem (ba u_n)$, so
$$\lim_{n \to \infty} \E[\Radem(ab u_n) - \Radem(u_n)] = \lim_{n \to
  \infty} \E[\Radem(ba u_n) - \Radem(u_n)],$$ and hence using
Lemma~\ref{L:shadow-measure} we have
\[
  \lim_{n \to \infty} \E[\Radem(w_{n+1}) - \Radem(w_n)] = 1 - 3
  \nu(G_{ab*})  = 1 - 3 \nu(G_{b*}) = \mfrac{1}{4}(7 - 3 \sqrt{5}).
\]
Thus 
\[
  \ell = \lim_{n \to \infty} \frac{\Radem(w_n)}{n} =
  \frac{1}{4}(7 - 3 \sqrt{5})
\]
completing the proof of Theorem~\ref{thm: exact drift}.
\end{proof}

\begin{remark}
  The measures $\mu^{*n}$ are quite special.  For example, it turns out
  that the probability $\mu^{*n}(g)$ depends only on the number of times
  the letters $a$, $b$, and $B$ appear in $g$ and not their order.
  Specifically, for $g \in \PSL{2}{\Z}$, define
  $c(g) = \big(\#_a(g), \#_b(g), \#_B(g)\big)$.  One can prove
  inductively that if $c(g') = c(g)$ or $c(g') = c(g) + (2, 2, -1)$,
  then $\mu^{*n}(g') = \mu^{*n}(g)$. This implies in particular that
  $\mu^{*n}(G_{B*}) = \mu^{*n}(G_{abab*})$ for all $n$, and the latter equals
  $\mu^{*n}(G_{bab*})$ by invariance under $\conj_a$.  
\end{remark}

\begin{proposition}
Recall that $\cA_L := \setdef{ e^{i \theta}}{|\theta- \pi| < \pi/3}$. Then almost every sample path $(w_n)$ 
satisfies
$$\liminf_n \ \nu_{w_n}(\cAbar_L) \geq  \ell_0 - \frac{2}{3} =  \frac{7 - 3\sqrt{5}}{12}  > 0.$$
\end{proposition}

\begin{proof}
By Propositions \ref{prop: roots on arc} and \ref{prop: sig for 3 braid}, letting $\theta_\epsilon := \frac{2 \pi}{3} - \epsilon$ for $\epsilon > 0$ small,
we obtain that almost surely (counting with multiplicity)
\begin{align*}
\# \left( Z_{w_n} \cap \cAbar_L \right) 
& = 2 \lim_{\epsilon \to 0} \# \left( Z_{w_n} \cap \{e^{i \theta}, \theta \in [\theta_\epsilon, \pi] \} \right)\\
& \geq \lim_{\epsilon \to 0}  |\sigma_{\widehat{w}_n}(e^{i \theta_\epsilon}) - \sigma_{\widehat{w}_n}(-1)| \\
& \geq \lim_{\epsilon \to 0} |\sigma_{\widehat{w}_n}(-1)| -   \frac{\theta_\epsilon}{\pi} n - 2 \\
& \geq |\sigma_{\widehat{w}_n}(-1)| -   \frac{2}{3} n - 2
\end{align*}
hence 
\begin{align*} 
\liminf_n \ \nu_{w_n}(\cAbar_L) & = \liminf_n \frac{1}{n} \# \left( Z_{w_n} \cap \cAbar_L \right)   \geq \liminf_n \frac{|\sigma_{\widehat{w}_n}(-1)|}{n}  -   \frac{2}{3}  = \ell_0 - \frac{2}{3} 
\end{align*}
as claimed.
\end{proof}

\subsection{Density on the unit circle}

Let $\Braid{3}^+$ be the semigroup of positive braids in $\Braid{3}$. 
\begin{lemma} \label{L:density-circle}
The set $\displaystyle \bigcup_{w \in \Braid{3}^+} Z_w$ is dense in the unit circle.
\end{lemma}

\begin{proof}
Consider the positive word $w_n := \sigma_1 \sigma_2 \sigma_1 \sigma_2^{n}$. 
By induction we have 
$$\sigma_2^n = \matr{1}{0}{\sum_{k = 1}^n (-1)^{k+1} t^k}{(-t)^n}$$
hence 
$$\sigma_1 \sigma_2 \sigma_1 \sigma_2^n = \matr{\sum_{k = 1}^n (-1)^k t^{k+1}}{(-t)^{n+1}}{-t^2}{0}$$
and 
$$\det(B_t(w_{n}) - I) = 1 + \sum_{k = 2}^{n+1} (-t)^k + (-t)^{n+3} 
= \frac{(1 + t + t^2) (1 + t^{n+2})}{1 + t}$$
so $Z_{w_n}$ includes the $(n+2)$th roots of $-1$, which become dense in the unit circle as $n \to \infty$.
\end{proof}

\section{Results along the real axis}
\label{sec:realax}

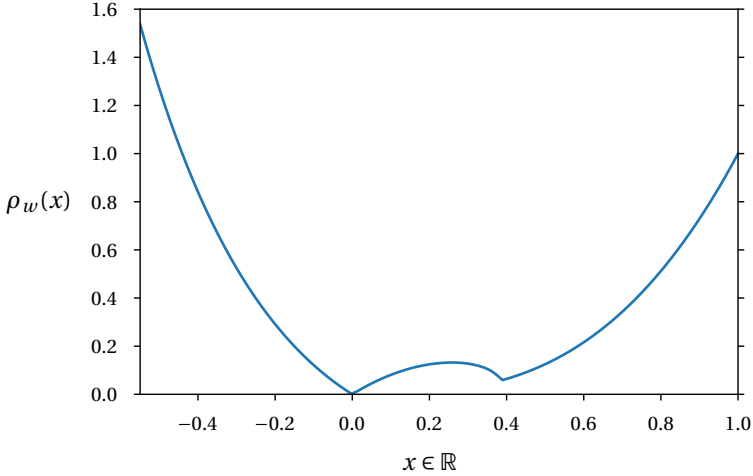
\begin{figure}
  \centering
  \begin{tikzoverlay}[width=0.7\textwidth]{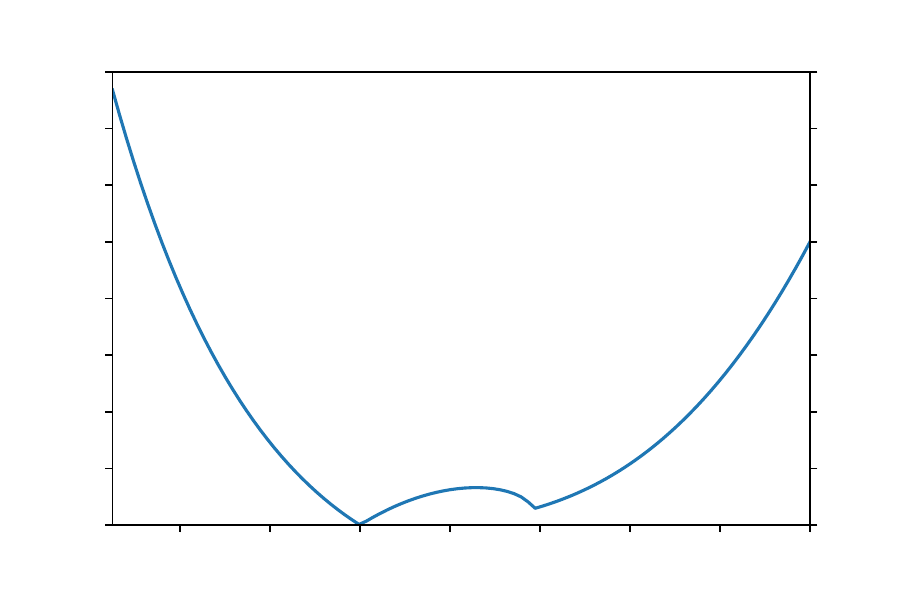}
  \begin{scope}[font=\scriptsize]
    \draw (20.000000, 6.712963) node[below] {$-0.4$};
    \draw (30.000000, 6.712963) node[below] {$-0.2$};
    \draw (40.000000, 6.712963) node[below] {$0.0$};
    \draw (50.000000, 6.712963) node[below] {$0.2$};
    \draw (60.000000, 6.712963) node[below] {$0.4$};
    \draw (70.000000, 6.712963) node[below] {$0.6$};
    \draw (80.000000, 6.712963) node[below] {$0.8$};
    \draw (90.000000, 6.712963) node[below] {$1.0$};
    \draw (10.879630, 8.333333) node[left] {$0.0$};
    \draw (10.879630, 14.625000) node[left] {$0.2$};
    \draw (10.879630, 20.916667) node[left] {$0.4$};
    \draw (10.879630, 27.208333) node[left] {$0.6$};
    \draw (10.879630, 33.500000) node[left] {$0.8$};
    \draw (10.879630, 39.791667) node[left] {$1.0$};
    \draw (10.879630, 46.083333) node[left] {$1.2$};
    \draw (10.879630, 52.375000) node[left] {$1.4$};
    \draw (10.879630, 58.666667) node[left] {$1.6$};
    \draw (50, 2) node[below, font=\footnotesize] {$x \in \R$};
    \draw (5, 33.5) node[left, font=\footnotesize] {$\rho_w(x)$};
    \begin{scope}[shift={(40.00000000, 8.33333333)}, xscale=50.00000000, yscale=31.45833333]
    \end{scope}
  \end{scope}
\end{tikzoverlay}
  \caption{The spectral radius $\rho_w$ on part of the real axis for
    $w = \sigma_1^3 \sigma_2^3$.  To the left, $\rho_w$ rises
    monotonically to $\rho_w(-1) \approx 6.85$.  There are two points
    of the set $d(w)$ from (\ref{eq:disc}) in $[-1, 1]$, namely $0$
    and $x_0 = \frac{3 - \sqrt 5}{2} \approx 0.382$.  The graph to the
    right of $x_0$ is precisely that of $x^3$.  The matrix $B_x(w)$
    has two real eigenvalues in $[-1, x_0]$ and a pair of complex
    conjugate eigenvalues in $(x_0, 1]$.  }
  \label{fig: rho real}
\end{figure}

Next, we study the behavior of $\rho_w$ on the real axis; one example
is shown in Figure~\ref{fig: rho real}.  Here are our two main
results, which are specific to when $\what$ is a knot.
\begin{theorem}
  \label{thm: rho pos reals}
  Consider a positive $w \in \Braid{3}$ where $\what$ is a knot.  Then
  $\rho_w(0) = 0$, $\rho_w(1) = 1$ and $\rho_w(t) < 1$ for
  $t \in [0, 1)$. Moreover, there is an $\epsilon > 0$ such that
  $\rho_w(t) = t^{\#w/2}$ on $(1 - \epsilon, 1 + \epsilon)$.
\end{theorem}

\begin{theorem}
  \label{thm: rho neg reals}
  Consider a positive $w \in \Braid{3}$ that is pseudo-Anosov and
  where $\what$ is a knot.  There is a unique $t \in [-1, 1)$ where
  $\rho_w(t) = 1$, which must lie in $\left(-1, \frac{\sqrt{5}-3}{2}\right]$.
  The function $\rho_w(t)$ is strictly decreasing on $[-1, 0]$ and
  analytic on $[-1, 0)$.
\end{theorem}

For $t \in \R$, equation (\ref{eq: eigens}) in the proof of
Lemma~\ref{lem: spec basics} shows that $B_{t}(w)$ either has two real
roots or two complex roots depending on whether the discriminant
$p(t)^2 - 4(-t)^{\#w}$ is nonnegative for
$p(t) = \tr\big(B_{t}(w)\big)$.  When the discriminant is negative,
the two complex roots are conjugate and hence have the same absolute
value, which one calculates as $\rho_w(t) = \abs{t}^{\#w/2}$.  We
start with:

\begin{lemma}
  \label{lem: nearone}
  Suppose $w \in \Braid{3}$ is a positive word where $\what$ is
  a knot.  Then there is an $\epsilon > 0$ such that
  $\rho_w(t) = t^{\#w/2}$ on $(1 - \epsilon, 1 + \epsilon)$.
\end{lemma}

\begin{proof}
Consider the homomorphism $\Braid{3} \to \Sym_3$ that records how a
braid permutes the punctures; its kernel is normally generated by
$\sigma_1^2$ \cite[\S 9.3]{FarbMargalit2012}.  As $B_1(\sigma_1)$ has
order two, it follows that there is a homomorphism
$\Sym_3 \to \GL{2}{\C}$ so that the composition
$\Braid{3} \to \Sym_3 \to \GL{2}{\C}$ is the Burau representation
$B_1 \maps \Braid{3} \to \GL{2}{\C}$.  If $\what$ is a knot, then the
image of $w$ in $\Sym_3$ is a 3-cycle.  (Note this implies $\#w$ is
even, since the sign of the permutation in $\Sym_3$ is $(-1)^{\#w}$
for any word.)  Consequently, $B_1(w)$ is either
$B_1(\sigma_1 \sigma_2) = \mysmallmatrix{0}{-1}{1}{-1}$ or
$B_1(\sigma_2 \sigma_1) = \mysmallmatrix{-1}{1}{-1}{0}$, both of whose
eigenvalues are $\zeta_3$ and $\zetabar_3$.  In particular, the two
eigenvalues are complex and the discriminant $p(t)^2 - 4 (-t)^{\#w}$
is negative at $t = 1$; specifically, it is $-3$. The discriminant has
to remain negative on some $(1 - \epsilon, 1 + \epsilon)$, and as
noted above, this means $\rho_w = t^{\#w/2}$ there, proving the lemma.
\end{proof}

\begin{lemma}
  \label{lem: conway pos}
  Suppose $w \in \Braid{n}$ is a positive word where $\what$ is a
  knot.  Then the Alexander polynomial $\Delta_\what(t)$ has no roots
  in $(0, 1)$.
\end{lemma}

\begin{proof}
The Conway polynomial $\nabla_\what(x) \in \Z[x]$ is a variant of the
Alexander polynomial, basically the latter written in a variable
$x = t^{-1/2} - t^{1/2}$, see e.g.~\cite[Chapter
8]{Lickorish1997}. For knots, $\nabla_\what(x)$ is a polynomial in
$x^2 = t + 1/t - 2$, and one has
\[
  \Delta_\what(t) = \nabla_\what\big(x^2 \big)
\]
where the normalization is so that
$\Delta_\what(1/t) = \Delta_\what(t)$.  By \cite[Page
151]{VanBuskirk1985}, if a knot $K$ is the closure of a positive
$n$-strand braid, then $\nabla_K(x)$ has all positive coefficients.
(This was extended to a broader class of links in Corollary~2.1 of
\cite{Cromwell1989}.)  Specifically:
\[
  \nabla_K(x) = x^{2m} + a_{2m - 2} x^{2m - 2} + \cdots  + a_2
  x^2 + 1 \mtext{where $\binom{m}{k} \leq a_{2(m - k)} \leq
  \binom{2m - k}{k}$,}
\]
where $m \geq 0$ since for any knot $\Delta_K(1) = \pm 1$.

For $t \in (0, 1)$, we have $x^2 = t + 1/t - 2$ is in $(0, \infty)$,
and no such value of $x^2$ corresponds to a root of $\nabla_K$, as the
coefficients of $\nabla_K$ are positive. So no $t \in (0, 1)$ is a
root of $\Delta_\what$, proving the lemma.
\end{proof}

\begin{proof}[Proof of Theorem~\ref{thm: rho pos reals}]
As the braid closure of $\sigma_i^n$ has either two or three
components, we know there is an occurence in $w$ of
$\sigma_a \sigma_b$ where $a \neq b$.  As $B_t(\sigma_a \sigma_b)$ has
every entry divisible by $t$, it follows that every entry of $B_t(w)$
is divisible by $t$.  In particular, $B_0(w) = 0$ and so
$\rho_w(0) = 0$.  Next, Lemma~\ref{lem: nearone} says
$\rho_w(t) = t^{\#w/2}$ near $t = 1$, which also gives
$\rho_w(1) = 1$.

It remains to show that given $t \in (0, 1)$, we have $\rho_w(t) <
1$. If eigenvalues of $B_t(w)$ are not real, then
$\rho_w(t) = \abs{t}^{\#w/2}$ by the discussion before Lemma~\ref{lem:
  nearone}, and hence $\rho_w(t) < 1$ as needed.  If instead the
eigenvalues of $B_t(w)$ are real, then $\rho_w(t) = 1$ is equivalent
to $\pm 1$ being an eigenvalue of $B_t(w)$, and hence $1$ being an
eigenvalue of $B_t(w^2)$.  Since $\what$ is a knot, $w$
gives a 3-cycle in $\Sym_3$; therefore, $w^2$ also gives a 3-cycle, and
so its braid closure is again a knot.  Thus, if we replace $w$ with
$w^2$, our remaining claim is equivalent to saying that
$\Delta_\what(t)$ has no roots in $(0, 1)$, which is Lemma~\ref{lem:
  conway pos}.
\end{proof}

Our final lemma does not require that $\what$ is a knot.

\begin{lemma}\label{lem:monotone}
  For a positive $w$ that gives a pseudo-Anosov mapping class, the
  function $\rho_w(t)$ is strictly decreasing on $[-1, 0]$ and
  analytic on $[-1, 0)$.
\end{lemma}
In contrast, $\rho_w(t)$ is typically not monotone on $[0, 1]$; see
Figure~\ref{fig: rho real}.
\begin{proof}
We show $g(x) = \rho_w(-x)$ is strictly increasing on $[0, 1]$ and
analytic on $(0, 1]$.  Using Theorem~\ref{thm:goodsign}, we can
exchange $w$ for one where $A(x) \assign \pm B_{-x}(w)$ consists of
positive polynomials in $x$ and as $\rho(A(x)) = g(x)$, we focus on
the matrix $A(x)$.

The first claim is that $A(x)$ is Perron-Frobenius for any
$x \in (0, 1]$.  Certainly each $A(x)$ has non-negative entries, the
only issue is that one might be zero.  Now by Proposition~\ref{prop:
  braid trichotomy}, we have that $A(1) \in \SL{2}{\Z}$ has trace at
least 3 and both eigenvalues are positive real, one larger that 1 and
the other smaller.  In particular, $\pm 1$ is not an eigenvalue of
$A(1)$, which means neither off-diagonal entry can be zero. Since one
of the diagonal entries of $A(1)$ is nonzero because of its trace, a
calculation shows that $A(1)^2$ has strictly positive entries.  As
they are positive polynomials, it follows that the entries of $A(x)^2$
are strictly positive for all $x \in (0, 1]$.  Consequently, $A(x)$
itself is Perron-Frobenius for all $x \in (0, 1]$.

As $A(x)$ is Perron-Frobenius, it has a positive real eigenvalue of
algebraic and geometric multiplicity 1.  Consequently, we must have
that the discriminant $p(-x)^2 - 4(x)^{\#w}$ is positive on $(0, 1]$,
and hence $g(x)$ is real analytic on $(0, 1]$ as per \eqref{eq:
  eigens}.  Moreover, $g(x)$ is nonconstant on $[0, 1]$ since
$g(0) = 0$ by Theorem~\ref{thm: rho pos reals} and
$g(1) = \rho(A(1)) > 1$.  To complete the proof, it is enough to show
that $g(x)$ is monotone nondecreasing, as then analyticity will force
it to be strictly increasing. 

For the entrywise $\ell^2$-norm, we have that $\norm{A(x)^n}_2^2$ is a
positive polynomial in $x$, and hence is an increasing function for
$x \in (0, 1)$.  Hence $\norm{A(x)^n}_2^{1/n}$ is an increasing
function of $x$ as well.  Gelfand's formula, which applies to any
matrix norm, gives
$\rho(A) = \lim_{n \to \infty} \norm{ A^n }_2^{1/n}$, and so
$g(x) = \rho(A(x))$ is a monotone nondecreasing function of $x$ as
claimed.
\end{proof}

\begin{proof}[Proof of Theorem~\ref{thm: rho neg reals}]
By Theorem~\ref{thm: rho pos reals}, we have $\rho_w(t) < 1$ on
$[0, 1)$ with $\rho_w(0) = 0$.  By Proposition~\ref{prop: braid
  trichotomy}, we know $\rho_w(-1) > 1$, and hence by continuity of
$\rho_w$ there exists at least one $t \in (-1, 0)$ with
$\rho_w(t) = 1$, and said $t \leq \frac{\sqrt{5}-3}{2}$
by Lemma~\ref{lem: zero-free
  disk}. Uniqueness then follows from Lemma~\ref{lem:monotone}, which
also gives the rest of the theorem.
\end{proof}

\section{Root-free regions} 
\label{sec: root-free}

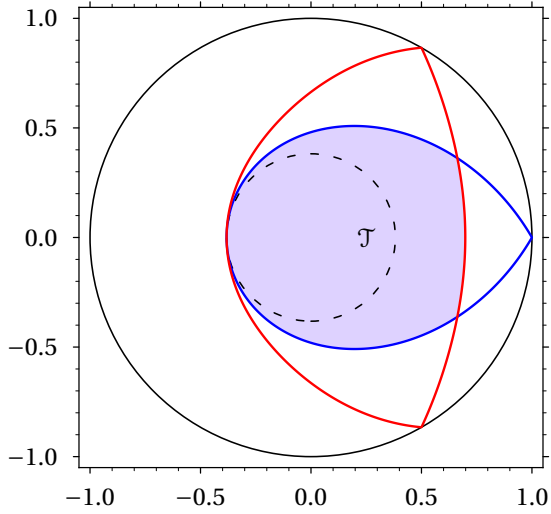
\begin{figure}
  \centering
  \begin{tikzpicture}[nmdstd]
  \begin{tikzoverlay*}[width=0.5\textwidth]{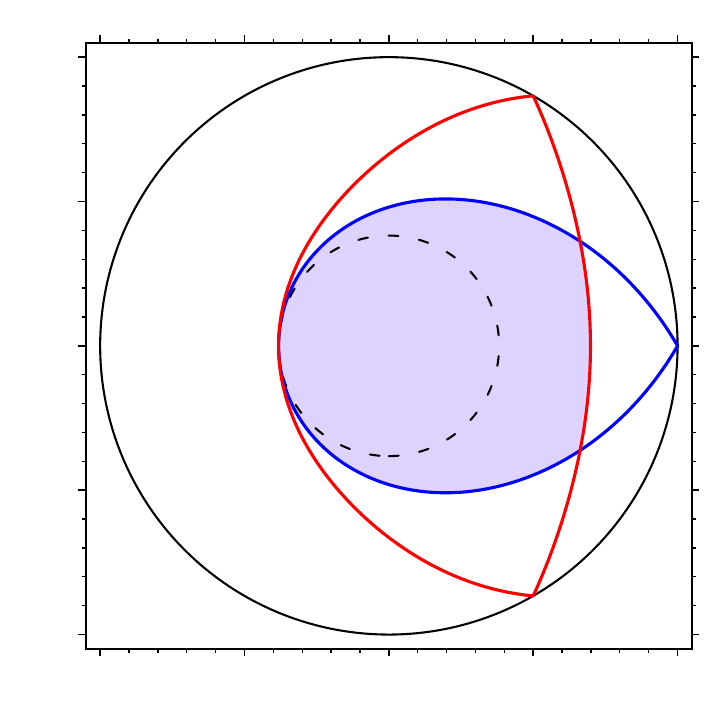}
    \draw (13.910590, 7.835937) node[below] {$-1.0$};
    \draw (33.959780, 7.835937) node[below] {$-0.5$};
    \draw (54.008970, 7.835937) node[below] {$0.0$};
    \draw (74.058160, 7.835937) node[below] {$0.5$};
    \draw (94.107350, 7.835937) node[below] {$1.0$};
    \draw (9.880208, 11.866319) node[left] {$-1.0$};
    \draw (9.880208, 31.915509) node[left] {$-0.5$};
    \draw (9.880208, 51.964699) node[left] {$0.0$};
    \draw (9.880208, 72.013889) node[left] {$0.5$};
    \draw (9.880208, 92.063079) node[left] {$1.0$};
    \begin{scope}[shift={(54.00896991, 51.96469907)},
      xscale=40.09837963, yscale=40.09837963]
      \draw (0.25, 0) node[font=\small] {$\cT$};
    \end{scope}
  \end{tikzoverlay*}
\end{tikzpicture}
  \caption{The open region $\cT$ is the shaded area inside the blue
    curve defined by $\abs{t}^{1/2} \abs{1 - t} + \abs{t}^2 = 1$ and
    the red curve defined by
    $\abs{t}^{1/2} \abs{1 - t + t^2} + \abs{t}^3 = 1$; its area is
    about $0.91$, or $29\%$ of the unit disk. The dotted circle is the
    one appearing in Lemma~\ref{lem: zero-free disk}.}
  \label{fig: new T}
\end{figure}

Recall that $R_w := \setdef{ t \in \Dbar}{\rho_w(t) = 1}$ is the
set of parameters for which the spectral radius equals $1$.  Let
\[
  \cT = \setdef{t \in \C}{%
         \mbox{$\abs{t}^{1/2} \abs{1 - t} + \abs{t}^2 < 1$ and
          $\abs{t}^{1/2} \abs{1 - t + t^2} + \abs{t}^3 < 1$}}
\]
which is shown in Figure~\ref{fig: new T}. The main result of this section is:

\begin{theorem} \label{thm: root-free} Suppose $w \in \Braid{3}$ is a
  positive word and not a power of $\sigma_1$ or $\sigma_2$. Then
  $R_w$ is disjoint from the open set $\cT$.
\end{theorem}

We will work in the $u$-coordinates of Remark \ref{R:nice-pres}, where
$t = -u^2$ and the Burau representation has been conjugated to
become:
\[
    \sigma_1 \mapsto  \matr{u^2}{u}{0}{1} \quad \sigma_2 \mapsto \matr{1}{0}{-u}{u^2} \quad
   \sigma_1 \sigma_2 \mapsto  \matr{0}{u^3}{-u}{u^2} \quad \sigma_2 \sigma_1 \mapsto  \matr{u^2}{u}{-u^3}{0}
  \]
as well as $\Omega \mapsto u^3 \mysmallmatrix{0}{1}{-1}{0}$.  In these
coordinates, $\cT$ corresponds to:
\[
  \cU = \setdef{u \in \C}{%
    \mbox{$\abs{u + u^3} + \abs{u}^4 < 1$ and
          $\abs{u + u^3 + u^5} + \abs{u}^6 < 1$}}.
\]

Recall that the $L^\infty$-norm of a vector $v = (x, y)$ in
$\mathbb{C}^2$ is $\norminf{v} := \max \big\{ \abs{x}, |y|\big\}$, and
the corresponding norm on $2 \times 2$-matrices is given by
\[
  \norminf{A} := \sup_{\norminf{v} \leq 1} \norminf{A v} =
  \max \big\{ \abs{a_{11}} + \abs{a_{12}}, \ \abs{a_{21}} + \abs{a_{22}}
\big\}.
\]
Note that the spectral radius satisfies
$\rho(A) \leq \norminf{A}$.  For any $w \in \Braid{3}$, we write
$\norminf{w}$ for $\norminf{B_{t}(w)}$, viewed as a function
of $u$. The main technique behind Theorem~\ref{thm: root-free}
will be to show $\norminf{w'} < 1$ on $\cU$ for some conjugate
$w'$ of $w$.  As a warm-up to illustrate this key idea, we next prove:

\begin{lemma}
\label{lem: zero-free disk}
If $w \in \Braid{3}$ is positive and not a power of $\sigma_1$ or
$\sigma_2$, then $R_w$ is disjoint from the open disk about $0$ of
radius $\frac{3 - \sqrt{5}}{2} \approx 0.38197$.
\end{lemma}

\begin{proof}
From the above formulas, we have
\begin{align*}
  \norminf{\sigma_1} &= \norminf{\sigma_2}
                      = \max \big\{ 1, |u| + |u|^2 \big\} \\
  \norminf{\sigma_1 \sigma_2} &= \norminf{\sigma_2 \sigma_1}
                                = \max \big\{ |u|^3, |u| + |u|^2 \big\}.
\end{align*}
Consider the open disk
$\cU_0 = \setdef{u \in \C}{\abs{u} < \frac{1}{2}\left(\sqrt{5} -
    1\right)}$, which is also precisely the region where $\abs{u} + \abs{u}^2 <
1$.  Then for $u \in \cU_0$, we have $\norminf{\sigma_1} =
\norminf{\sigma_2} = 1$ and $\norminf{\sigma_1 \sigma_2} =
\norminf{\sigma_2 \sigma_1} < 1$.  Thus for any positive $w$ that is
not a power of a $\sigma_i$, sub-multiplicativity of
$\norminf{\cdotspaced}$ implies that $\norminf{w} < 1$ on $\cU_0$ as
well.  Hence $\rho_w(u) < 1$ on $\cU_0$, and consequently 
$R_w$ is disjoint from $\setdef{t \in \C}{\abs{t} < \frac{1}{2}\left(
   \sqrt{5} - 3\right)}$ as claimed.
\end{proof}

\begin{remark}
  It turns out that Lemma~\ref{lem: zero-free disk} gives the
  largest possible disk about $0$ that is disjoint from all such
  $R_w$.  To see this, set $\chi_n(u) = 1 + u^2 + u^4 + \cdots + u^{2n
    - 2} = (1 - u^{2n})/(1 - u^2)$ and show inductively that
  \[
    \sigma_1^n \sigma_2^n \mapsto
    \twobytwomatrix{u^{2n}}{u \chi_n(u)}{0}{1} \cdot
    \twobytwomatrix{1}{0}{-u \chi_n(u)}{u^{2n}}
    = \twobytwomatrix{u^{2n} - u^2 \chi_n^2(u)}{u^{2n+1} \chi_n(u)}%
                     {-u \chi_n(u)}{u^{2n}}.
  \]
  For $\abs{u} < 1$, taking the limit as $n \to \infty$ gives
  \[
    \twobytwomatrix{-\frac{u^2}{(1 -u^2)^2}}{0}{-\frac{u}{(1 - u^2)}}{0}%
    \mtext{which has eigenvalues $-\frac{u^2}{(1 -u^2)^2}$ and $0$.}
  \]
  Thus the limiting locus of the $R_w$ is the set where
  $\abs{u} = \abs{1 - u^2}$, which includes
  $u = \frac{1}{2}\left(\sqrt{5} - 1\right)$ on the boundary of the
  disk $\cU_0$.  Hence it is not possible to increase the radius in
  Lemma~\ref{lem: zero-free disk}.
\end{remark}

\begin{lemma}
  \label{lem: R circular}
  All of $\sigma_1 \sigma_2$, $\Omega^k$,
  $\Omega^k \sigma_1$, and $\Omega^k \sigma_1 \sigma_2$ for $k \geq 1$
  have $R_w = \bdry \Dbar$.
\end{lemma}

\begin{proof}
For any $w \in \Braid{3}$, we have $R_{w^2} = R_w$, so it is suffices
to determine $R_{w^j}$ for some $j \geq 1$.  Now
$\Omega^2 \mapsto -u^6 \cdot I$ and hence
$\rho_{\Omega^2}(u) = \abs{u}^6$ and $R_{\Omega^2} = \bdry \Dbar$.  Thus
$R_{\Omega^k} = \bdry \D$ for all $k \geq 1$.  Moreover, if $R_{w}$
is $\bdry \Dbar$ then so is $R_{\Omega^2 w}$.  Hence, we need only prove
the lemma for $\sigma_1 \sigma_2$, $\Omega \sigma_1$,
$\Omega \sigma_1 \sigma_2$ and $\Omega^2 \sigma_1$. The cases of
$\Omega \sigma_1 \sigma_2$ and $\Omega^2 \sigma_1$ are easy as they go
to upper-triangular matrices whose spectral radii are $\abs{u}^4$ and
$\abs{u}^6$ respectively. The other two elements have a positive power
equal to $\Omega^{2c}$ for some $c \geq 1$,
namely $(\sigma_1 \sigma_2)^3 = \Omega^2$ and
$(\Omega \sigma_1)^3 = \Omega^4$, which proves the lemma in the
remaining cases.
\end{proof}

\begin{proof}[Proof of Theorem~\ref{thm: root-free}]

We will show the following:

\begin{claim}
  \label{claim: core}
  \begin{enumerate}[label=(\alph*), ref=\alph*]
  
  \item \label{enum: power of gen}
    For $a \geq 2$ and $u \in \cU$, we have
    $\norminf{ \sigma_1^a } = \norminf{ \sigma_2^a} = 1$.
    
  \item \label{enum: prod of two}
    For $a, b \geq 2$ and $u \in \cU$, we have
    $\norminf{\sigma_1^a \sigma_2^b} < 1$.
    
  \item \label{enum: omega}
    For $k \geq 1$ and $\abs{u} < 1$, we have $\norminf{\Omega^k} <  1$.   
  \end{enumerate}
\end{claim}
Before proving Claim~\ref{claim: core}, we explain why it suffices to
prove the theorem.  By Corollary~\ref{cor: norm with power(s)}, we can
replace $w$ by a conjugate word of one of the following four forms:

\begin{enumerate}
\item $\Omega^k \sigma_1^a$ where $k \geq 0$ and $a \geq 0$. Since the
  original $w$ is not a power of $\sigma_i$, we must have $k \geq 1$
  by Remark~\ref{rem: sigma pows}.  If $a \geq 2$, then we have
  $\norminf{w} \leq \norminf{\Omega^k} \cdot \norminf{\sigma_1^a} < 1$
  on $\cU$ by Claim~\ref{claim: core}. The case of $a=1$ is covered by
  Lemma~\ref{lem: R circular} since $k \geq 1$.
  
\item $\Omega^k \sigma_1 \sigma_2$ where $k \geq 0$, which is included
  in Lemma~\ref{lem: R circular}.

\item
  $\Omega^k \sigma_1^{a_1} \sigma_2^{a_2} \cdots \sigma_1^{a_{\ell -
      1}}\sigma_2^{a_\ell}$ where $k \geq 0$, $\ell \geq 2$ is even,
  and all $a_i \geq 2$.  Part (\ref{enum: prod of two}) of
  Claim~\ref{claim: core} gives that
  $\norminf{\sigma_1^{a_1} \sigma_2^{a_2}} < 1$ for $u \in \cU$
  and that all the rest of the terms have norm at most 1.  Hence
  $\norminf{w} < 1$ on $\cU$ by sub-multiplicativity.
    
\item
  $\Omega^k \sigma_1^{a_1} \sigma_2^{a_2} \cdots \sigma_2^{a_{\ell -
      1}} \sigma_1^{a_{\ell}}$ where $k \geq 1$ is odd, $\ell \geq 3$
  is odd, and all $a_i \geq 2$.  This is identical to the previous
  case.
\end{enumerate}

Thus, to prove the theorem, it remains to establish Claim~\ref{claim:
  core}.  To start with (\ref{enum: power of gen}), since $\sigma_1$
and $\sigma_2$ are interchanged by conjugating by the isometry
$\mysmallmatrix{0}{1}{-1}{0}$, it suffices to check this for
$\sigma_1$.  As $\sigma_1^a$ has $1$ as an eigenvalue, we know
$\norminf{\sigma_1^a} \geq 1$.  As the norm is sub-multiplicative, we
only need check $\norminf{\sigma_1^a} \leq 1$ for $a = 2$ and $a = 3$,
which is the case since one calculates
\[
  \norminf{\sigma_1^2} = \max\left\{1, \ \abs{u + u^3} + \abs{u}^4\right\} %
  \mtext{and}
  \norminf{\sigma_1^3} = \max\left\{1, \ \abs{u + u^3 + u^5} + \abs{u}^6\right\}
\]
and so we have
$\norminf{\sigma_1^2} = \norminf{\sigma_1^3} = 1$ for
$u \in \cU$.

Next, we consider (\ref{enum: prod of two}). By sub-multiplicativity
and part (\ref{enum: power of gen}), it suffices to check
this for $\sigma_1^2 \sigma_2^2$, $\sigma_1^2 \sigma_2^3$,
$\sigma_1^3 \sigma_2^2$, and $\sigma_1^3 \sigma_2^3$. Calculating the
matrices, one finds in each case that on $\cU$ the second row has larger norm
and hence
\begin{align*}
  \norminf{\sigma_1^2 \sigma_2^2} &=  \norminf{\sigma_1^3 \sigma_2^2}  = \abs{u + u^3} + \abs{u}^4 \\ 
  \norminf{\sigma_1^2 \sigma_2^3} &=  \norminf{\sigma_1^3 \sigma_2^3} = \abs{u + u^3 + u^5} + \abs{u}^6
\end{align*}
all of which are strictly smaller than $1$ on $\cU$ as claimed. So we
have shown (\ref{enum: prod of two}).  Finally, for (\ref{enum:
  omega}), we have $\Omega \mapsto u^3 \mysmallmatrix{0}{1}{-1}{0}$
and hence $\norminf{\Omega^k} \leq \abs{u}^{3k} < 1$ when
$\abs{u} < 1$ as needed.
\end{proof}

We end this subsection with the following lemma.
\begin{lemma}
  \label{lem: neg reals are hard}
  For any word $w \in \Braid{3}$, not necessarily positive,
  $\rho_w(t) \leq \rho_w(-\abs{t})$.
\end{lemma}

\begin{proof}
By Theorem~\ref{thm:goodsign}, we can replace $w$ with $w'$ so that
$B_{-t}(w') = \pm M(t)$ where $M$ is a matrix with all entries
positive polynomials.  Note that, if $p(t)$ is a polynomial with
positive coefficients, then $|p(t)| \leq p(|t|)$ for any
$t \in \mathbb{C}$. Hence, for any matrix $M(t)$ whose entries are
polynomials with positive coefficients, for any $t \in \mathbb{C}$,
$\norminf{M(t)} \leq \norminf{ M(|t|)}$.  Similarly, for each
$n \geq 1$, the entries of the power matrix $M^n(t)$ are also
polynomials with positive coefficients, hence
$\norminf{ M^n(t) } \leq \norminf{M^n(|t|)}$. Finally, by the Gelfand
formula, the spectral radius $\rho$ satisfies
\[
  \rho(M(t)) = \lim_{n \to \infty} \norm{ M^n(t) }^{1/n}_\infty \leq
  \lim_{n \to \infty} \norm{ M^n(|t|)}^{1/n}_\infty = \rho\big(M(|t|)\big).
\]
By applying this to $B_{-t}(w')$, we obtain
$\rho_w(t) = \rho_{w'}(t) \leq \rho_{w'}(-|t|) = \rho_w(-|t|)$ as
claimed.
\end{proof}

\subsection{Roots in the right half-plane} \label{S:right-plane}

In this final subsection, we introduce Conjecture~\ref{conj: RH},
which is effectively a combinatorial condition on the coefficients of
$\Delta_\what$, and show:

\begin{proposition}
  Consider a positive $w \in \Braid{3}$ where $\what$ is a knot. If
  Conjecture~\ref{conj: RH} holds, then all roots of $\Delta_\what$ in
  the half-plane $\setdef{t}{\Re(t) > 0}$ lie on the unit circle.
\end{proposition}

Suppose $w \in \Braid{3}$ is positive with $\what$ a knot. As
noted when proving Lemma~\ref{lem: nearone}, the length of $w$
is even, say $2m$.  From \eqref{eq: burau det} and Theorem~\ref{thm:
  burau pos}, we have
\[
  g(t) = \det (B_t(w) - I) = t^{2m} - \tr(B_t(w)) + 1
\]
has degree $2m$.  As $g(t)$ is symmetric, there exists a real monic
polynomial $f(x)$ of degree $m$ such that
\[
  f(t+ t^{-1}) := t^{-m} \det (B_t(w) - I).
\]
Finally, let us define the polynomial
\[
  h(y) := \Re\big(i^{-m} f(i y)\big) = \sum_{k = 0}^{m/2} b_k y^{m-2k}
\]
where by $\Re$ we mean taking the real parts of the coefficients. Note
that $b_0 = 1$, since $f$ is monic. The following is supported by much
experimental evidence:

\begin{conjecture}
  \label{conj: RH}
  Let $w \in \Braid{3}$ be a positive word where $\what$ is a knot.
  Then all coefficients of $h$ are positive: that is,
  $b_{k} > 0$ for each $0 \leq k \leq m/2$, where
  $\#w = 2m$.
\end{conjecture}



\begin{lemma} \label{L:find-roots}
  Let $w \in \Braid{3}$ be a positive word where $\what$ is a knot, with
  $\# w = 2m$.  Then the function $f$ has: 
  \begin{enumerate}
  \item if $m$ is even, at least $\frac{m}{2}$ roots on $[0, 2]$; 
  \item if $m$ is odd:
    \begin{itemize}
    \item if $(-1)^{\frac{m-1}{2}} f(0) < 0$, at least $\frac{m+1}{2}$
      roots on $[0, 2]$, or
    \item  if $(-1)^{\frac{m-1}{2}} f(0) > 0$, at least
      $\frac{m-1}{2}$ roots on $[0, 2]$.
    \end{itemize}
  \end{enumerate}
\end{lemma}

\begin{proof}
Note that, setting $t = e^{i \theta}$ and $r(\theta) = e^{-i m \theta}
\big(\tr B_{e^{i \theta}}(w)\big)$, one can write for any $0 \leq \theta \leq \pi/2$, 
$$f(2 \cos \theta) = 2 \cos(m \theta) - r(\theta).$$
Here, $|r(\theta)| \leq 2$ by Lemma~\ref{lem: sig on circle}.

Now, if $m$ is even, note that $u(\theta) := 2 \cos(m \theta)$ has
precisely $\frac{m}{2}$ monotonicity intervals for
$\theta \in [0, \pi/2]$, and moreover, $u(\pi/2) = \pm 2$, hence the
equation $u(\theta) = r(\theta)$ has, by the intermediate value
theorem, at least $\frac{m}{2}$ solutions (counted with multiplicities
when the graphs of $u(\theta)$ and $r(\theta)$ are tangent to each
other).

If $m$ is odd, note that $u(\pi/2) = 0$. Since $u(\theta)$ has exactly
$\frac{m-1}{2}$ monotonicity intervals for
$\theta \in \big[0, \frac{m-1}{m} \frac{\pi}{2}\big]$, there are at
least $\frac{m-1}{2}$ solutions of the equation
$u(\theta) = r(\theta)$.

Suppose that $m \equiv 1 \bmod 4$ and $f(0) < 0$; then setting
$\theta = \pi/2$, we have $f(0) = - r(\pi/2) < 0$ hence $r(\pi/2) > 0$,
which means that
$u\big( \frac{m-1}{m} \frac{\pi}{2}\big) = +2 \geq r\big( \frac{m-1}{m}
\frac{\pi}{2}\big)$ while $u( \pi/2 ) = 0 < r(\pi/2)$.  Hence there is at
least another root of $u(\theta) = r(\theta)$ in
$\big[ \frac{m-1}{m} \frac{\pi}{2}, \frac{\pi}{2}\big]$, yielding a total of
at least $\frac{m+1}{2}$ roots.

The case of $m \equiv 3 \bmod 4$ and $f(0) > 0$ is symmetric.
\end{proof}

\begin{proposition}
  Let $w \in \Braid{3}$ be positive where $\what$ is a knot, with
  $\#w = 2m$.  If Conjecture~\ref{conj: RH} is true, then all roots
  of $g(t)$ in the half-plane $\setdef{ t }{\Re(t) > 0}$ lie on the
  unit circle.  Moreover, if $m$ is even, the number of such roots is
  exactly $m$; if $m$ is odd, it is either $m-1$ or $m+1$.
\end{proposition}

\begin{proof}
Let us first assume that $m$ is even.  Using the Routh-Hurwitz
criterion, if we write $f(i y) = P_0(y) + i P_1(y)$, and $p$ is the
number of roots of $f$ in the left half-plane and $q$ is the number of
roots of $f$ in the right half-plane, we have
$$p - q = \frac{1}{\pi} \Delta\ \textup{arg } f(iy) = - I_{-\infty}^{+\infty} \frac{P_1(y)}{P_0(y)},$$ 
where $I_{-\infty}^{+\infty}$ denotes the Cauchy index.  Now, if
Conjecture~\ref{conj: RH} is true, then $P_0(y) = (-1)^{m/2} h(y)$ has
no roots, so the real part of $f(iy)$ has constant sign; moreover, since the degree of $\tr B_t(w)$ is at most $2m -1$ (Theorem~\ref{thm: burau pos}), one has $\textup{deg} P_1 \leq m - 1 < m = \textup{deg} P_0$, 
and hence
$\lim_{y \to \pm \infty} \frac{P_1(y)}{P_0(y)} = 0.$
This implies that 
 $\Delta\ \textup{arg } f(iy) = 0 $,
showing $p = q$.  Since the total
number of roots of $f$ is $m$ and we already found $\frac{m}{2}$ roots
in the interval $[0, 2]$ by Lemma \ref{L:find-roots}, these must be
all of them.

Recall that the map $\pi(t) = t + t^{-1}$ is a rational map of degree
$2$ that maps the imaginary axis onto itself, and maps the portion of
the unit circle in the right half-plane onto the interval $[0, 2]$.
As $\what$ is a knot,
$f(2) = g(1) = 3 \Delta_\what (1) = \pm 3 \neq 0$. 
Thus, $g(t) t^{-m} = f(\pi(t))$ has exactly $m$
roots in the right half-plane, and all of them lie on the unit circle,
and the same is true for $g(t)$.

If $m$ is odd, the roles of $P_0(y)$ and $P_1(y)$ are reversed.  In
fact, if Conjecture~\ref{conj: RH} is true, one can write
$P_1(y) = y Q_1(y^2)$ where all coefficients of the polynomial $Q_1$
have the same sign, thus $P_1(y)$ has exactly one root on the
imaginary axis, namely $y = 0$.  So, given that
$$p - q = \frac{1}{\pi} \Delta\ \textup{arg } f(iy) = + I_{-\infty}^{+\infty} \frac{P_0(y)}{P_1(y)},$$ 
we have $p -q = \pm 1$, from which it follows that $q$ equals either
$\frac{m-1}{2}$ or $\frac{m+1}{2}$.

We denote the coefficients of $f$ as 
$$f(x) = a_m x^m + a_{m-1} x^m  + \dots + a_1 x + a_0$$
where we note that $a_m = 1$.
To complete the proof, we need to show that, if $q = \frac{m+1}{2}$,
then all such roots are in the right half-plane.  By the previous
equation, this happens if and only if
$$ I_{-\infty}^{+\infty}  \frac{P_0(y)}{P_1(y)} = -1,$$
which is equivalent to $\frac{a_0}{a_1} < 0$. Note that, since $\what$
is a knot, $g(i) \neq 0$, hence also $f(0) \neq 0$ \cite[Corollary
6.11]{Lickorish1997}. There are two cases:
\begin{enumerate}
\item $m \equiv1 \bmod 4$. Then the conjecture implies $a_1 > 0$. So,
  if in addition $ a_0 = f(0) < 0$, then by Lemma \ref{L:find-roots}
  there are at least $\frac{m+1}{2}$ roots of $f$ on $[0, 2]$.

\item $m \equiv 3 \bmod 4$. Then the conjecture implies $a_1 < 0$. So,
  if in addition $ a_0 = f(0) > 0$, then by Lemma \ref{L:find-roots}
  there are at least $\frac{m+1}{2}$ roots of $f$ on $[0, 2]$.
\end{enumerate}
In either case, if $q = \frac{m+1}{2}$, then all such roots are on
$[0, 2]$, hence there are exactly $m + 1$ roots on the right portion
of the unit circle for $g(t)$.
\end{proof}


\section{The Lyapunov exponent and equidistribution}
\label{sec: Lyapunov}

Let $\mu$ be a finitely supported probability measure on the positive semigroup $\Braid{3}^+$ in $\Braid{3}$. 
Let us consider the random walk $w_n = g_1 g_2 \dots g_n$ on $\Braid{3}^+$ driven by $\mu$. For
any $t \in \mathbb{C}$, we define the \emph{Lyapunov exponent} of the
random walk as
\begin{equation} \label{E:lyap}
\lambda(t) := \lim_{n \to \infty} \frac{1}{n} \int \log \Vert B_t(w_n) \Vert \ d \mathbb{P}
\end{equation}
(where $\lambda(0) = - \infty$).  
Note that since all matrix
norms are equivalent up to multiplicative constants, the definition of
$\lambda(t)$ does not depend on the norm chosen.  Now let us consider, for $w \in \Braid{3}$, 
the set
\[
  Z_w :=\setdef{ t \in \C}{ \Delta_{\what}(t) = 0} 
 \]
of roots of the Alexander polynomial, counted with multiplicities, and the associated measure 
$$\nu_w := \frac{1}{ \# Z_w} \sum_{t \in Z_w} \delta_t.$$

Set $\log^+(x) := \max \big\{ 0, \log(x) \big\}$, and define 
$$\chi(t) := \max \big\{ \lambda(t), \log^+ |t| \big\}$$
and the \emph{bifurcation measure} 
$$\nu_\bif := \frac{1}{2 \pi} d d^c \chi,$$
where $d d^c$ equals the usual laplacian as we are in one complex dimension.

For $u \in \mathbb{C}^\times$, we define $A_u(w) := u^{- | w |} B_{-u^2}(w)$, so that 
$A_u : \Braid{3}^+ \to \textrm{SL}_2 \mathbb{C}$ is a homomorphism. 
Recall that, for $t \in \mathbb{C}^\times$,  we defined $\Gamma_t^+ <
\textrm{PSL}_2 \mathbb{C}$ as the semigroup generated by the images of
$A_u(\sigma_1)$ and $A_u(\sigma_2)$, where $u$ is chosen with $u^2 = -t$ (note that $\Gamma_t^+$ only depends on $t$ and not on $u$).
 We need to look at the set  
$$V := \setdef{ t \in \C^\times}{\mbox{$\Gamma_t^+$ is non-elementary}}.$$
By Proposition \ref{P:non-elt}, we have 
$$\mathbb{C} \setminus V = \{ 0 \} \cup \cAbar_R.$$ 

Let $F := \overline{\setdef{ t \in V}{\lambda(t) = \log^+|t|}}$, and
$U := \mathbb{C} \setminus F$.   We will show in Lemma \ref{L:lambda} that $U$ contains the arc $\cA_L$
of the unit circle. Moreover, conjecturally, no point of
$\cA_R$ is accumulated by points in $F$, hence $U$ is expected to
contain $S^1 \setminus \{ \zeta_3, \overline{\zeta}_3 \}$. 
We say that $\mu$ is \emph{generating} if the semigroup generated by the support of $\mu$ is $\Braid{3}^+$.

The main theorem of this section is the following equidistribution result: 

\begin{theorem} \label{T:equid} Let $\mu$ be a finitely supported,
  generating probability measure on the positive semigroup
  $\Braid{3}^+$.  Then for a.e.~sample path $(w_n)$, we have the
  convergence
  $$\lim_{n \to \infty} \frac{1}{n} \nu_{w_n} \vert_{U} = \nu_\bif\vert_{U}$$
  as measures on $U$, i.e.~as elements of the dual of $C^0_c(U)$.
\end{theorem}
 
We conjecture that the same equidistribution holds if we replace $U$ by $\mathbb{C}$. 
As an example application of Theorem \ref{T:equid}, we obtain 

\begin{corollary}  \label{C:zero-density}
Let $K \subseteq \D $ be a compact set contained in $\setdef{t}{\lambda(t) <  0}$. Then for almost every sample path $(w_n)$, 
$$\lim_{n \to \infty} \frac{1}{n} \# \setdef{ t \in K }{\Delta_{\what_n}(t) = 0} = 0.$$
\end{corollary}

\begin{proof}
By Lemma \ref{lem: weak sub}, the set $\setdef{t \in \D}{\lambda(t) <  0}$ is open.  Let $W \subseteq \D$ be open, relatively compact in $\D$, with
$K \subseteq W \subseteq \setdef{t \in \C}{\lambda(t) < 0}$.
Choose
$\phi \in C^0_c(W)$ nonnegative with $\phi(t) = 1$ for any $t \in
K$. Thus, $\chi(t) = \max \big\{\lambda(t), 0 \big\} = 0$ for
$t \in W$, hence $\nu_\bif\vert_W = 0$.  Thus,
$$\frac{1}{n} \# \{ t \in K \ : \ \Delta_{\what_n}(t) = 0 \} =
\frac{1}{n} \langle \mathbb{1}_K, \nu_{w_n} \rangle \leq \frac{1}{n}
\langle \phi, \nu_{w_n} \rangle \to \langle \phi, \nu_\bif \rangle  =
0$$
proving the corollary.
\end{proof}

\subsection{The Lyapunov exponent}

Given an open set $U \subseteq \mathbb{C}$, we say a function $f \in L^1_{loc}(U)$ is \emph{weakly subharmonic} if $d d^c f \geq 0$ as a distribution; that is, its distributional laplacian is a positive distribution, hence it is a positive measure. 
A function is \emph{subharmonic} if it is weakly subharmonic and upper semicontinuous.

\begin{lemma}\label{lem: weak sub}
  The functions $\chi(t)$ and $\lambda(t)$ are subharmonic on
  $\mathbb{C}$, with $\chi(t)$ continuous on $\mathbb{C}$ and
  $\lambda(t)$ continuous on $\mathbb{C} \setminus \{0 \}$.
\end{lemma}

\begin{proof}
First, we show $\lambda(t)$ is weakly subharmonic.
For any sample path $\omega = (w_n)$ and any $t \in \mathbb{C}$, denote 
\[
  \lambda_n(t, \omega) := \frac{1}{n} \log \Vert B_t(w_n) \Vert
  \mtext{which is subharmonic in $t$. }
\]
Now, by Kingman's ergodic theorem,
$\lim_{n \to \infty} \lambda_n(t, \omega) = \lambda(t)$ for almost every $\omega$ and any $t \in \mathbb{C}$.
If $A \in \textrm{SL}_2 \mathbb{C}$, then its spectral radius satisfies $\rho(A) \geq 1$, so also $\Vert A \Vert \geq \rho(A) \geq 1$. 
Thus, 
$$\lambda_n(t, \omega) = \frac{1}{n} \log \Vert B_t(w_n) \Vert = \frac{1}{2} \log |t| + \frac{1}{n} \log \Vert A_u(w_n) \Vert \geq \frac{1}{2} \log |t|$$
and also, since the norm is submultiplicative, 
\begin{equation} \label{E:lambda_n-above}
  \frac{1}{n} \log \Vert B_t(w_n) \Vert \leq  \log M(|t|)
  \mtext{where $M(r) := \sup_{|t'| \leq r} \max_{i = 1, 2} \Vert B_{t'}(\sigma_i) \Vert$.}
\end{equation}
Since $f(t) = \frac{1}{2} \log |t|$ lies in $L^1_{loc}(\mathbb{C})$,
the dominated convergence theorem shows that
$\lambda_n(t, \omega) \to \lambda(t)$ in $L^1_{loc}(\mathbb{C})$.
Then, $\lambda(t)$ is the limit in $L^1_{loc}$ of a sequence of
subharmonic functions, so it is weakly subharmonic. Now
$\chi(t)$ is also weakly subharmonic as it is the max of two such
functions.

Continuity of $\lambda(t)$ for $t \in \mathbb{C} \setminus \{0 \}$
follows immediately from \cite[Theorem~A]{BockerViana2017}. As
$\lambda(0) = -\infty$, we see $\lambda(t)$ is upper semicontinuous
and hence subharmonic on all of $\C$.  The function $\chi(t)$ is
$\log^+\abs{t} = 0$ on a neighbourhood of $0$ and hence is continuous
and subharmonic everywhere.
\end{proof}

Recall $\log^+(x) := \max \{ 0, \log(x) \}$ and $\log^-(x) := - \min \{ 0, \log(x) \}$, so that $\log(x) = \log^+(x) - \log^-(x)$
and $|\log(x)| = \log^+(x) + \log^-(x)$. The following uniform integrability property will be useful.


\begin{lemma} \label{L:log-}
For any $\epsilon > 0$ there exists a $\delta > 0$ such that for any measurable set $A \subseteq \mathbb{C}$ with $m(A) < \delta$ and 
for any sequence $(p_n(t))_{n \in \mathbb{N}}$ of nonzero polynomials with integer coefficients, so that $p_n(t)$ has degree at most $n$, 
we have 
$$\lim_{n \to \infty} \int_{A} \frac{1}{n} \log^- |p_n(t)| \ dm(t) < \epsilon.$$
\end{lemma}

\begin{proof}
The family of functions $(f_u)_{u \in \mathbb{C}}$ defined as $f_u(t) := \log^- |t - u|$ is uniformly integrable on $\mathbb{C}$, 
since 
$$\inf_{a \geq 0} \sup_{u \in \mathbb{C}} \int_{\{ f_u \geq a \} }  \ f_u(t) \ dm(t) = 0.$$
Thus, for any $\epsilon > 0$, there exists $\delta > 0$ such that if $m(A) < \delta$ we have 
$$\int_A f_u(t)  \ d m< \epsilon \qquad \textup{for any }u \in \mathbb{C}.$$
Now, if $p_n(t) = a_n t^n + \dots + a_0 = a_n \prod_{i = 1}^n (t - t_i)$ is a polynomial of degree $n$ with integer coefficients, then 
$$\frac{1}{n} \log^- |p_n(t)| \leq
\frac{1}{n} \left( \log^- |a_n| +  \sum_{i = 1}^n \log^- |t - t_i| \right) \leq  \frac{1}{n} \left( \sum_{|t - t_i| \leq 1}^n \log^- |t - t_i| \right)$$
so, if $m(A) < \delta$,  then
$$\int_{A} \frac{1}{n} \log^- |p_n(t)| \ dm \leq \frac{1}{n} \sum_{|t - t_i| \leq 1} \int_{A} f_{t_i}(t)  \ dm < \epsilon$$
as desired. 
\end{proof}

\begin{lemma}\label{L:exist}
Let $\mu$ be a generating, finitely supported measure on $\Braid{3}^+$.
For any $t \in V$, both sequences
$$\lim_{n \to \infty} \frac{1}{n} \log  \Vert  B_t(w_n) \Vert, \qquad \lim_{n \to \infty} \frac{1}{n} \log | \tr B_t(w_n) |$$
converge for a.e. $(w_n)$ to $\lambda(t)$. Moreover, for a.e. $(w_n)$
the convergence also holds in $L^1_{loc}(\mathbb{C})$.
\end{lemma}

We use the following lemma about random walks \cite[Theorem 9]{Guivarch1990}.

\begin{lemma}  \label{L:RW}
Let $\mu$ be a finitely supported measure on $\textup{PSL}_2 \mathbb{C} = \textup{Isom}^+(\mathbb{H}^3)$. 
If the semigroup generated by the support of $\mu$ contains two independent loxodromics, then there exists $\ell > 0$ such that
for almost every sample path $w_n$ we have
$$\lim_{n \to \infty} \frac{1}{n} \log \Vert w_n \Vert = \lim_{n \to \infty} \frac{1}{n} \log |\tr w_n| = \ell.$$
\end{lemma}

\begin{proof}[Proof of Lemma \ref{L:exist}]
By the previous Lemma \ref{L:RW}, for $t \in V$ and $u^2 = -t$, the limit 
$$\ell(u) := \lim_{n \to \infty} \frac{1}{n} \log \Vert A_u(w_n) \Vert = \lim_{n \to \infty} \frac{1}{n} \log |\tr A_u(w_n)|$$
exists almost surely, and $\ell(u) > 0$.
Since $\tr B_t(w_n) = u^n \tr A_u(w_n)$ and $\Vert B_t(w_n) \Vert = |u|^n \Vert A_u(w_n) \Vert$, we obtain 
$$\lim_{n \to \infty} \frac{1}{n} \log |\tr B_t(w_n) | = \lim_{n \to \infty} \frac{1}{n} \log \Vert B_t(w_n) \Vert = \ell(u) + \frac{1}{2} \log |t|$$
so the first claim holds since $ \lambda(t) = \ell(u) + \frac{1}{2} \log |t|$.
As we have seen in the proof of Lemma \ref{lem: weak sub}, for a.e.~$\omega = (w_n)$, we have $\lambda_n(t, \omega) \to \lambda(t)$ in $L^1_{loc}(\mathbb{C})$.

Now, there exists a universal $C > 1$, depending on the choice of norm $\Vert \cdot \Vert$, such that 
$|\tr B_t(w_n)| \leq C  \Vert B_t(w_n) \Vert$
so 
\begin{equation} \label{E:tau_n-above}
\tau_n(t, \omega) := \frac{1}{n} \log |\tr B_t(w_n)| \leq \frac{\log C}{n} + \lambda_n(t, \omega).
\end{equation}
Now, by Lemma \ref{L:RW}, there exists a countable dense subset $E \subseteq V$, such that for almost every $\omega$ we have $\tau_n(t, \omega) \to \lambda(t)$ for all $t \in E$. Using \cite[Lemma 4.3]{DeroinDujardin2012}, since the limit $\lambda(t)$ is continuous on the set $V$, 
 it follows that also $\tau_n(t, \omega)$ converges to $\lambda(t)$  in $L^1_{loc}(V)$. 
 
Now, for any $\epsilon > 0$ let $\delta > 0$ be given by Lemma \ref{L:log-}; since $\C \setminus V$ has measure zero, there exists an open set $A_\epsilon$
with $\C \setminus V \subseteq A_{\epsilon} \subseteq \{ |t| < 2 \}$ with $m(A_{\epsilon}) < \delta$, hence by Lemma \ref{L:log-}
$$\int_{A_\epsilon} \frac{1}{n} \log^-  |\tr B_t(w_n)| \ dm(t)< \epsilon \qquad \textup{for any }n \geq 0.$$
Moreover, by equations \eqref{E:lambda_n-above} and \eqref{E:tau_n-above}, $\tau_n(t, \omega) \leq \frac{\log C}{n} + \log M(|t|)$, so $\tau_n(t, \omega)$
is uniformly bounded above on $A_{\epsilon}$, hence
$$\lim_{\epsilon \to 0} \sup_{n \geq 0} \int_{A_\epsilon} |\tau_n(t, \omega)| \ dm(t) = 0,$$
so $\tau_n(t, \omega)$ converges to $\lambda(t)$ in $L^1_{loc}(\C)$.

\end{proof}

\begin{lemma} \label{L:lambda}
Let $\mu$ be a finitely supported, generating probability measure on $\Braid{3}^+$, 
and consider the random walk driven by $\mu$. 
We have 
\begin{enumerate}
\item \label{item: at zero}
$\lim_{t \to 0} \lambda(t) = -\infty$; 
\item \label{item: sym}
$\lambda(1/t) = \lambda(t) - \log |t|$ for any $t \in \C^\times$;
\item \label{item: above log}
$\lambda(t) > \frac{1}{2} \log |t|$ on $V$; in particular, $\lambda(t) > 0$ on $\cA_L$.
\item \label{item: lambda on A_R}
 $\lambda(t) = 0$ on $\cAbar_R$.
\end{enumerate}
\end{lemma}

As a corollary, the function $\chi$ satisfies:

\begin{corollary}
Let $\mu$ be a finitely supported, generating probability measure on $\Braid{3}^+$, 
 and consider the random walk driven by $\mu$. 
We have 
\begin{enumerate}
\item 
There is a neighbourhood of $0$ where $\chi(t) = 0$; 
\item 
$\chi(1/t) = \chi(t) - \log |t|$ for any $t \in \C^\times$;
\item
$\chi(t) > 0$ on $\cA_L$.
\item 
$\chi(t) = 0$ on $\cAbar_R$.
\end{enumerate}
\end{corollary}


\begin{proof}[Proof of Lemma \ref{L:lambda}]
Since 
\begin{equation} \label{E:ell}
\lambda(t) = \ell(u) + \frac{1}{2} \log |t|
\end{equation} 
and $\ell(u)$ is bounded above for $u \in \D$, we get (\ref{item: at zero}).

To prove (\ref{item: sym}), we introduce the following notation: 
given a positive word $w =  \sigma_{i_1} \cdots \sigma_{i_n} $, we denote $\wbar := \sigma_{3 - i_1} \cdots \sigma_{3 - i_n}$ 
and $\check{w} := \sigma_{3 - i_n} \cdots \sigma_{3 - i_1}$.
First, note that 
$$B_{t^{-1}} (\sigma_i) = -t^{-1} E \big(B_t(\sigma_{3-i})\big)^t E \qquad \textup{for }i = 1, 2$$
where $E = \mysmallmatrix{1}{0}{0}{-1}$,  
so for any word $w_n$ of length $n$ in the positive semigroup, 
$$B_{t^{-1}}(w_n) = (-1)^n t^{-n} E \big(B_{t}(\check{w}_n)\big)^{t} E$$
hence $\Vert B_{t^{-1}}(w_n) \Vert =  |t|^{-n} \Vert B_{t}(\check{w}_n) \Vert$, and by taking logs and letting $n \to \infty$, 
\begin{equation} \label{E:1t-1}
\lambda(1/t) = - \log |t| + \lim_n \int \frac{1}{n} \log \Vert B_{t}(\check{w}_n) \Vert \ d \mathbb{P}.
\end{equation}
Now, noting that $\wbar_n$ and $\check{w}_n$ have the same distribution, for any $n$ 
\begin{equation} \label{E:1t-2}
\int \frac{1}{n} \log \Vert B_t(\check{w}_n) \Vert \  d \mathbb{P} = \int
\frac{1}{n}  \log \Vert B_t(\wbar_n) \Vert \ d \mathbb{P}.
\end{equation}
Finally, setting $V := \mysmallmatrix{0}{u^{-1}}{-u}{0}$ with $u^2 = -t$, we have the identity 
$$V B_t(\sigma_i) V^{-1} = B_t(\sigma_{3-i}) \qquad \textup{for }i = 1, 2.$$
Thus for any word $w_n$ we get
$V B_t( w_n ) V^{-1} = B_t( \wbar_n)$
which implies
$\big|\log \Vert B_t(w_n) \Vert - \log \Vert B_t(\wbar_n) \Vert \big|  \leq \log( \Vert V \Vert \cdot  \Vert V^{-1} \Vert)$
hence 
$$\lim_n \int \frac{1}{n} \log \Vert B_t(\wbar_n) \Vert \ d\mathbb{P}= \lim_n \int \frac{1}{n} \log \Vert B_t(w_n) \Vert \ d \mathbb{P}$$
which, together with \eqref{E:1t-1} and \eqref{E:1t-2} yields
(\ref{item: sym}). 

To prove (\ref{item: above log}), note that $\ell(u) > 0$ on $V$ and use equation \eqref{E:ell}.  

To show (\ref{item: lambda on A_R}), first note that for $t$ with $|t| = 1$ and $\Re(t) \geq -1/2$, we have by the above equation 
$\lambda(t) \geq \frac{1}{2} \log |t| \geq 0$.
Moreover, since $\Gamma_t^+$ is conjugate into $\textrm{PU}_2$, which
is compact,
$\Vert B_t(w_n) \Vert$ is uniformly 
bounded above for any $t$, hence $\lambda(t) \leq 0$. 
\end{proof}

\subsection{Equidistribution}

Let us recall some fundamental facts from potential theory. 
Let $p(t)$ be a polynomial. A standard computation shows
\begin{equation} \label{E:ddc}
d d^c \log |p(t)| = 2 \pi \sum_{p(z) = 0} \delta_z.
\end{equation}
Moreover, we use the following: 

\begin{lemma} \label{L:distr}
Let $U \subseteq \mathbb{C}$, and let $(\varphi_n)$ be a sequence of weakly subharmonic functions defined on $U$, 
each belonging to $L^1_{loc}(U)$. If $\varphi_n \to \varphi$ in $L^1_{loc}(U)$, then 
$d d^c \varphi_n \to d d^c \varphi$ in the sense of distributions and in the sense of measures (i.e., as elements of the dual of $C^0_c(U)$).
\end{lemma}

We now consider the function 
$$u(w, t):=  \log \big| \det (B_t(w) - I) \big|$$ 
so that,  since $\det (B_t(w) - I) = \Delta_{\what}(t) (t^2 + t + 1)$, we obtain by \eqref{E:ddc}
$$\frac{1}{2 \pi} d d^c u(w, t) =  \#Z_w \nu_w + \delta_{e^{2 \pi i/3}} + \delta_{e^{- 2 \pi i/3}}.$$
In order to apply the previous machinery, we will show the following. 

\begin{proposition} \label{P:conv}
Let $U$ and $\mu$ be as in Theorem~\ref{T:equid}.
Then for almost every sample path $(w_n)$ we have 
$$\frac{1}{n} \log \big| \det (B_t(w_n) - I) \big| \to \chi(t)
\mtext{in $L^1_{loc}(U)$.}$$
\end{proposition}

We start by proving the following almost sure convergence. 

\begin{lemma} \label{L:chi}
Let $V$ and $\mu$ be as in Theorem~\ref{T:equid}.  
Then for any $t \in V$, we have a.s.
$$\limsup_{n \to \infty} \frac{1}{n} \log \big| \det(B_t(w_n) - I) \big| \leq \chi(t).$$
Let $t \in V$ be such that $\lambda(t) \neq \log^+ |t|$. 
Then for almost any $(w_n)$ we have 
$$\lim_{n \to \infty} \frac{1}{n} \log \big| \det(B_t(w_n) - I) \big| = \chi(t).$$
\end{lemma}

\begin{proof}
A simple computation shows
$$| \det(B_t(w_n) - I)| = \big|1 -  \tr B_t(w_n) + (-t)^n \big| \leq 3 \max\big\{ 1, |\tr B_t(w_n)|, |t|^n\big\}$$
so 
$$\limsup_{n \to \infty} \frac{1}{n} \log\big| \det(B_t(w_n) - I)
\big|  \leq \max \big\{ 0, \log |t|, \lambda(t) \big\} = \chi(t),$$
proving the first claim,

Now suppose that $\lambda(t) > \log^+ |t|$.  Then for any
$\epsilon > 0$, for $n$ large enough,
$$| \det(B_t(w_n) - I)| = |1 -  \tr B_t(w_n) + (-t)^n | \geq |\tr B_t(w_n) | - 1 - |(-t)^n| \geq e^{n(\lambda(t) - \epsilon)} - 1 - |t|^n $$
so 
$$\liminf_{n \to \infty} \frac{1}{n} \log| \det(B_t(w_n) - I) |  \geq \lambda(t) - \epsilon$$
for any $\epsilon$, as required. The case of $\lambda(t) < \log^+|t|$ is symmetric.
\end{proof}

\subsection{An example}

Let us first consider the sequence $w_n := \Omega^n$, where $\Omega =
\sigma_1 \sigma_2 \sigma_1$. In this case the trace is identically $0$, so 
$$\det (B_t(w_n) -I ) = 1 + (-t)^{3 n}.$$
Let $m$ denote the Lebesgue measure on the plane.

\begin{lemma}  \label{L:log}
We have the convergence
$\displaystyle  \frac{1}{n} \log | 1  + (-t)^n | \to \log^+ |t| $
in $L^1_{loc}(\mathbb{C})$.
\end{lemma}

\begin{proof}
Since the boundary of the unit circle has (Lebesgue) measure zero, the claim is equivalent to proving the following two statements: 
$$\lim_{n \to \infty}  \frac{1}{n} \int_{\D} \big| \, \log | 1  + (-t)^n |
\, \big|  \ dm(t)  = 0$$
where $m$ is the Lebesgue measure, and for any $R > 1$, 
$$\lim_{n \to \infty} \int_{1 < |t| < R}  \left| \frac{1}{n} \log | 1
+ (-t)^n | - \log |t| \, \right| \ dm(t) = 0.$$
Let us consider the interior of the disc $\D$. Setting $t = - r e^{i \theta}$, we need to show that 
 $$ \frac{1}{n} \int_{\D} \absm{\big}{\log | 1  + (-t)^n | }  \ dm(t)
 = \int_{\D} \frac{1}{n} \absm{\big}{\log | 1  + r^n  e^{i n \theta}
   |} r dr d\theta \to 0.$$
Since for any $z \in \D$, 
$$\log |1 + z |  = \Re\big( \log ( 1 + z) \big) = \Re \left( \sum_{k \geq 1} (-1)^k \frac{z^k}{k} \right)$$
so 
$$\big|\log |1 + z |\,\big| \leq  \sum_{k \geq 1}  \frac{|z|^k}{k} = -
\log \big|1 - |z| \, \big|$$
then 
$$\frac{1}{n} \int_{\D}  \big|\log | 1  + r^n  e^{i n \theta} | \, \big|r dr d\theta \leq  \frac{ 2 \pi}{n} \int_{0}^1 - \log( 1 - r^n) r   dr
 \leq  \frac{ 2 \pi}{n} \int_{0}^1 - \log( 1 - r ) r   dr
$$
which tends to $0$ since $\int_{0}^1 - \log( 1 - r ) r   dr$ converges.

For the other integral, note that 
$$  \left| \frac{1}{n} \log | 1  + (-t)^n | - \log |t| \, \right| = 
\left| \frac{1}{n} \log | (-t)^{-n}  + 1 | \,  \right|
$$
and apply the change of variable $u = 1/t$.
\end{proof}

\begin{proof}[Proof of Proposition \ref{P:conv}]
Let us fix $ R > 1$, and consider a bounded set $K \subseteq \{|t| \leq R \}$.
Denote $U^+ := \setdef{ t \in U }{ \lambda(t) > \log^+ |t|}$ and
$U^- := \setdef{ t \in U }{\lambda(t) < \log^+|t|}$.
$$p_n(t) := \det (B_t(w_n) - I) = 1 + (-t)^n - \tr B_t(w_n).$$
Denote $a_n(t) := 1 + (-t)^n$ and $b_n(t) := - \tr B_t(w_n)$.  Since
$U \setminus ( U^- \cup U^+) \subseteq \cAbar_R$ has measure zero, the
claim is proven if we establish
\begin{align*}
  &\int_{U^+ \cap K } \left| \frac{1}{n} \log |a_n + b_n| - \lambda(t)
    \right| \ dm(t) \to 0 \mtext{and} \\
  &\int_{U^- \cap K} \left| \frac{1}{n} \log |a_n + b_n| - \log^+|t|
    \right| \ dm(t) \to 0.
\end{align*}
Note that for a generic $\omega = (w_n)$ there exists $n_0$ such that for $n \geq n_0$
\begin{equation}
  \label{eq: a_n b_n bound}
  |a_n(t)| \leq 2 R^n, \qquad |b_n(t)| \leq | \tr B_{-R}(w_n) | \leq e^{2 n \lambda(-R)}.
\end{equation}
Let $U_n^+ := \setdefm{\big}{ t \in U }{|b_n| \geq 2 |a_n|\, }$,
$U_n^- := \setdefm{\big}{ t \in U }{ |a_n| \geq 2 |b_n| \, }$, and
$U_n^c := U \setminus \{ U_n^+ \cup U_n^-\}$.
By definition we have 
$$\bigcap_{n} \bigcup_{m \geq n} U^c_m \subseteq U \setminus ( U^+ \cup U^-)$$
so 
$$\limsup_{n \to \infty} m(U^c_n) \leq m(U \setminus ( U^- \cup U^+)) = 0.$$
By Lemmas~\ref{L:log} and~\ref{L:exist}, 
we have 
$$\frac{1}{n} \log |a_n(t)| \to \log^+|t| \mtext{and}  \frac{1}{n} \log |b_n(t)| \to \lambda(t)$$
almost surely and in $L^1_{loc}(\mathbb{C})$.
If $t \in U_n^+$, then $\frac{1}{2} |b_n(t)| \leq |a_n(t) + b_n(t)| \leq \frac{3}{2} |b_n(t)|$
so 
$$  \left| \frac{1}{n} \log |a_n(t) + b_n(t)| - \frac{1}{n} \log  |b_n(t)| \, \right| \leq  \frac{\log(2)}{n} $$
thus
$$\int_{U_n^+ \cap K} \left| \frac{1}{n} \log |a_n + b_n| - \lambda(t) \right| \ dm
\leq \int_{ K} \left| \frac{1}{n} \log | b_n| - \lambda(t) \right| \ dm  + \frac{\log(2)}{n}  m(K)$$
which, by taking the limit, yields
$$\int_{U_n^+ \cap K} \left| \frac{1}{n} \log |p_n| - \lambda(t) \right| \ dm \to 0.$$
Similarly, by swapping the role of $a_n$ and $b_n$, we obtain 
$$\int_{U_n^- \cap K} \left| \frac{1}{n} \log |p_n| - \log^+ |t| \, \right| \ dm \to 0.$$
On the other hand, since $m(U_n^c) \to 0$ as $n \to \infty$, by Lemma \ref{L:log-} we obtain 
$$\int_{U_n^c \cap K} \frac{1}{n} \log^- |p_n(t)| \ dm \to 0.$$
Moreover, if $t \in U_n^c$ then $|a_n(t)| \leq 2 R^n$ and $|b_n(t)|
\leq 2 |a_n(t)| \leq 4 R^n$. Thus
$$|p_n(t)| \leq |a_n(t)| + |b_n(t)| \leq 6 R^n, \qquad \frac{1}{n} \log^+|p_n(t)| \leq \frac{1}{n} \log(6) + \log(R)$$
hence also 
$$\int_{U_n^c \cap K} \frac{1}{n} \log^+ |p_n(t)| \ dm  \to 0,
\mtext{and thus}
\int_{U_n^c \cap K} \left| \frac{1}{n} \log |p_n(t)| \right| \ dm  \to 0,$$
completing the proof. 
\end{proof}

\begin{proof}[Proof of Theorem \ref{T:equid}]
The proof follows from Proposition \ref{P:conv}, by applying $d d^c$ on both sides and 
using Lemma \ref{L:distr} and equation \eqref{E:ddc}.
\end{proof}

We propose the following:
\begin{conjecture} \label{C:measure-zero}
  $m\big(\setdef{ t \in \D }{\lambda(t) = 0}\big) = 0.$
\end{conjecture}
because of:
\begin{proposition}
  \label{prop: last conj}
  If $m\big(\setdef{ t \in \D }{\lambda(t) = 0}\big) = 0$, then
  Theorem~\ref{T:equid} extends to all of $\mathbb{C}$.
\end{proposition}

\begin{proof}
 If $m\big(\setdef{ t \in \D}{\lambda(t) = 0}\big) = 0$, then by considering the transformation $t \mapsto t^{-1}$ and using Lemma~\ref{L:lambda}(\ref{item: sym}) we obtain
$m\big(\setdef{ t \in \C}{\lambda(t) =  \log^+|t| }\big) = 0$.

Now, we proceed exactly as in the proof of Proposition~\ref{P:conv} replacing $U := \C$, 
$U^+ := \setdef{ t \in \C}{\lambda(t) > \log^+|t| }$, and $U^- := \setdef{ t \in \C}{\lambda(t) < \log^+|t| }$, and noting that
\[
  \C \setminus (U^- \cup U^+) \subseteq \setdef{ t \in \C}{\lambda(t) =  \log^+|t| }
\]
has Lebesgue measure zero, thus obtaining for almost every sample path $(w_n)$
\[
  \lim_{n \to \infty} \frac{1}{n}
  \log\abs{\det\big(B_t(w_n) - I\big)}  = \chi(t)
  \quad \textup{in }L^1_{loc}(\C);
\]
note that
$\int_K \frac{1}{n} \log\abs{\det\big(B_t(w_n) - I\big)} \ dm(t)$ is
uniformly bounded in $n$ for any compact $K \subseteq \mathbb{C}$, see
(\ref{eq: a_n b_n bound}) and Lemma~\ref{L:log-}. Having thus extended
Proposition~\ref{P:conv} to all of $\C$, the rest of the argument is
the same as in the proof of Theorem~\ref{T:equid} itself.
\end{proof}

\begin{remark}
With regard to Conjecture \ref{C:measure-zero}, we point out that the zero set of a non-zero subharmonic function may in fact have positive Lebesgue measure. 
To see this, let $K \subseteq \C $ be a compact set with positive Lebesgue measure and empty interior, and let $\mu$ be a measure on $K$ that  maximizes the energy $I(\mu) := - \iint  \log |z - w| d\mu(z) d\mu(w)$. Then the potential function $u(z) := \int \log|z - w| \ d\mu(w) - I(\mu)$ is subharmonic and vanishes almost everywhere on $K$ (see e.g. \cite[Theorem 1.12]{Saff}).

Note moreover that the Lyapunov exponent function $\lambda(t)$ is known to be H\"older continuous in $t$  \cite{LePage}, but there are examples of random walks, even in two generators with polynomial entries in $t$, for which $\lambda(t)$ is not smooth \cite{Simon-Taylor}. 

\end{remark}

\small
\RaggedRight
\bibliographystyle{nmd/math}
\bibliography{pos_roots.bib}

\end{document}